\documentclass{amsart}
\usepackage{amssymb,amscd,color}
\usepackage{tikz,tikz-cd}
\usetikzlibrary{calc, intersections, through}
\usepgflibrary{arrows}
\usetikzlibrary{positioning}
\usepackage{array}
\usepackage[all]{xy}
\usepackage{pdflscape}
\usepackage{afterpage}
\usepackage{capt-of}
\usepackage{enumitem}
\usepackage{lipsum}
\usepackage{graphicx,caption}
\usepackage{bbm}
\usepackage[symbol]{footmisc}
\usepackage{xcolor}

\tikzstyle{wbullet}=[circle, draw=black, fill=white, thick, inner sep=0pt, minimum size=1.5mm]
\tikzstyle{bbullet}=[circle, draw=black, fill=black, inner sep=0pt, minimum size=1.5mm]
\tikzstyle{cross}=[circle, draw=black, fill=white, inner sep=0pt, minimum size=1.5mm]

\usepackage[
urlcolor=blue,
colorlinks=true,
linkcolor=blue,
citecolor=blue,
]{hyperref}

\usepackage{todonotes}
\newtheorem{thm}{Theorem}[section]

\newtheorem{lem}[thm]{Lemma}
\newtheorem{lem-defn}[thm]{Lemma-Definition}

\newtheorem{prop}[thm]{Proposition}
\newtheorem{cor}[thm]{Corollary}

\theoremstyle{definition}
\newtheorem{claim}[thm]{Claim}
\newtheorem{defn}[thm]{Definition}
\newtheorem{nota}[thm]{Notation}
\newtheorem{rmk}[thm]{Remark}
\newtheorem{ex}[thm]{Example}

\theoremstyle{remark}

\numberwithin{equation}{section}

\newcommand{\sC}{\mathcal C}

\newcommand{\C}{\mathbb{C}}
\newcommand{\CC}{\mathbb{C}}
\newcommand{\FF}{\mathbb{F}}
\newcommand{\KK}{\mathbb{K}}
\newcommand{\NN}{\mathbb{N}}
\newcommand{\PP}{\mathbb{P}}
\newcommand{\QQ}{\mathbb{Q}}
\newcommand{\R}{\mathbb{R}}
\newcommand{\RR}{\mathbb{R}}
\newcommand{\ZZ}{\mathbb{Z}}

\newcommand{\alg}{\mathrm{alg}}

\newcommand{\bbk}{\mathbbm{k}}
\newcommand{\rmb}{\mathrm{b}}

\newcommand{\cha}{\mathrm{char~}}
\newcommand{\Diff}{\mathrm{Diff}}
\newcommand{\Exc}{\mathrm{Exc}}

\newcommand{\I}{\mathrm{I}}
\newcommand{\II}{\mathrm{II}}
\newcommand{\III}{\mathrm{III}}
\newcommand{\IV}{\mathrm{IV}}

\newcommand{\im}{\mathrm{Im}}

\newcommand{\ns}{\mathrm{ns}}

\newcommand{\s}{\mathrm{s}}
\newcommand{\Aut}{\mathrm{Aut}}

\newcommand{\mult}{\mathrm{mult}}

\newcommand{\res}{\mathrm{res}}

\newcommand{\supp}{\mathrm{supp}}
\newcommand{\vol}{\mathrm{vol}}

\newcommand{\sE}{\mathcal{E}}
\newcommand{\sF}{\mathcal{F}}
\newcommand{\sG}{\mathcal{G}}
\newcommand{\sO}{\mathcal{O}}
\newcommand{\sS}{\mathcal{S}}

\newcommand{\disp}{\displaystyle}

\begin{document}
\title[Minimal volume of log surfaces]{The minimal volume of log canonical surfaces of general type with positive geometric genus}
\author{Wenfei Liu}
\address{School of  Mathematical Sciences, Xiamen University, Xiamen, Fujian 361005, P.~R.~China}
\email{wliu@xmu.edu.cn}
\thanks{}
\date{\today}
\dedicatory{}
\begin{abstract}
We determine the minimal possible volume of a projective log canonical surface of general type with prescribed positive geometric genus. As applications, we provide effecitive Noether type inequalities for log canonical threefolds and stable surfaces. Also, we obtain a uniform bound on the order of the symplectic automorphism group $\Aut_s(S)$ of smooth projective surfaces $S$ of general type.
\end{abstract}
\maketitle

\tableofcontents

\section{Introduction}
Let $(X, B)$ be a projective log canonical surface defined over an algebraically closed field $\bbk$. The volume $\vol(K_X+B)$ measures the asymptotic growth of the pluri-canonical linear system: 
\[
h^0(X, \lfloor m(K_X+B)\rfloor) = \frac{\vol(K_X+B)}{2} m^2 + \mathrm{o}(m^2).
\]
One says that $K_X+B$ is big if $\vol(K_X+B)>0$; if this is the case, $(X, B)$ is said to be of general type. 

Let us take a subset $\sC\subset (0,1]$, and consider the set $\sS(\sC)$ of projective log canonical surfaces $(X, B)$ of general type such that the coefficients of $B$ lie in $\sC$. (When $\sC=\emptyset$, we understand that $B$ is the zero divisor and $(X, 0)$ is just a projective log canonical surface \emph{without boundary}.) Set
\[
\KK^2(\sC) := \{\vol(K_X+B) \mid (X, B)\in\sS(\sC)\}
\]
It is a fundamental result of Alexeev \cite{Ale94} that, if $\sC$ satisfies the descending chain condition (DCC), then so does $\KK^2(\sC)$.\footnote[2]{See \cite{HK19} for a new treatment.}  In particular, for a DCC set $\sC\subset(0,1]$, any subset of $\KK^2(\sC)$ attains the minimum. 

It is thus interesting to find the minima of the volumes of naturally appearing classes of log canonical surfaces. For a fixed subset $\sC\subset (0,1]$ and a nonnegative integer $p_g$, we can define the following subsets of $\sS(\sC)$ and $\KK^2(\sC)$ respecitvely:
\[
\sS(\sC, p_g) :=\{(X, B)\in\sS(\sC)\mid p_g(X, B)=p_g\}
\]
where $p_g(X, B) :=h^0(X, K_X+\lfloor B \rfloor) $ is the \emph{geometric genus} of $(X, B)$, and
\[
\KK^2(\sC, p_g) :=\{\vol(K_X+B) \mid (X, B)\in \sS(\sC, p_g)\}.
\]
Obviously, we have $\sS(\sC) = \bigcup_{p_g\geq 0}\sS(\sC, p_g)$ and $\KK^2(\sC) = \bigcup_{p_g\geq 0}\KK^2(\sC, p_g)$. 

The aim of this paper is to find the minimum of $\KK^2(\sC, p_g)$ when $p_g>0$ and when the coefficient set $\sC$ satisfies the DCC. Actually, as it turns out, the minimum of $\KK^2(\sC, p_g) $ for $p_g>0$ can be attained under a weaker condition on $\sC$ than the DCC:
\begin{thm}\label{thm: main}
Let $\sC\subset (0,1]$ be a subset such that $\sC\,\cup\,\{1\}$ attains the minimum, say $c$.  Let $p_g$ be a positive integer. Then the following holds.
\begin{enumerate}[leftmargin=*]
\item The set $\KK^2(\sC, p_g) $ attains its minimum, and we have
\[
\min \KK^2(\sC, p_g) = \min \KK^2(\{c\}_{<1}, p_g). 
\]
\item If $p_g\geq 2$, then
\[
\min\KK^2(\sC, p_g)=
\begin{cases}  
(2c-c^2)(p_g-1)-2c^2 & \text{ if } c < \frac{p_g-1}{p_g+1}\\
 p_g-3 + \frac{4}{p_g+1} & \text{ if }  c\geq \frac{p_g-1}{p_g+1} 
\end{cases}  
\]
\item If $p_g = 1$, then
\[
\min\KK^2(\sC, 1)=
\begin{cases}
\frac{1}{42}c^2, & \text{ if } c\leq \frac{7}{13} \\
-\frac{11}{6} c^2 +2c -\frac{7}{13}, & \text{ if }  \frac{7}{13}< c\leq\frac{6}{11} \\
\frac{1}{143}, & \text{ if }  c>\frac{6}{11}.
\end{cases}
\] 
\end{enumerate}
\end{thm}
For a fixed $p_g>0$, $\min\KK^2(\sC, p_g)$ is a piecewise smooth, continuous function of $c = \min(\sC\cup\{1\})$. This lies in the fact that we can take the same smooth projective surface $W$ and reduced divisors $B_W'$ and $D$ so that, setting $B_W^{(c)} = B_W'+cD$, we have $p_g(W, B_W^{(c)} )=p_g$ and $\vol(K_W+B_W^{(c)} ) = \min\KK^2(\sC, p_g)$; see Section~\ref{sec: ex}.

For the following frequently used sets of coefficients
\begin{equation}\label{eq: coeff sets}
\sC_0=\emptyset,\, \sC=\{1\},\,\sC_2=\left\{1-\frac{1}{n} \biggm| n\in\NN_{\geq 2}\right\}\cup\{1\},
\end{equation}
we have $\min (\sC_0\cup\{1\}) = \min (\sC_1\cup\{1\})=1$, and $\min (\sC_2\cup\{1\})=\frac{1}{2}$, and Theorem~\ref{thm: main} (iii) gives
\begin{equation}
\min\KK^2(\sC_0, 1) = \min\KK^2(\sC_1, 1) = \frac{1}{143},\quad \min\KK^2(\sC_2, 1) = \frac{1}{168}
\end{equation}
Therefore, if $(X, B)\in \sS(\sC_1)$ has $\vol(K_X+B)<\frac{1}{143}$ then $p_g(X, B)=0$ and, in fact, $X$ should be a rational surface (Corollary~\ref{cor: <1/143}).

Theorem~\ref{thm: main} can be reformulated without involving the coefficient set $\sC$:
\begin{thm}\label{thm: main2}
Let $p_g$ be a positive integer and $c\in (0,1]$. Let $(X, B)$ be any projective log canonical surface $p_g(X, B) = p_g$ and $\min(\sC_B\cup\{1\})= c$, where $\sC_B$ denotes the set of nonzero coeffcients of $B$.
\begin{enumerate}
\item If $p_g\geq 2$, then
\[
\vol(K_X+B)\geq
\begin{cases}  
(2c-c^2)(p_g-1)-2c^2 & \text{ if } c < \frac{p_g-1}{p_g+1}\\
 p_g-3 + \frac{4}{p_g+1} & \text{ if }  c\geq  \frac{p_g-1}{p_g+1} 
\end{cases}  
\]
\item If $p_g = 1$, then
\[
\vol(K_X+B)\geq
\begin{cases}
\frac{1}{42}c^2, & \text{ if } c\leq \frac{7}{13} \\
-\frac{11}{6} c^2 +2c -\frac{7}{13}, & \text{ if }  \frac{7}{13}<c\leq\frac{6}{11} \\
\frac{1}{143}, & \text{ if }  c>\frac{6}{11}.
\end{cases}
\] 
\end{enumerate}
Moreover, the inequalities are optimal in the sense that the equalities can be realized by some $(X, B) \in \sS(\{c\}_{<1}, p_g)$.
\end{thm}

When $\sC\subset\{1\}$, Theorem~\ref{thm: main2} recovers the inequality obtained in \cite{TZ92}:
\begin{equation}\label{eq: tz}\vol(K_X+B)\geq p_g(X, B) - 3+\frac{4}{p_g(X, B)+1}.\end{equation} 
Note that \eqref{eq: tz} is trivially true if $p_g(X,B)=1$, since its right hand side is negative in this case. The upshot of Theorem~\ref{thm: main2} is to provide the optimal lower bound also in this case, at the same time allowing  the boundary divisors to have any coefficients in $(0,1]$.

As a consequence of Theorem~\ref{thm: main} or \ref{thm: main2}, we can bound the log canonical volume from below by a linear function of the geometric genus.
\begin{thm}\label{thm: linear lower bound}
For any projective log canonical surface $(X, B)$ of general type, we have
\begin{equation}
\vol(K_X+B)\geq (2c-c^2)p_g(X, B)-(2c+c^2) 
\end{equation}
where $c:=\min(\sC_B\cup\{1\})$. Moreover, the inequality is optimal in the following sense:
\begin{enumerate}[leftmargin=*]
\item  if  $c<1$, then the equality can be attained for  $(X, B)$ with  $p_g(X, B)\geq \frac{1+c}{1-c}$. 
\item  if $c=1$ then the inequality is strict, but there is a sequence of projective log canonical surfaces  $X_i$ of general type such that
\[
\lim_{i\to \infty}\vol(K_{X_i}) - p_g(X_i) + 3 =0
\]
\end{enumerate}
\end{thm}

The paper is organized as follows: In Section~\ref{sec: pre} we recall some preliminaries on the birational geometry of surfaces, such as log canonical singularities, the Zariski decomposition, and the volume of $\RR$-divisors. Then we set out to find the minimal possible volume of a projective log canonical surface of general type with prescribed (positive) geometric genus. It can roughly be divided into the following parts:
\begin{enumerate}[leftmargin=*]
\item In Section~\ref{sec: fixed surf} we describe how to find the minimal volume of higher models over a fixed log surface $(Z, B_Z)$; see Lemma~\ref{lem: min fix surf}. This is then explicitly carried out in the case when the log surface contains an extended canonical type curve; see Section~\ref{sec: can curve}. 
\item In Section~\ref{sec: ss}, we introduce the semistable decomposition $B = B^\s + B^\ns$  for a projective log canonical surface $(X, B)$, and study its behavior under blow-ups. We can then divide the set $\sS(\sC, p_g)$ into several subsets $\sS(\sC, p_g; \kappa)$ according to the Kodaira dimension $\kappa$ of $(\tilde X, B^\s_{\tilde X})$, where $(\tilde X, B_{\tilde X})\rightarrow (X, B)$ is the minimal resolution.
\item In Section~\ref{sec: nec cond} we find out several necessary conditions that a log surface $(X, B)\in\sS(\sC, p_g; \kappa)$ achieving the minimal volume should satisfy; see Proposition~\ref{prop: nec cond}. One of them is that the semistable part $B^\s$ and the non-semistable part $B^\ns$ are disjoint; this simplifies the discussion considerably. 
\item In Section~\ref{sec: lower bound} we use the necessary conditions found in Section~\ref{sec: nec cond} to give lower bounds on the volumes of log surfaces in $\sS(\sC, p_g; \kappa)$ when $\kappa\geq 0$. These lower bounds turn out to be optimal by the examples provided in Section~\ref{sec: ex}, which in turn depends on the explicit computation in Section~\ref{sec: fixed surf}. This finishes the search for the minimal volume of log surfaces in $\sS(\sC, p_g; \kappa)$ with $p_g>0$ and hence also of $\sS(\sC, p_g)$ with $p_g>0$.
\end{enumerate}
In the final Section~\ref{sec: app}, the main results are applied to give explicit Noether type inequalities for log canonical threefolds of general type and stable surfaces, and to bound the symplectic automorphisms of surfaces of general type.

\medskip

\noindent{\bf Notation and Conventions.}  
\begin{itemize}[leftmargin=*]
\item We work over an algebraically closed field $\bbk$ of arbitrary characteristic from Section~\ref{sec: pre} to Section~\ref{sec: ex}, unless otherwise stated. For the applications in Section~\ref{sec: app} we restrict to characteristic $0$.
\item  Given a birational projective morphism $f\colon X\rightarrow Y$ between two normal varieties, we always choose the canonical divisors $K_X$ and $K_Y$ in such a way that $K_Y=f_* K_X$, so if the dimension is two or $K_Y$ is $\QQ$-Cartier then $K_X-f^*K_Y$ is a well-defined $\QQ$-divisor supported on the exceptional locus $\Exc(f)$.
\item For a coherent sheaf $\sF$ on a normal projective variety $X$, we denote $h^i(X, \sF) = \dim_\bbk H^i(X, \sF)$ and $\chi(\sF) = \sum_i (-1)^i h^i(X, \sF)$, the Euler characteristic of $\sF$. The numbers $q(X):=h^1(X, \sO_X)$ and $p_g(X):=h^0(X, \sO_X(K_X))$ are called the \emph{irregularity} and the \emph{geometric genus} of $X$ respectively. For a Weil divisor $D$ on $X$, we usually write $H^i(X, D)$ instead of $H^i(X, \sO_X(D))$.
\item On a smooth projective surface, a $(-n)$-curve means a smooth rational curve $C$ with $C^2=-n$.
\item For a subset $\sC\subset \RR$ of real numbers and $a\in \RR$, we define the subset $\sC_{\leq a}:=\{x\in \sC \mid x\leq a\}$. The subsets $\sC_{<a}$, $\sC_{\geq a}$ and $\sC_{> a}$ are similarly defined.
\end{itemize}

\medskip

\noindent{\bf Acknowledgement.} The work has been supported by NSFC (No.~11501012 and No.~11971399) and by the Presidential Research Fund of Xiamen University (No.~20720210006). Thanks go to Valery Alexeev, S\"onke Rollenske, Stephen Coughlan and Thomas Bauer for helpful discussions which clarify and improve my arguments. The initial results of this paper were obtained during my trips to Universit\"at Bayreuth, Philipps-Universit\"at Marburg and the University of Georgia in 2016. The author enjoyed the hospitality and the inspiring academic atmosphere in these institutions. Over the years leading to this version of the paper, I benefited from discussions on related problems with Chen Jiang, Zhi Jiang, Jingjun Han and Meng Chen at Fudan University; I am also grateful to Yong Hu and Tong Zhang for their interest; a stimulating question of Hang Zhao leads to Theorem~\ref{thm: 3 Noether}.

\section{Preliminaries}\label{sec: pre}
We recall the notions of birational geometry in dimension two that are needed in this paper; most of them have a generalization to all dimensions (\cite{KM98, Fuj17}).

\subsection{Divisors}
For $\FF=\ZZ, \QQ$ or $\RR$, an $\FF$-divisor on a normal surface $X$ is a formal finite $\FF$-linear combination $B=\sum_{j\in J} b_j B_j$, where $B_j$ are distinct prime divisors and $J$ is a finite set. We understand that $J=\emptyset$ if $B=0$ is the zero divisor. Denote by $\sC_B=\{b_j \}_{j\in J} \setminus\{0\}$ the set of (nonzero) coefficients of $B$, so that $B=0$ if and only if $\sC_B=\emptyset$. The support of $B$ is defined to be $\disp\supp(B):=\bigcup_{b_j\neq 0}\supp(B_j)$.

An $\FF$-divisor is said to be $\FF$-Cartier if it is a finite $\FF$-linear combination of Cartier divisors. Two $\FF$-divisors $B$ and $B'$ are said to be \emph{$\FF$-linearly equivalent}, denoted $B\sim_\FF B'$, if $B-B'$ is a finite $\FF$-linear combination of principle divisors.  When $\FF=\ZZ$, then a $\ZZ$-divisor (resp.~$\ZZ$-Cartier) is just a usual Weil (resp.~Cartier) divisor and $\ZZ$-linear equivalence is the usual linear equivalence, which is then denoted by $B\sim B'$. A non-zero effective $\ZZ$-divisor on a normal surface is usually called a \emph{curve}. Two $\FF$-divisors $B$ and $B'$ are said to be \emph{numerically equivalent} if $B\cdot C = D\cdot C$ for any curve $C$, and it is denoted by $B\equiv B'$.

For a real number $a$, we write
\[
B^{=a} = \sum_{b_j=a} b_j B_j, \hspace{.5cm} \lfloor B \rfloor =\sum \lfloor b_j\rfloor B_j,
\]
where $\lfloor b_j \rfloor$ denotes the integer satisfying the condition $b_j-1 <\lfloor b_j \rfloor \leq b_j$; the $\RR$-divisors  $B^{>a},\, B^{<a},$ and $\lceil B \rceil$ are similarly defined. For a prime divisor $P$ on $X$, we denote by $\mult_PB$ the coefficient of $P$ in $B$; for a point $p\in X$, the \emph{multiplicity} of $B$ at $p$ is $\mult_p B:=\sum_{j\in J} b_j\,\mult_p B_j$. 

\subsection{(Sub-)log surfaces and their singularities}
A \emph{sub-log surface} (resp.~\emph{log surface}) $(X, B)$ consists of a \emph{connected} normal surface $X$ and an $\RR$-divisor $B$ such that the coefficient set $\sC_B$ is contained in $\RR_{\leq 1}$ (resp.~$(0,1]$) and $K_X+B$ is $\RR$-Cartier.  We call a sub-log surface $(X, B)$ \emph{smooth} if $X$ is so. A \emph{higher model} of a sub-log surface $(X, B)$ consists of a triple $(Y, B_Y, \rho)$, where $(Y, B_Y)$ is a sub-log surface and  $\rho\colon Y\rightarrow X$  is a birational projective morphism such that $\rho_* B_{Y} = B$. The higher model is called
\begin{itemize}
\item \emph{effective} if $B_Y$ is so;
\item \emph{crepant} if $K_Y+B_Y=\rho^*(K_X+B)$;
\item \emph{a resolution} if $Y$ is smooth;
\item \emph{the minimal resolution} if it is a crepant resolution and the exceptional locus $\Exc(\rho)$ does not contain any $(-1)$-curves;
\item \emph{a log resolution} if it is a resolution and $\Exc(\rho)\cup \rho^{-1}_*B$ has a simple normal crossing support, where $\rho^{-1}_*B$ is the strict transform of $B$.
\end{itemize} 
For simplicity, we often omit the birational morphism $\rho$ and say that $(Y, B_Y)$ is a higher model of $(X,B)$.  

Let $(Y, B_Y, \rho)$ be a crepant log resolution of a sub-log surface $(X, B)$. For a prime divisor $E$ on $Y$, $a_E(X, B):=1-\mult_E B_Y$ is called the \emph{log discrepancy} of $E$ with respect to $(X, B)$. The sub-log surface $(X,B)$ is said to be \emph{sub-klt} (resp.~\emph{sub-log canonical})  if the coefficients of $B_Y$ are less than $1$  (resp.~at most $1$); when $B$ is effective, we drop the prefix "sub" everywhere. We refer to \cite{Ale92, Kol13} for a classification of log canonical surface singularities. A \emph{nonklt center} $Z\subset X$ of $(X, B)$ is the image of a stratum of the simple normal crossing divisor $B_Y^{=1}$; it is called \emph{accessible} if $\supp B_Y^{>0}$ is singular at some point of $\rho^{-1}(Z)$; see \cite[Definition~2.3]{AL19b}. It is easy to see that the accessibility of $Z$ does not depend on the choice of the crepant log resolution.

\subsection{The Iitaka--Kodaira dimension and the volume of divisors}

\begin{defn}
Let $X$ be a normal projective surface and $D$ an $\RR$-Cartier $\R$-divisor on $X$. The \emph{Iitaka--Kodaira dimension} $\kappa(D)$ of $D$ is defined as follows: $\kappa(D)$ is $-\infty$ if $h^0(X, \lfloor mD\rfloor) =0$ for every positive integer $m>0$; otherwise $\kappa(D)$ is defined to be the nonnegative integer $k$ satisfying
\[
0 < \limsup_{m\rightarrow\infty} \frac{h^0(X, \lfloor mD\rfloor )}{m^k} <\infty.
\]
It is clear that $\kappa(D)\in\{-\infty,0,1,2\}$. The $\RR$-divisor $D$ is said to be \emph{big} if $\kappa(D)=2$.  The \emph{volume} of $D$ is defined to be
\[
\vol(D)=\limsup_{m\rightarrow\infty} \frac{h^0(X, \lfloor mD\rfloor )}{m^2/2}
\]
Obviously, $\vol(D)>0$ if and only if $\kappa(D)=2$. We say a projective log canonical surface $(X, B)$ is \emph{of general type} if $K_X+B$ is big. The \emph{(log canonical) volume} of a log surface $(X, B)$ is defined to be $\vol(K_X+B)$.
\end{defn}

We refer to \cite{Laz04a} and \cite[Section 2.2]{HMX18} for the basic properties of the volume function. 
%\begin{lem}\label{lem: vol ZD}
%Let $D$ be a pseudo-effective $\R$-divisor on a normal projective surface $X$.
%\begin{enumerate}[leftmargin=*]
%\item If $D$ is an effective $\RR$-Cartier $\RR$-divisor having the same support as a big $\RR$-divisor, then $D$ is also big. 
%\item Suppose that $f\colon X\rightarrow Y$ be a birational morphism between normal projective surfaces. Then $\vol(D)\leq \vol(f_*D)$.
%\end{enumerate}
%\end{lem}

\subsection{The Zariski decomposition}
 The Zariski decomposition, defined as follows, is an essential tool to compute volumes.
\begin{defn}\label{defn: zariski}
Let $D$ be an $\RR$-Cartier $\R$-divisor on a normal projective surface. Then the \emph{Zariski decomposition} of $D$ is a decomposition $D= P+N$ such that
\begin{itemize}
\item[(i)] $P$ is a nef $\RR$-divisor, that is, $P$ is $\RR$-Cartier and $P\cdot C\geq 0$ for any curve $C$ on $X$;
\item[(ii)] $N$ is either zero or a nonzero effective $\RR$-divisor whose intersection matrix is negative definite;
\item[(iii)] $P\cdot N_i=0$ for each irreducible component $N_i$ of $N$.
\end{itemize} 
We call $P$ the \emph{positive part} of $D$, and $N$ the \emph{negative part}. 
\end{defn}

\begin{ex}
Let $(X, B)$ be a projective log canonical surface of general type. Then there is a birational morphism $\pi\colon (X, B)\rightarrow (\bar X, \bar B)$, such that $\bar B=\pi_* B$, $(\bar X, \bar B)$ is a projective log canonical surface with ample $K_X+B$ and $K_X+B-\pi^*(K_{\bar X}+\bar B)$ is an effective divisor supported on the exceptional locus $\Exc(\pi)$; see \cite{Fuj21}. We call $(\bar X, \bar B)$, together with the birational morphism $\pi$, the \emph{ample model} of $(X, B)$. Let $P:=\pi^*(K_{\bar X}+\bar B)$ and $N:=K_X+B-\pi^*(K_{\bar X}+\bar B)$. Then $K_X+B = P+N$ is the Zariski decomposition of $K_X+B$. 
\end{ex}

Recall that an $\RR$-Cartier $\RR$-divisor $D$ on a normal projective surface $X$ is \emph{pseudo-effective} if $D\cdot H\geq 0$ for any ample divisor $H$.

\begin{thm}
Let $D$ be an $\RR$-Cartier $\R$-divisor on a normal projective surface $X$. Then the following holds.
\begin{enumerate}
\item The $\RR$-divisor $D$ admits a Zariski decomposition if and only if $D$ is pseudo-effective.
\item The negative part $N$ in the Zariski decomposition depends only on the numerical class of $D$.
\item The positive part $P$ in the Zariski decomposition can be characterized as the maximal nef $\RR$-divisor such that $D-P$ is effective.
\end{enumerate} 
\end{thm}
\begin{proof}
When $X$ is smooth, the theorem is proved for effective $\QQ$-divisors $D$ in \cite{Zar61} and then generalized to pseudo-effective $\QQ$-divisors by \cite{Fujita79}; see also \cite[Theorem~14.14]{badescu01}. The proof for the general case is similar; see \cite[Theorem~2.2]{Pro03}.
\end{proof}

We omit the proof of the following easy lemma.
\begin{lem}\label{lem: vol ZD}
Let $D$ be a pseudo-effective $\R$-divisor on a normal projective surface $X$, and $D=P+N$ the Zariski decomposition with positive part $P$.
\begin{enumerate}[leftmargin=*]
\item $\vol(D) = \vol(P) = P^2 \geq D^2$, and the last inequality is an equality if and only if $D$ is nef.
\item If $D$ is big and if $E$ is an effective $\RR$-Cartier $\RR$-divisor such that $N-E$ is not effective, then $\vol(D)> \vol(D-E)$.
\item If $E$ is a pseudo-effective  $\RR$-Cartier $\RR$-divisor such that $D+E$ is big, then $D+tE$ is big for any $t>0$.
\end{enumerate}
\end{lem}

\subsection{The volumes of certain $\RR$-divisors depending on a parameter}
We compute the volume of an $\RR$-divisor of the form $D+tC$ on a normal projective surface, where $D$ is nef and $C$ is an irreducible curve.
\begin{lem}\label{lem: nef+curve}
Let $X$ be a normal projective surface. Suppose that $D$ is a nef $\RR$-divisor and $C$ an irreducible curve on $X$.  Define 
\[
t_0=\sup\{t\in \RR_{\geq 0}\mid D+t C \text{ is nef }\}\in \RR_{\geq 0}\cup\{\infty\}.
\] 
Then the following holds.
\begin{enumerate}
\item
$t_0 = \begin{cases} -\frac{C\cdot D}{C^2} & \text{ if $C^2<0$} \\ \infty &  \text{ if $C^2\geq 0$} \end{cases}.$
\item
For $t\in \RR_{\geq 0}$ the positive part of $D+tC$ in the Zariski decomposition is $P:=D+\min\{t, t_0\}C$, and consequently its volume is
\begin{equation}\label{eq: nef+curve}
\vol(D+tC) = P^2 = D^2+2\min\{t, t_0\}C\cdot D +(\min\{t, t_0\})^2C^2 
\end{equation}
\item If $C\cdot D>0$, then $D+tC$ is big for any $t>0$.
\end{enumerate}
\end{lem}
\begin{proof}
(i) and (ii). If $C^2\geq 0$ then $D+tC$ is nef for any $t\geq 0$ and hence $t_0=\infty$. In this case the positive part of $D+tC$ is $D+tC$ itself. Now suppose  that $C^2<0$. Then one verifies $(D+tC)\cdot C<0$ if and only if $t>-\frac{C\cdot D}{C^2}$. It follows readily that $t_0=-\frac{C\cdot D}{C^2}$ and $D=P+(t-\min\{t_0, t\})C$ is the Zariski decomposition of $D+tC$ with $P$ as the positive part. 

(iii)  If $C\cdot D>0$, then $t_0>0$, and for $0<\epsilon \ll t$, we have by \eqref{eq: nef+curve}
\[
\vol(D+t C) \geq  \vol(D+\epsilon C)  = D^2+ \epsilon C\cdot D +\epsilon^2C^2 > 0
\]
\end{proof}

\begin{lem}\label{lem: fib+horiz}
Let $X$ be a smooth projective surface, and $B$ an $\RR$-divisor such that $K_X+ B$ is nef. Let $C$ be an irreducible curve such that $B\cdot C=0$ and $d:=K_X\cdot C>0$. Then the following holds.
\begin{enumerate}[leftmargin=*]
\item If $p_a(C)>0$, then, for $0\leq t\leq 1$, we have 
\[
\vol(K_X+B+tC) = (K_X+ B)^2+ t^2(2p_a(C) -2) + m(2t-t^2).
\]
\item If $p_a(C)=0$, then, for $0\leq t\leq 1$, we have 
\begin{equation}\label{eq: vol pa=0}
\vol(K_X+B+tC) = 
\begin{cases}
(K_X+ B)^2+ d(2t-t^2) - 2 t^2 &  \text{if $t< \frac{d}{d+2}$}\\
(K_X+ B)^2+ d-2+ \frac{4}{d+2} & \text{if } \frac{d}{d+2}\leq t\leq 1 
\end{cases}
\end{equation}
\end{enumerate} 
\end{lem}
\begin{proof}
For $0\leq t\leq 1$, we have
\begin{equation}\label{eq: tC}
(K_X+B+tC)\cdot C =t(K_X+C)\cdot C + (1-t) K_X\cdot C =t(2p_a(C) -2) + d(1-t).
\end{equation}
(i) Assume that $p_a(C)>0$. Then, for any $0\leq t\leq 1$, we have $(K_X+B+tC)\cdot C\geq d(1-t)\geq 0$ by \eqref{eq: tC}. Since $K_X+B$ is nef, so is $K_X+B+tC$, and hence
\begin{align*}
\vol(K_X+B+tC) & = (K_X+B+tC)^2 \\
& = (K_X+B)^2+2t(K_X+B)\cdot C + t^2 C^2\\
& = (K_X+B)^2+ 2t K_X\cdot C + t^2 C^2\\
&=  (K_X+B)^2+t^2(K_X\cdot C +C^2 ) + (2t-t^2) K_X\cdot C \\
& = (K_X+B)^2+t^2(2p_a(C) -2) + (2t-t^2) K_X\cdot C \\
&=  (K_X+B)^2+t^2(2p_a(C) -2) + d(2t-t^2)  
\end{align*}

(ii) Now assume that $p_a(C)=0$, so $C\cong \PP^1$. By the adjunction formula, we have $C^2=K_X\cdot C -2 = -d-2<0$. By Lemma~\ref{lem: nef+curve}, we have 
\begin{align*}
\vol(K_X+B+tC) & = (K_X+B)^2+2\min\{t, t_0\}C\cdot (K_X+B) +(\min\{t, t_0\})^2C^2 \\
 & = (K_X+ B)^2+ 2d \min\{t, t_0\}  - (d+2)(\min\{t, t_0\})^2.
\end{align*}
where $t_0 = -\frac{(K_X+B)\cdot C}{C^2} = \frac{d}{d+2}$. Now \eqref{eq: vol pa=0} is obtained by spelling out the two cases $t<t_0$ and $t\geq t_0$  of the above formula.
\end{proof}

\section{The minimal volume over a fixed log surface}\label{sec: fixed surf}
In this section, we figure out the minimal possible volume of effective higher models $(U, B_U)$ of a fixed smooth projective log surface $(Z, B_Z)$.
\subsection{A method of finding the minimal volume over a fixed log surface}
\begin{lem}\label{lem: bdd below}
Let $\rho\colon U\rightarrow Z$ be a birational morphism between smooth projective surfaces. Let $B_Z'$ and $B_Z''$ be two effective divisors on $Z$ such that $\rho$ is an isomorphism over a neighborhood of $\supp(B_Z')$. Denote $\displaystyle m=\max_{p}\{ \mult_p B_Z''\}$ with $p\in Z$ running through the points blown up by $\rho$. Then
\[
K_{U}+\rho_*^{-1}(B_Z'+B_Z'')\geq \rho^*\left(K_Z + B_Z'+ \frac{1}{m}B_Z''\right).
\]
\end{lem}
\begin{proof}
We write $\rho\colon U\rightarrow Z$ as the composition $\rho_1\circ\rho_{2}\circ\cdots\circ\rho_n$ of blow-ups, and let $\sE_i\subset U$ be the total transform of the exceptional curve of $\rho_i$ ($1\leq i\leq n$).
Then the following equalities (respectively inequality) hold:
\begin{itemize}[leftmargin=*]
\item $\displaystyle K_U =\rho^* K_Z+ \sum_{1\leq i\leq n} \sE_i$.
\medskip
\item  $\rho^{-1}_*B_Z' = \rho^*B_Z'$, since $\rho$ is an isomorphism over a neighborhood of $\supp(B_Z')$.
\medskip
\item $\displaystyle\rho_*^{-1}B_Z''\geq  \rho^* B_Z''  -m \sum_{1\leq i\leq n} \sE_i$ by the definition of $m$.
\end{itemize}  
Combining these facts, one obtains
\begin{multline*}
K_U+\rho_*^{-1}(B_Z'+B_Z'')= K_U+\rho_*^{-1}B_Z''+\rho^*B_Z' \\
\geq \rho^*K_Z + \frac{1}{m} \rho^* B_Z'' + \rho^*B_Z' = \rho^*\left(K_Z+B_Z'+\frac{1}{m}B_Z\right).
\end{multline*}
\end{proof}

\begin{nota}\label{nota: SZBZ}
Let $(Z, B_Z)$ be a smooth projective log surface satisfying the following conditions:
\begin{itemize}
\item $K_Z+B_Z$ is big,
\item there is a decomposition $B_Z=B_Z'+B_Z''$ such that $B_Z'$ and $B_Z''$ are effective $\RR$-divisors without common components and $\kappa(K_Z+B_Z')\geq 0$.
\end{itemize}
Define $\sS(Z, B_Z; B_Z')$ to be the set of triples $(U, B_{U}, \rho_U)$, where $(U, B_{U})$ is a smooth effective higher model of $(Z, B_Z)$, and $\rho_U \colon U\rightarrow Z$ is birational morphism that is an isomorphism over a neighborhood of $\supp(B_Z')$. Finally, define
\[
\KK^2(Z, B_Z; B_Z'):=\{\vol(K_U+B_U) \mid (U, B_{U}, \rho_U) \in \sS(Z, B_Z; B_Z') \}
\]
\end{nota}

\begin{lem}\label{lem: min fix surf}
Let $(Z, B_Z)$ be as in Notation~\ref{nota: SZBZ} and $(U, B_U, \rho_U)\in \sS(Z, B_Z; B_Z')$. Then the following holds. 
\begin{enumerate}
\item $K_{U}+B_{U}$ is big.
\item If $B_{U}$ is the strict transform of $B_Z$ and $\mult_q  B_U'' \leq 1$ for any $q\in U$, where $B_U''=\rho_{U*}^{-1} B_Z''$ is the strict transform, then $\vol(K_U +B_{U}) = \min \KK^2(Z, B_Z; B_Z')$. In particular, if $(U, B_{U}, \rho_U)\in \sS(Z, B_Z; B_Z')$ resolves the singularities of $B_Z''$ and $B_U$ is the strict transform of $B_Z$, then $\vol(K_U +B_{U}) = \min \KK^2(Z, B_Z; B_Z')$.
\end{enumerate}
\end{lem}
\begin{proof}
(i) Since $\kappa(K_Z+B_Z')\geq 0$, there is an effective $\RR$-divisor $D$ on $Z$ such that $K_Z+B_Z'\sim_{\RR} D$, so $D+B_Z''\sim_{\RR} K_Z+B_Z$ is a big effective divisor. By Lemma~\ref{lem: bdd below}
\[
K_{U}+B_U  \geq K_{U}+\rho_{U*}^{-1}(B_Z'+B_Z'')  \geq \rho_U^*\left(K_Z + B_Z'+ \frac{1}{m}B_Z''\right)
\]
where $\displaystyle m=\max_{p}\{ \mult_p B_Z''\}$. Note that the last $\RR$-divisor is $\RR$-linearly equivalent to $\rho_U^*(D+\frac{1}{m}B_Z'')$, which is big by Lemma~\ref{lem: vol ZD} (iii). It follows that $K_{U}+B_U$ is big.

(ii) Let $(V, B_V, \rho_V)$ be any log surface in $\sS(Z, B_Z; B_Z')$. Then the birational map $\phi = \rho_V^{-1}\circ\rho_U\colon U \dashrightarrow V$ over $Z$ is an isomorphism over an open neighborhood of $\supp(B_Z')$. Let $\gamma_U\colon W\rightarrow U$ and $\gamma_V\colon W\rightarrow V$ be a common resolution of $U$ and $V$ such that $\gamma_U$ and $\gamma_V$ are isomorphisms over a neighborhood of $\supp(B_Z')$:
\[
\begin{tikzcd}
&  W \arrow[ld,"\gamma_U"']\arrow{rd}{\gamma_V}& \\
U  \arrow[rr, dashed, "\phi"] \arrow[rd, "\rho_U"']&& V\arrow[ld, "\rho_V"] \\
& Z &
\end{tikzcd} 
\]
Let  $B_{W}$ be the strict transform of $B_Z$ and $\rho_W=\rho_U\circ \gamma_U = \rho_V\circ \gamma_V$. Then $(W,B_{W}, \rho_W)\in\sS(Z, B_Z; B_Z')$ and $\gamma_{U*}B_{W}=B_{U}$. Since $\gamma_U$ is an isomorphism over a neighborhood of $B_U'$ and $\mult_{q}B_{U}''\leq 1$ for any point $q\in U$, we have by Lemma~\ref{lem: bdd below}
\[
K_{ W}+B_{W}\geq \gamma_U^*(K_{U}+B_{U}).
\] 
On the other hand, 
\[
\gamma_{U*}(K_{ W}+B_{W})= K_{U}+B_{U} \text{ and } \gamma_{V*}(K_{ W}+B_{W})\leq K_{V}+B_V.
\]
It follows that
\[
\vol(K_V+B_V)\geq \vol(K_{ W}+B_{W})=\vol(K_{U}+B_{U}).
\]
Since $(V, B_V, \rho_V)\in \sS(Z, B_Z; B_Z')$ is arbitrary, we infer that $\vol(K_U +B_{U}) = \min \KK^2(Z, B_Z; B_Z')$.  
\end{proof}

\subsection{Minimal volumes over log surfaces with an extended canonical type curve}\label{sec: can curve}
The task of this subsection is to compute $\min \KK^2\left(Z, B_Z^{(c)}; B_Z'\right)$, where $Z$, $B_Z^{(c)}$, and $B_Z'$ satisfy the following conditions:
\begin{itemize}[leftmargin=*]
\item $(Z, B_Z')$ is smooth projective log canonical surface with $B_Z'$ reduced and $K_Z + B_Z'\sim_\QQ 0$.
\item $B_Z^{(c)}= B_Z' + c(C+D)$, where $c\in (0,1]$, $C$ is a connected reduced curve supporting a curve of canonical type, and $D$ is a $(-2)$-curve intersecting $C$ transversally at exactly one point, and $\supp(B_Z') \cap \supp(C+D) = \emptyset$.
\item There are no $(-1)$-curves $G$ such that $(K_Z+B_Z')\cdot G=0$.
\end{itemize}
We recall that a curve $C=\sum n_i C_i$ on a smooth projective surface $Z$ is \emph{of canonical type} if $K_Z\cdot C_i= C\cdot C_i = 0$ for all $i$ (\cite{Mum69}). Indecomposable curves of canonical type are classified into types $\I_b\, (b\geq 0)$, $\II$, $\III$, $\IV$, $\I^*_b\, (b\geq 0)$, $\II^*$, $\III^*$, $\IV^*$ (\cite{BHPV}). 

Lemma~\ref{lem: min fix surf} gives the recipe of finding $\min \KK^2\left(Z, B_Z^{(c)}; B_Z'\right)$:
\begin{enumerate}
\item[(1)] Take $(W, B_W^{(c)}, \rho_W)\in \sS (Z, B_Z^{(c)}; B_Z')$ such that $\rho_W\colon W\rightarrow Z$ resolves the singularities of $C+D$ and $B_W^{(c)} = B_W' + c(C^w + D^w)$, where $B_W', C^w$ and  $D^w$ are the strict transforms of $B_Z', C$ and $D$ respectively. Then $\min \KK^2\left(Z, B_Z^{(c)}; B_Z'\right) =\vol(K_W+B_W^{(c)})$. 
\item[(2)] To compute $\vol(K_W+B_W^{(c)})$, we need to consider the ample model, say $\gamma_c\colon (W, B_W^{(c)})\rightarrow (X^{(c)}, B_{X^{(c)}})$. Then we have $\vol(K_W+B_W^{(c)}) = (K_{X^{(c)}}+B_{X^{(c)}})^2$. The latter is in turn computed by pulling back via the minimal resolution $\pi_c\colon (\tilde X^{(c)}, B_{\tilde X^{(c)}}) \rightarrow (X^{(c)}, B_{X^{(c)}})$. Obviously, we have an induced birational morphism $\gamma_c\colon W\rightarrow \tilde X^{(c)}$ such that $\gamma_c= \pi_c\circ\tilde \gamma_c $. It can be readily checked that $\tilde \gamma_c$ is an isomorphism on a neighborhood of $B_W'$ and that $\Exc(\gamma_c) \subset \Exc(\rho_W)$. Hence there is a birational morphism $\rho_c\colon \tilde X^{(c)}\rightarrow Z$ such that the following diagram commutes:
\begin{equation}\label{eq: recipe}
\begin{tikzcd}
W\arrow[r, "\tilde \gamma_c"] \arrow[rd, "\rho_W"']  \arrow[rr, bend left, "\gamma_c"]& \tilde X^{(c)} \arrow[r, "\pi_c"] \arrow[d, "\rho_c"] & X^{(c)}\\
 & Z&
\end{tikzcd}
\end{equation}
\end{enumerate} 
We explain the relation among the models $\tilde X^{(c)}$ as $c$ varies. For $0<c<c'\leq 1$, we have $\Exc(\tilde \gamma_{c'}) \subset \Exc(\tilde \gamma_c)$, and hence there is a birational morphism $\tilde \gamma_{c'c}\colon \tilde X^{(c')}\rightarrow \tilde X^{(c)}$ such that $\tilde \gamma_{c'c}\circ \tilde \gamma_{c'} = \tilde \gamma_c$. Since every $\tilde X^{(c)}$ sits between $W$ and $Z$, there are finitely many values $0=c_0 < c_1 < \cdots <c_r=1$ such that  for any $c_i < c \leq c_{i+1} (0\leq i\leq r-1)$, $\tilde \gamma_c\colon W\rightarrow \tilde X^{(c)}$ stay the same and $\tilde X^{(c)} \rightarrow \tilde X^{(c_i)}$ is a nontrivial birational contraction; we call $c_i\,\, (1\leq i\leq r-1)$ the \emph{critical values} of $c$. Note that $\tilde X^{(c)}= Z$ for $0<c\ll 1$, and $\tilde X^{(1)}$ dominate all the other $\tilde X^{(c)}$ for $0<c\leq 1$; we will omit the subscript when $c=1$ and denote 
\[
\tilde\gamma:=\tilde \gamma_1,\, \rho := \rho_1,\, \gamma := \gamma_1,\, X:=X^{(1)},  \text{ and } \tilde X:=\tilde X^{(1)}.
\]

Before diving into the details of the computation, we introduce some additional notations and convention:
\begin{itemize}[leftmargin=*]
\item We call the curves appearing in $C+D$ and its inverse images on $\tilde X$ and $W$  the \emph{visible curves}. If the curve $C+D$ is a normal crossing divisor, then we will present the visible curves on $Z$, $\tilde X$ and $W$ by dual graphs; otherwise, the curves are sketched more concretely to indicate the (unique) worse-than-nodal singularity. 
\item The strict transforms of $C_i$ and $D$ on $W$ (resp.~on $\tilde X$) are denoted by  $C_i^w$ and $D^w$ (resp.~by $\tilde C_i$ and $\tilde D$). The $\rho_W$-exceptional curves on $W$ are denoted by $E_i$. If there is only one such curve, then the subscript is omitted. Correspondingly, their image curves on $\tilde X$ are denoted by $\tilde E_i$. These curves are denoted by black bullets in the dual graphs and dashed lines in the sketches of visible curves respectively.
\end{itemize}

Now we proceed according to the canonical type curve that $C$ supports. 

\medskip

\noindent{\bf Case: $C$ supports a curve of canonical type $I_0$.}  In this case $C$ is a smooth elliptic curve with $C^2=0$ and $D$ is a $(-2)$-curve intersecting $C$ transversally at a smooth point. The dual graphs of visible curves on $W\xrightarrow{\rho_W}Z$ are as follows: 
\begin{center}
\begin{tikzpicture}[font=\tiny]
\begin{scope}[xshift = -2.5cm]
\draw (0,0)--(2,0);
\foreach \x in {0, 2}
\draw [fill=white] (\x,0) circle[radius=2pt];  
\draw[fill=black](1,0) circle[radius=2pt];
\node[above] at (0, 0){\tiny $C^w$};
\node[above] at (1, 0){\tiny $E$};
\node[above] at (2, 0){\tiny $D^w$};
\draw[->] (2.5,0) --node[above]{\tiny $\rho_W$}(3.5,0);
\end{scope}
\begin{scope}[xshift = 1.5cm]
\draw (0,0)--(1,0);
\foreach \x in {0, 1}
\draw [fill=white] (\x,0) circle[radius=2pt];  
\node[above] at (0, 0){$C$};
\node[above] at (1, 0){$D$};
\end{scope}
\end{tikzpicture}
\end{center}
There is exactly one critical value $c=\frac{2}{3}$: If $\frac{2}{3}<c\leq 1$, then $\tilde X^{(c)} = W$. In this case, $K_{\tilde X^{(c)}} + B_{\tilde X^{(c)}} \sim_{\RR} c C^w + \frac{1}{3}D^w + E$ and hence
\[
\vol(K_{\tilde X^{(c)}} + B_{\tilde X^{(c)}}) =  \left(E + c C^w + \frac{1}{3}D^w\right)^2 =  -c^2 +2c-\frac{2}{3}.
\]
 If $0<c\leq \frac{2}{3}$, then $\tilde X^{(c)} = Z$, $K_{\tilde X^{(c)}} + B_{\tilde X^{(c)}} \sim_{\RR} c C + \frac{c}{2} D$ and hence
\[
\vol(K_{\tilde X^{(c)}} + B_{\tilde X^{(c)}}) =\left (c C + \frac{c}{2} D\right)^2=\frac{c^2}{2}.
\]
In conclusion, we obtain in this case
\begin{equation}\label{eq: I_0}
\min \KK^2\left(Z, B_Z^{(c)}; B_Z'\right) = \vol(K_{\tilde X^{(c)}} + B_{\tilde X^{(c)}}) = 
\begin{cases}
\frac{1}{2}c^2 & \text{if } c\leq \frac{2}{3} \\
 -c^2 +2c-\frac{2}{3}& \text{if } c >  \frac{2}{3}  \\
\end{cases}
\end{equation}
\medskip

\noindent{\bf Case: $C$ supports a curve of canonical type $\mathrm{I}_b$ with $b\geq 1$}. In this case, $C=\sum_{i=1}^b C_i$ is a cycle of $(-2)$-curves if $b\geq 2$, and $D$ is $(-2)$-curve intersecting one component, say $C_1$, of $C$.  The dual graph of visible curves on $W\xrightarrow{\tilde \gamma} \tilde X \xrightarrow{\rho} Z$ are as follows:
\begin{center}
\begin{tikzpicture}[font=\tiny]
\begin{scope}[xshift=0cm]
\draw (180:1) -- (240:1)--(300:1);
\draw (60:1)--(120:1)--(180:1) -- (180:2);
\draw[dashed] (300:1) .. controls (0:1.2) .. (60:1);
\foreach \x in {180, 240, 300, 60, 120}
\draw[fill=white] (\x: 1) circle[radius=2pt]; 
\draw[fill=white] (180: 2) circle[radius=2pt]; 
\foreach \x in {210, 270, 90, 150}
\draw[fill=black] (\x: .866) circle[radius=2pt];
\draw[fill=black] (180: 1.5) circle[radius=2pt];
\node[below] at (-1.5, 0) {$E_0$};
\node at (210:0.6) {$E_1$};
\node at (270: 0.6) {$E_2$};
\node[below=.1] at (90:1) {$E_{b-1}$};
\node at (150:0.6) {$E_b$};
\node [above] at (-2, 0) {$D^w$};
\node[above] at (-1.05, 0) {$C_1^w$};
\node[below] at (240: 1) {$C_2^w$};
\node[below] at (300: 1) {$C_3^w$};
\node[above] at (60: 1) {$C_{b-1}^w$};
\node[above] at (120: 1) {$C_b^w$};
\draw[->] (1.5,0) --node[above]{$\tilde \gamma$}(2,0);
\end{scope}

\begin{scope}[xshift=4.5cm]
\draw (180:1) -- (240:1)--(300:1);
\draw (60:1)--(120:1)--(180:1) -- (180:2);
\draw[dashed] (300:1) .. controls (0:1.2) .. (60:1);
\foreach \x in {180, 240, 300, 60, 120}
\draw[fill=white] (\x: 1) circle[radius=2pt]; 
\draw[fill=white] (180: 2) circle[radius=2pt]; 
\foreach \x in {210, 150}
\draw[fill=black] (\x: .866) circle[radius=2pt];
\node at (210:0.6) {$\tilde E_1$};
\node at (150:0.6) {$\tilde E_b$};
\node [above] at (-2, 0) {$\tilde D$};
\node[above] at (-1.05, 0) {$\tilde C_1$};
\node[below] at (240: 1) {$\tilde C_2$};
\node[below] at (300: 1) {$\tilde C_3$};
\node[above] at (60: 1) {$\tilde C_{b-1}$};
\node[above] at (120: 1) {$\tilde C_b$};
\draw[->] (1.5,0) --node[above]{$\rho$}(2,0);
\end{scope}

\begin{scope}[xshift=9cm]
\draw (180:1) -- (240:1)--(300:1);
\draw (60:1)--(120:1)--(180:1) -- (180:2);
\draw[dashed] (300:1) .. controls (0:1.2) .. (60:1);
\foreach \x in {180, 240, 300, 60, 120}
\draw[fill=white] (\x: 1) circle[radius=2pt]; 
\draw[fill=white] (180: 2) circle[radius=2pt]; 
\node [above] at (-2, 0) {$D$};
\node[above] at (-1.05, 0) {$C_1$};
\node[below] at (240: 1) {$C_2$};
\node[below] at (300: 1) {$C_3$};
\node[above] at (60: 1) {$C_{b-1}$};
\node[above] at (120: 1) {$C_b$};
\end{scope}
\end{tikzpicture}
\end{center}
If $b=1$, then $C$ is a nodal rational curve with $C^2=0$, and the above dual graph degenerate to
\begin{center}
\begin{tikzpicture}[font=\tiny]
\begin{scope}
\draw  (-1, 0) --(0,0);
\draw (0,0) [out = 30, in = 150] to (0:1);
\draw (0,0) [out = -30, in = -150] to (0:1);
\draw[fill=white] (0,0) circle[radius=2pt];
\draw[fill=white]  (-1, 0) circle[radius=2pt];
\draw[fill=black] (-.5,0) circle[radius=2pt];
\draw[fill=black] (1, 0) circle[radius=2pt];
\node[above]  at (1, 0) {$E_1$};
\node[below] at (-.5,0){$E_0$};
\node[above] at (0,0) {$C^w$};
\node[above] at (-1, 0) {$D^w$};
\draw[->] (1.5,0) --node[above]{$\tilde \gamma$}(2.5,0);
\end{scope}

\begin{scope}[xshift=4cm]
\draw  (-1, 0) --(0,0);
\draw (0,0) [out = 30, in = 150] to (0:1);
\draw (0,0) [out = -30, in = -150] to (0:1);
\draw[fill=white] (0,0) circle[radius=2pt];
\draw[fill=white]  (-1, 0) circle[radius=2pt];
\draw[fill=black] (1, 0) circle[radius=2pt];
\node[above]  at (1, 0) {$\tilde E_1$};
\node[above] at (0,0) {$\tilde C$};
\node[above] at (-1, 0) {$\tilde D$};
\draw[->] (1.5,0) --node[above]{$\rho$}(2.5,0);
\end{scope}

\begin{scope}[xshift=8cm]
\draw  (-1, 0) --(0,0);
\draw (.5,0) circle[radius=.5cm];
\draw[fill=white] (0,0) circle[radius=2pt];
\draw[fill=white]  (-1, 0) circle[radius=2pt];
\node[above] at (-.1,0) {$ C$};
\node[above] at (-1, 0) {$ D$};
\end{scope}
\end{tikzpicture}
\end{center}
There is exactly one critical value $c=\frac{1}{2}$.  For $\frac{1}{2}<c\leq 1$,  we have $\tilde X^{(c)}=\tilde X$ and
\[
K_{\tilde X^{(c)}} + B_{\tilde X^{(c)}}\sim_\RR   a\tilde C_1 + \frac{1}{2} \sum_{2\leq i\leq b} \tilde C_i + \frac{a}{2} \tilde D +\sum_{i\in\{1,b\}}\tilde E_i
\]
where $a=\min\{\frac{4}{7}, c\}$. For $0<c\leq \frac{1}{2}$, we have $\tilde X^{(c)} = Z$ and 
\[
K_{\tilde X^{(c)}} + B_{\tilde X^{(c)}}\sim_\RR cC+\frac{c}{2}D
\]
Now it is straightforward to compute
\begin{equation}\label{eq: I_b}
\min \KK^2\left(Z, B_Z^{(c)}; B_Z'\right) = \vol(K_{\tilde X^{(c)}} + B_{\tilde X^{(c)}}) = 
\begin{cases}
\frac{1}{2}c^2 & \text{if } c\leq \frac{1}{2} \\
-\frac{7}{2}c^2+4c-1 & \text{if } \frac{1}{2}< c\leq \frac{4}{7} \\
\frac{1}{7} & \text{if } c> \frac{4}{7} 
\end{cases}
\end{equation}

\medskip

\noindent{\bf Case: $C$ supports a curve of canonical type II.} In this case, $C$ is a cuspical rational curve with $C^2=0$ and $D$ is a $(-2)$-curve intersecting $C$ transversally at a smooth point. The sketches of visible curves on $W\xrightarrow{\tilde \gamma} \tilde X \xrightarrow{\rho} Z$ are as follows:
\begin{center}
\begin{tikzpicture}[font=\tiny]
\begin{scope}
\path (0,0) coordinate (O);
\path (.8,1) coordinate (P1);
\path (.8,-1) coordinate (P2);
\path (0, 1) coordinate (Q1);
\path (0, -1) coordinate (Q2);
\path (.25, .8) coordinate (R1);
\path (1.5, .8) coordinate (R2);
\path (1,1) coordinate (S1);
\path (1.5,.25) coordinate (S2);
\draw (P1) to  [out=210, in = 90] (O);
\draw (P2) to  [out=150, in = -90] (O);
\draw[dashed] (Q1) -- (Q2);
\draw[dashed] (R1) -- node[below]{$E_1$}(R2);
\draw (S1) -- node[very near end, right]{$D^w$}(S2);
\node[right] at (P2) {$C^w$};
\node[left] at (Q1) {$E_2$};
\draw[->] (2.5,0) --node[above]{$\tilde \gamma$}(3.5,0);
\end{scope}
\begin{scope}[xshift = 4cm]
\path (0,0) coordinate (O);
\path (.8,1) coordinate (P1);
\path (.8,-1) coordinate (P2);
\path (0, 1) coordinate (Q1);
\path (0, -1) coordinate (Q2);
\path (.25, .8) coordinate (R1);
\path (1.5, .8) coordinate (R2);
\path (1,1) coordinate (S1);
\path (1.5,.25) coordinate (S2);
\draw (P1) to  [out=210, in = 90] (O);
\draw (P2) to  [out=150, in = -90] (O);
\draw[dashed] (Q1) -- (Q2);
\draw (R1) -- node[very near end, above]{$\tilde D$}(S2);
\node[right] at (P2) {$\tilde C$};
\node[left] at (Q1) {$\tilde E_2$};
\draw[->] (2.5,0) --node[above]{$\rho$}(3.5,0);
\end{scope}
\begin{scope}[xshift = 8cm]
\path (0,0) coordinate (O);
\path (1,1) coordinate (P);
\path (1,-1) coordinate (Q);
\path (.5,0) coordinate (R);
\path (1, .5) coordinate (T1);
\path (1, -.5) coordinate (T2);
\path (.5, 1) coordinate (S1);
\path (1.5, .5) coordinate (S2);
\draw (O) .. controls (R) and (T1).. (P);
\draw (O) .. controls (R) and (T2).. (Q);
\draw (S1) -- (S2);
\node[above] at (S2) {$D$};
\node[right] at (Q) {$C$};
\end{scope}
\end{tikzpicture}
\end{center}
There is exactly one critical value $c=\frac{1}{2}$. For $\frac{1}{2}<c\leq 1$, we have $\tilde X^{(c)} =\tilde  X$ and 
\[
K_{\tilde X^{(c)}} + B_{\tilde X^{(c)}}\sim_\RR  a\tilde C + \frac{a}{2} \tilde D + \tilde E_2
\]
where $a=\min\{c,\frac{4}{7}\}$. For $0<c\leq \frac{1}{2}$, we have $\tilde X^{(c)}=Z$ and
\[
K_{\tilde X^{(c)}} + B_{\tilde X^{(c)}}\sim_\RR  a C + \frac{a}{2}  D  
\]
The formula for $\min \KK^2\left(Z, B_Z^{(c)}; B_Z'\right) = \vol(K_{\tilde X^{(c)}} + B_{\tilde X^{(c)}})$ is the same with the case of $I_b, b\geq 1$; see \eqref{eq: I_b}.

\medskip

\noindent{\bf Case: $C$ supports a curve of canonical type III.}
  In this case, $C=C_1+C_2$ consists of two $(-2)$-curves intersecting at one point with multiplicity two, and $D$ is a $(-2)$-curve intersecting $C$ transversally at a smooth point, say, of $C_1$.  The sketches of visible curves on $W\xrightarrow{\tilde \gamma} \tilde X \xrightarrow{\rho} Z$ are as follows:
 \begin{center}
\begin{tikzpicture}[font=\tiny]
\begin{scope}
\path (0,0) coordinate (O);
\path (.8,1) coordinate (P1);
\path (.8,-1) coordinate (P2);
\path (-.8,1) coordinate (P3);
\path (-.8,-1) coordinate (P4);
\path (-0.8, 0) coordinate (Q1);
\path (0.8, 0) coordinate (Q2);
\path (.25, .8) coordinate (R1);
\path (1.5, .8) coordinate (R2);
\path (1,1) coordinate (S1);
\path (1.5,.25) coordinate (S2);
\draw (P1) to  (P4);
\draw (P2) to (P3);
\draw[dashed] (Q1) node[left]{$E_2$} -- (Q2);
\draw[dashed] (R1) -- node[below]{$E_1$}(R2);
\draw (S1) -- node[very near end, right]{$D^w$}(S2);
\node[right] at (P2) {$C_2^w$};
\node[left] at (P4) {$C_1^w$};
\draw[->] (2,0) --node[above]{$\tilde \gamma$}(3,0);
\end{scope}

\begin{scope}[xshift=5cm]
\path (0,0) coordinate (O);
\path (.8,1) coordinate (P1);
\path (.8,-1) coordinate (P2);
\path (-.8,1) coordinate (P3);
\path (-.8,-1) coordinate (P4);
\path (1, 0.5) coordinate (Q1);
\path (0, 1) coordinate (Q2);
\path (-0.8, 0) coordinate (R1);
\path (0.8, 0) coordinate (R2);
\draw (P1) to  (P4);
\draw (P2) to (P3);
\draw (Q1) node[right]{$\tilde D$} -- (Q2);
\draw[dashed] (R1) node[left]{$\tilde E_2$} -- (R2);
\node[right] at (P2) {$\tilde C_2$};
\node[left] at (P4) {$\tilde C_1$};
\draw[->] (2,0) --node[above]{$\rho$}(3,0);
\end{scope}

\begin{scope}[xshift =10cm]
\path (0,0) coordinate (O);
\path (.8,1) coordinate (P1);
\path (.8,-1) coordinate (P2);
\path (-.8,1) coordinate (P3);
\path (-.8,-1) coordinate (P4);
\path (1, 0.5) coordinate (Q1);
\path (0, 1) coordinate (Q2);
\draw (P1) to  [out=210, in = 90] (O);
\draw (P2) to  [out=150, in = -90] (O);
\draw (P3) to  [out=-30, in = 90] (O);
\draw (P4) to  [out=30, in = -90] (O);
\draw (Q1) -- (Q2);
\node[right] at (P2) {$C_2$};
\node[left] at (P4) {$C_1$};
\node[right] at (Q1) {$D$};
\end{scope}
\end{tikzpicture}
\end{center}
There is exactly one critical value $c=\frac{1}{2}$. For $\frac{1}{2}<c\leq 1$,  we have $\tilde X^{(c)} =\tilde  X$ and
\[
K_{\tilde X^{(c)}} + B_{\tilde X^{(c)}}\sim_\RR a\tilde C_1 + \frac{a+1}{3}\tilde C_2 + \frac{a}{2} \tilde D +  \tilde E_2
\]
where $a=\min\{c,\frac{8}{13}\}$. For $0<c\leq \frac{1}{2}$, we have $\tilde X^{(c)}=Z$ and
\[
K_{\tilde X^{(c)}} + B_{\tilde X^{(c)}}\sim_\RR    c C + \frac{c}{2} D
\]
The formula for the minimal volume is as follows:
\begin{equation}\label{eq: III}
\min \KK^2\left(Z, B_Z^{(c)}; B_Z'\right) = \vol(K_{\tilde X^{(c)}} + B_{\tilde X^{(c)}}) = 
\begin{cases}
\frac{1}{2}c^2 & \text{if }  c\leq \frac{1}{2} \\
-\frac{13}{6}c^2+\frac{8}{3}c-\frac{2}{3} & \text{if } \frac{1}{2}<  c\leq \frac{8}{13} \\
\frac{2}{13} & \text{if } c> \frac{8}{13} 
\end{cases}
\end{equation}
\medskip

\noindent{\bf Case: $C$ supports a curve of canonical type IV.}
 In this case, $C=C_1+C_2+C_3$ consists of three $(-2)$-curves intersecting at one point, and $D$ is a $(-2)$-curve intersecting, say $C_1$, transversally at a smooth point.   The sketches of visible curves on $W\xrightarrow{\tilde \gamma} \tilde X \xrightarrow{\rho} Z$ are as follows:
 \begin{center}
\begin{tikzpicture}[font=\tiny]
\begin{scope}[xshift=0cm]
\path (0,0) coordinate (O);
\path (-.5, .5) coordinate (P1);
\path (.8, .5) coordinate (P2);
\path  (-.5, 0) coordinate (Q1);
\path  (.5, 0) coordinate (Q2);
\path (-.5, -.5)  coordinate (R1);
\path (.5, -.5)  coordinate (R2);
\path (.5,0.4) coordinate (S1);
\path (.5, 1) coordinate (S2);
\path (0,-1) coordinate (T1);
\path (0, 1) coordinate (T2);
\path (0.25,0.75) coordinate (U1);
\path (1.2, 1) coordinate (U2);
\draw (P1)--(P2);
\draw (Q1)--(Q2);
\draw (R1)--(R2);
\draw[dashed] (S1) -- (S2);
\draw[dashed] (T1)--(T2);
\draw (U1)--(U2);
\node[left] at (P1) {$C_1^w$};
\node[left] at (R1) {$C_3^w$};
\node[left] at (Q1) {$C_2^w$};
\node[above] at (U2){$D^w$};
\node[above] at (S2) {$E_1$};
\node[right] at (T1) {$E_2$};
\draw[->] (1.5,0) --node[above]{$\tilde \gamma$}(2.5,0);
\end{scope}

\begin{scope}[xshift=4cm]
\path (0,0) coordinate (O);
\path (-.5, .5) coordinate (P1);
\path (.8, .5) coordinate (P2);
\path  (-.5, 0) coordinate (Q1);
\path  (.5, 0) coordinate (Q2);
\path (-.5, -.5)  coordinate (R1);
\path (.5, -.5)  coordinate (R2);
\path (0.25,0.25) coordinate (S1);
\path (1, 1) coordinate (S2);
\path (0,-1) coordinate (T1);
\path (0, 1) coordinate (T2);
\draw (P1)--(P2);
\draw (Q1)--(Q2);
\draw (R1)--(R2);
\draw (S1) -- (S2);
\draw[dashed] (T1)--(T2);
\node[left] at (P1) {$\tilde C_1$};
\node[left] at (R1) {$\tilde C_3$};
\node[left] at (Q1) {$\tilde C_2$};
\node[right]at (S2){$\tilde D$};
\node[right] at (T1) {$\tilde E_2$};
\draw[->] (1.5,0) --node[above]{$\rho$}(2.5,0);
\end{scope}

\begin{scope}[xshift=8cm]
\path (0,0) coordinate (O);
\path (0:1) coordinate (P1);
\path (180:1) coordinate (P2);
\path (60:1) coordinate (Q1);
\path (240:1) coordinate (Q2);
\path (120:1) coordinate (R1);
\path (300:1) coordinate (R2);
\path (0,.7) coordinate (S1);
\path (1, 0.5) coordinate (S2);
\draw (P1)--(P2);
\draw (Q1)--(Q2);
\draw (R1)--(R2);
\draw (S1)--(S2);
\node[right] at (P1) {$C_3$};
\node[right] at (R2) {$C_2$};
\node[left] at (Q2) {$C_1$};
\node[right]at (S2){$D$};
\end{scope}
\end{tikzpicture}
\end{center}
There is exactly one critical value $c=\frac{1}{3}$. For $\frac{1}{3}<c\leq 1$,  we have $\tilde X^{(c)} =\tilde  X$ and
\[
K_{\tilde X^{(c)}} + B_{\tilde X^{(c)}}\sim_\RR   + a\tilde C_1 +\frac{1}{3}(\tilde C_2  + \tilde C_3) + \frac{a}{2} \tilde D  + \tilde E_2
\]
where $a=\min\{c,\frac{2}{5}\}$. For $0<c\leq \frac{1}{3}$, we have $\tilde X^{(c)}=Z$ and
\[
K_{\tilde X^{(c)}} + B_{\tilde X^{(c)}}\sim_\RR  c C + \frac{c}{2} D 
\]
The formula for the minimal volume is as follows:
\begin{equation}\label{eq: IV}
\min \KK^2\left(Z, B_Z^{(c)}; B_Z'\right) = \vol(K_{\tilde X^{(c)}} + B_{\tilde X^{(c)}}) = 
\begin{cases}
\frac{1}{2}c^2 & \text{if } c\leq \frac{1}{3} \\
-\frac{5}{2}c^2+2c-\frac{1}{3} & \text{if } \frac{1}{3}< c\leq \frac{2}{5} \\
\frac{1}{15} & \text{if } c> \frac{2}{5} 
\end{cases}
\end{equation}
\medskip

\noindent{\bf  Case: $C$ supports a curve of canonical type $\I_0^*$, and $D$ is attached to the central component of $C$.}
In this case, the dual graphs of visible curves on $W\xrightarrow{\rho_W} Z$ are as follows:
\begin{center}
\begin{tikzpicture}[font=\tiny]
\begin{scope}[scale=1.5]
\node[wbullet, label=below:$D^w$] (B1) at (0:1) {};
\node[wbullet, label=right:$C_1^w$] (B2) at (72:1) {};
\node[wbullet, label=above:$C_2^w$] (B3) at (144:1) {};
\node[wbullet, label=below:$C_3^w$] (B4) at (216:1){};
\node[wbullet, label=below:$C_4^w$] (B5) at (288:1){};
\node[wbullet](B0) at (0,0){}
edge node[bbullet, midway]{} (B1)
edge node[bbullet, midway]{} (B2)
edge node[bbullet, midway]{} (B3)
edge node[bbullet, midway]{} (B4)
edge node[bbullet, midway]{} (B5);
\node at (-36:.3) {$C_0^w$};
\node at (20:0.5){$E_0$};
\node at (92:0.5){$E_1$};
\node at (164:0.5){$E_2$};
\node at (236:0.5){$E_3$};
\node[right] at (288:0.5){$E_4$};
\draw[->] (1.5,0) -- (2.5,0) node[midway, above]{$\rho_W$};
\end{scope}

\begin{scope}[xshift=6cm, scale=1.5]
\node[wbullet, label=below:$D$] (B1) at (0:1) {};
\node[wbullet, label=right:$C_1$] (B2) at (72:1) {};
\node[wbullet, label=above:$C_2$] (B3) at (144:1) {};
\node[wbullet, label=below:$C_3$] (B4) at (216:1){};
\node[wbullet, label=below:$C_4$] (B5) at (288:1){};
\node[wbullet](B0) at (0,0){}
edge (B1)
edge  (B2)
edge  (B3)
edge  (B4)
edge  (B5);
\node at (-36:.3) {$C_0$};
\end{scope}
\end{tikzpicture}
\end{center}
There is exactly one critical value $c=\frac{2}{3}$. For $\frac{2}{3}<c\leq 1$, we have $\tilde X^{(c)} = W$ and 
\[
K_{\tilde X^{(c)}} + B_{\tilde X^{(c)}}\sim_\RR a C_0^w  +  \frac{1}{3} \sum_{1\leq i\leq 4} C_i^w +\frac{1}{3} D^w+  \sum_{0\leq i\leq 4} E_i 
\]
where $a=\min\{{5\over 7}, c\}$. For $0 <c\leq \frac{2}{3}$, we have $\tilde X^{(c)} = Z$ and
\[
K_{\tilde X^{(c)}} + B_{\tilde X^{(c)}}\sim_\RR cC_0 + \frac{c}{2}\left(\sum_{1\leq i\leq 4} C_i + D\right)
\]
The formula for the minimal volume is
\begin{equation}\label{eq: I_0*1}
\min \KK^2\left(Z, B_Z^{(c)}; B_Z'\right) = \vol(K_{\tilde X^{(c)}} + B_{\tilde X^{(c)}}) =
\begin{cases}
\frac{1}{2}c^2 & \text{if } c\leq \frac{2}{3} \\
-7c^2 + 10 c - {10\over 3}& \text{if } \frac{2}{3}< c\leq \frac{5}{7} \\
{5\over 21} & \text{if } c> \frac{5}{7} 
\end{cases}
\end{equation}
\medskip

\noindent{\bf  Case: $C$ supports a curve of canonical type $\I_0^*$,  and $D$ is attached to an end component of $C$.}
The dual graphs of the visible curves on $\begin{tikzcd}W\arrow[r, "\tilde \gamma"]\arrow[rr, bend right,"\rho_W"']& \tilde X\arrow[r, "\rho"] & Z\end{tikzcd}$ are as follows:
\begin{center}
\begin{tikzpicture}[font=\tiny]
\begin{scope}[xshift=2.5cm, scale =1.2]
\node[wbullet, label = below:$D^w$] (D) at (180:2){};
\node[wbullet, label = below right :$C_0^w$] (C0) at (0,0){};
\node[wbullet, label = below:$C_3^w$] (C3) at (0:1){};
\node[wbullet, label = right:$C_4^w$] (C4) at (90:1){};
\node[wbullet, label = below:$C_1^w$] (C1) at (180:1){};
\node[wbullet, label = right:$C_2^w$] (C2) at (270:1){};
\path (C0) edge node[bbullet, midway, label=above:$E_1$]{} (C1)
edge node[bbullet, midway, label=left:$E_2$](E2){}(C2)
edge node[bbullet, midway, label=above:$E_3$](E3){}(C3)
edge node[bbullet, midway, label=left:$E_4$](E4){}(C4);
\draw (D) -- node[bbullet, midway, label=above:$E_0$]{}(C1);
\draw[->](-1,-1)--node[midway,sloped, above]{$\tilde \gamma$}(-2,-2);
\draw[->](1,-1)--node[midway,sloped, above]{$\rho_W$}(2,-2);
\end{scope}
\begin{scope}[yshift=-4cm]
\node[wbullet, label = below:$\tilde D$] (D) at (180:2){};
\node[wbullet, label = below right :$\tilde C_0$] (C0) at (0,0){};
\node[wbullet, label = below:$\tilde C_3$] (C3) at (0:1){};
\node[wbullet, label = right:$\tilde C_4$] (C4) at (90:1){};
\node[wbullet, label = below:$\tilde C_1$] (C1) at (180:1){};
\node[wbullet, label = right:$\tilde C_2$] (C2) at (270:1){};
\path (C0) edge node[bbullet, midway, label=above:$\tilde E_1$]{} (C1)
edge (C2)
edge (C3)
edge (C4);
\draw (D) -- (C1);
\draw[->](1.5,0)--node[midway,sloped, above]{$\rho$}(2.5,0);
\end{scope}
\begin{scope}[xshift=5cm, yshift = -4cm]
\node[wbullet, label = below:$D$] (D) at (180:2){};
\node[wbullet, label = below right :$C_0$] (C0) at (0,0){};
\node[wbullet, label = below:$C_3$] (C3) at (0:1){};
\node[wbullet, label = right:$C_4$] (C4) at (90:1){};
\node[wbullet, label = below:$C_1$] (C1) at (180:1){};
\node[wbullet, label = right:$C_2$] (C2) at (270:1){};
\path (C0) edge (C1)
edge (C2)
edge (C3)
edge (C4);
\draw (D) --(C1);
\end{scope}
\end{tikzpicture}
\end{center}
There is exactly one critical value $c=\frac{3}{5}$. If $\frac{3}{5}<c\leq 1$, we have $\tilde X^{(c)}= \tilde X$ and 
\[
K_{\tilde X^{(c)}} + B_{\tilde X^{(c)}}\sim_\RR a\tilde C_0 + \frac{2}{5} \tilde C_1  +\frac{a}{2}\sum_{2\leq i\leq 4} \tilde C_i+ \frac{1}{5} \tilde D + \tilde E_1  
\]
where $a=\min\{{2\over 3}, c\}$.  If $0<c\leq \frac{3}{5}$, we have $\tilde X^{(c)} = Z$ and
\[
K_{\tilde X^{(c)}} + B_{\tilde X^{(c)}}\sim_\RR  c C_0  +  \frac{2c}{3} C_1 + \frac{c}{2}\sum_{2\leq i\leq 4} C_i  + \frac{c}{3} D
\]
The formula for the minimal volume is
\begin{equation}\label{eq: I_0*2}
\min \KK^2\left(Z, B_Z^{(c)}; B_Z'\right) = \vol(K_{\tilde X^{(c)}} + B_{\tilde X^{(c)}}) =
\begin{cases}
\frac{1}{6}c^2 & \text{if } c \leq \frac{3}{5}\\
-\frac{3}{2}c^2 + 2 c -\frac{3}{5}& \text{if }\frac{3}{5}< c\leq \frac{2}{3} \\
{1\over 15} & \text{if } c>  \frac{2}{3}
\end{cases}
\end{equation}

\medskip

\noindent{\bf  Case: $C$ supports a curve of canonical type $\I_b^*,\, b\geq 1$,  and $D$ is attached to an end component of $C$.}
The dual graphs of the visible curves on $\begin{tikzcd}W\arrow[r, "\tilde \gamma"]\arrow[rr, bend right,"\rho_W"']& \tilde X\arrow[r, "\rho"] & Z\end{tikzcd}$ are as follows:
\begin{center}
\begin{tikzpicture}[inner sep = 0, font=\tiny, scale =.7]
\begin{scope}[xshift=4cm, scale=1.5]
\node[wbullet, label = below:$D^w$] (D) at (-2,0){};
\node[wbullet, label = below:$C_{1}^w$] (C1) at (-1,0){};
\node[wbullet, label = below:$C_2^w$] (C2) at (0,0){};
\node[wbullet, label = below:$C_3^w$] (C3) at (1,0){};
\node[wbullet, label = below:$C_{b+1}^w$] (Cb+1) at (3,0){};
\node[wbullet, label = below:$C_{b+2}^w$] (Cb+2) at (4,0){};
\node[wbullet, label = below:$C_{b+3}^w$] (Cb+3) at (5,0){};
\node[wbullet, label = right:$C_{b+5}^w$] (Cb+5) at (4,1){};
\node[wbullet, label = right:$C_{b+4}^w$] (Cb+4) at (0,1){};
\draw (D) --node[bbullet, label = above: $E_1$](E1){} (C1) --node[bbullet, label = above: $E_{2}$](2){}  (C2) --node[bbullet, label = above: $E_3$](E3){}  (C3);
\draw[dashed](C3)--(Cb+1);
\draw (Cb+1)--node[bbullet, label = above: $E_{b+2}$](Eb+2){} (Cb+2)--node[bbullet, label = above: $E_{b+3}$](Eb+3){} (Cb+3);
\draw (Cb+2)--node[bbullet, label =  left: $E_{b+5}$](Eb+5){} (Cb+5);
\draw (C2)--node[bbullet, label =  right: $E_{b+4}$](Eb+4){} (Cb+4);
\draw[->](0,-1)--node[midway,sloped, above]{$\tilde \gamma$}(-.6,-1.6);
\draw[->](4,-1)--node[midway,sloped, above]{$\rho_W$}(4.6,-1.6);
\end{scope}
\begin{scope}[yshift=-4cm]
\node[wbullet, label = below:$\tilde D$] (D) at (-2,0){};
\node[wbullet, label = below:$\tilde C_1$] (C1) at (-1,0){};
\node[wbullet, label = below:$\tilde C_2$] (C2) at (0,0){};
\node[wbullet, label = below:$\tilde C_3$] (C3) at (1,0){};
\node[wbullet, label = below:$\tilde C_{b+1}$] (Cb+1) at (3,0){};
\node[wbullet, label = below:$\tilde C_{b+2}$] (Cb+2) at (4,0){};
\node[wbullet, label = below:$\tilde C_{b+3}$] (Cb+3) at (5,0){};
\node[wbullet, label = right:$\tilde C_{b+4}$] (Cb+4) at (0,1){};
\node[wbullet, label = right:$\tilde C_{b+5}$] (Cb+5) at (4,1){};
\draw (D) -- (C1) -- (C2) --node[bbullet, label = above: $\tilde E_3$](E3){}   (C3);
\draw[dashed](C3)--(Cb+1);
\draw (Cb+1)--(Cb+2)--(Cb+3);
\draw (Cb+2)--(Cb+5);
\draw (C2)--(Cb+4);
\draw[->](5.5,0)--node[midway,sloped, above]{$\rho$}(6.5,0);
\end{scope}
\begin{scope}[xshift=9cm, yshift=-4cm]
\node[wbullet, label = below:$D$] (D) at (-2,0){};
\node[wbullet, label = below:$C_1$] (C1) at (-1,0){};
\node[wbullet, label = below:$C_2$] (C2) at (0,0){};
\node[wbullet, label = below:$C_3$] (C3) at (1,0){};
\node[wbullet, label = below:$C_{b+1}$] (Cb+1) at (3,0){};
\node[wbullet, label = below:$C_{b+2}$] (Cb+2) at (4,0){};
\node[wbullet, label = below:$C_{b+3}$] (Cb+3) at (5,0){};
\node[wbullet, label = right:$C_{b+4}$] (Cb+4) at (0,1){};
\node[wbullet, label = right:$C_{b+5}$] (Cb+5) at (4,1){};
\draw (D) -- (C1) -- (C2) -- (C3);
\draw[dashed](C3)--(Cb+1);
\draw (Cb+1)--(Cb+2)--(Cb+3);
\draw (C2)--(Cb+4);
\draw (Cb+2)--(Cb+5);
\end{scope}
\end{tikzpicture}
\end{center}
There is exactly one critical value $c=\frac{1}{2}$. If $\frac{1}{2}<c\leq 1$, we have $\tilde X^{(c)}= \tilde X$ and 
\[
K_{\tilde X^{(c)}} + B_{\tilde X^{(c)}}\sim_\RR \frac{2a}{3} \tilde C_{1} + a\tilde C_2 + \frac{1}{2} \sum_{3\leq i\leq b+2} \tilde C_i +\frac{1}{4} \tilde C_{b+3} + \frac{a}{2} \tilde C_{b+4}+\frac{1}{4} \tilde C_{b+5} + \frac{a}{3} \tilde D+ \tilde E_3 
\]
where $a=\min\{{6\over 11}, c\}$.  If $0<c\leq \frac{1}{2}$, we have $\tilde X^{(c)} = Z$ and
\[
K_{\tilde X^{(c)}} + B_{\tilde X^{(c)}}\sim_\RR  \frac{2c}{3} C_{1} + c\sum_{2\leq i\leq b+2} C_i + \frac{c}{2}\sum_{b+3\leq i\leq b+5} C_i  + \frac{c}{3} D 
\]
The formula for the volume is
\begin{equation}\label{eq: I_0*}
\min \KK^2\left(Z, B_Z^{(c)}; B_Z'\right) = \vol(K_{\tilde X^{(c)}} + B_{\tilde X^{(c)}}) =
\begin{cases}
\frac{1}{6}c^2 & \text{if } 0<c\leq \frac{1}{2} \\
-{11\over 6}c^2 + 2 c - {1\over 2}& \text{if }  \frac{1}{2}< c\leq \frac{6}{11} \\
{1\over 22} & \text{if } c>  {6\over 11}
\end{cases}
\end{equation}

\medskip

\noindent{\bf Case: $C$ supports a curve of canonical type $\II^*$ and $D$ is attached to the end component of the longest chain of $C$.}
The dual graphs of the visible curves on $\begin{tikzcd}W\arrow[r, "\tilde \gamma"]\arrow[rr, bend right,"\rho_W"']& \tilde X\arrow[r, "\rho"] & Z\end{tikzcd}$ are as follows:
\begin{center}
\begin{tikzpicture}[font=\tiny, scale=.7]
\begin{scope}[xshift=3.3cm, scale=1.5]
\node[wbullet, label = below:$D^w$]  (D) at (0,0) {};
\foreach \x in {1,2,3,4,5,6,7,8}
\node[wbullet, label = below:$C_\x^w$] (C\x) at (\x,0) {};
\node[wbullet, label = right:$C_9^w$] at (6, 1) (C9){};
\foreach \x/\y in {1/2,2/3,3/4,4/5,5/6,6/7,7/8}
\draw (C\x)--node[bbullet, label = above: $E_\y$](E\y){}(C\y);
\draw (D) -- node[bbullet, label = above: $E_1$](E1){}(C1);
\draw (C6) -- node[bbullet, label = right: $E_9$](E9){}(C9);
\draw[->] (3,-1) --node[above, sloped]{$\tilde \gamma$} (2.4,-1.6);
\draw[->] (5,-1) --node[above, sloped]{$\rho_W$} (5.6,-1.6);
\end{scope}
\begin{scope}[yshift=-4cm]
\node[wbullet, label = below:$\tilde D$]  (D) at (0,0) {};
\foreach \x in {1,2,3,4,5,6,7,8}
\node[wbullet, label = below:$\tilde C_\x$] (C\x) at (\x,0) {};
\node[wbullet, label = right:$\tilde C_9$] at (6, 1) (C9){};
\draw (D) --(C1)--(C2)--(C3)--(C4)--(C5)--node[bbullet, midway, label = above: $\tilde E_6$](E6){}(C6)--(C7)--(C8);
\draw (C6) -- (C9);
\draw[->] (8.5,0) --node[above]{$\rho$} (9.5,0);
\end{scope}
\begin{scope}[xshift=10cm, yshift=-4cm]
\node[wbullet, label = below:$D$]  (D) at (0,0) {};
\foreach \x in {1,2,3,4,5,6,7,8}
\node[wbullet, label = below:$C_\x$] (C\x) at (\x,0) {};
\node[wbullet, label = right:$C_9$] at (6, 1) (C9){};
\foreach \x/\y in {1/2,2/3,3/4,4/5,5/6,6/7,7/8, 6/9}
\draw (C\x)--(C\y);
\draw (D) -- (C1);
\end{scope}
\end{tikzpicture}
\end{center}
There is exactly one critical value $c=\frac{7}{13}$. If $\frac{7}{13}<c\leq 1$, we have $\tilde X^{(c)}= \tilde X$ and 
\[
K_{\tilde X^{(c)}} + B_{\tilde X^{(c)}}\sim_\RR \tilde E_6 + \frac{1}{13} \left(\tilde D+\sum_{1\leq i\leq 5} (i+1)\tilde C_i\right) + \frac{a}{6}\left(6\tilde C_6 +4 \tilde C_{7}   +  2\tilde C_{8}+ 3 \tilde C_9\right)
\]
where $a=\min\{{6\over 11}, c\}$.  If $0<c\leq \frac{7}{13}$, we have $\tilde X^{(c)} = Z$ and
\[
K_{\tilde X^{(c)}} + B_{\tilde X^{(c)}}\sim_\RR  c C_6 + \frac{c}{7} \left(D +\sum_{1\leq i\leq 5} (i+1) C_i\right) + \frac{c}{6}\left(4 \tilde C_{7}   +  2\tilde C_{8}+ 3 \tilde C_9\right)
\]
The formula for the minimal volume is
\begin{equation}\label{eq: II*}
\min \KK^2\left(Z, B_Z^{(c)}; B_Z'\right) = \vol(K_{\tilde X^{(c)}} + B_{\tilde X^{(c)}}) =
\begin{cases}
\frac{1}{42}c^2 & \text{if } c\leq\frac{7}{13} \\
-{11\over 6}c^2 + 2 c - \frac{7}{13}& \text{if } \frac{7}{13}< c\leq \frac{6}{11} \\
{1\over 143} & \text{if } c \geq \frac{6}{11}
\end{cases}
\end{equation}

\noindent{\bf Case: $C$ supports a curve of canonical type $\III^*$.}
The dual graphs of the visible curves on $\begin{tikzcd}W\arrow[r, "\tilde \gamma"]\arrow[rr, bend right,"\rho_W"']& \tilde X\arrow[r, "\rho"] & Z\end{tikzcd}$ are as follows:
\begin{center}
\begin{tikzpicture}[font=\tiny, scale=.7]
\begin{scope}[xshift=3.3cm, scale=1.5]
\node[wbullet, label = below:$D^w$]  (D) at (0,0) {};
\foreach \x in {1,2,3,4,5,6,7}
\node[wbullet, label = below:$C_\x^w$] (C\x) at (\x,0) {};
\node[wbullet, label = right:$C_8^w$] at (4, 1) (C8){};
\foreach \x/\y in {1/2,2/3,3/4,4/5,5/6,6/7}
\draw (C\x)--node[bbullet, label = above: $E_\y$](E\y){}(C\y);
\draw (D) -- node[bbullet, label = above: $E_1$](E1){}(C1);
\draw (C4) -- node[bbullet, label = right: $E_8$](E8){}(C8);
\draw[->] (3,-1) --node[above, sloped]{$\tilde \gamma$} (2.4,-1.6);
\draw[->] (5,-1) --node[above, sloped]{$\rho_W$} (5.6,-1.6);
\end{scope}
\begin{scope}[yshift=-4cm]
\node[wbullet, label = below:$\tilde D$]  (D) at (0,0) {};
\foreach \x in {1,2,3,4,5,6,7}
\node[wbullet, label = below:$\tilde C_\x$] (C\x) at (\x,0) {};
\node[wbullet, label = right:$\tilde C_8$] at (4, 1) (C8){};
\draw (D) --(C1)--(C2)--(C3)--node[bbullet, midway, label = above: $\tilde E_4$](E4){}(C4)--(C5)--(C6)--(C7);
\draw (C4) -- (C8);
\draw[->] (7.5,0) --node[above]{$\rho$} (8.5,0);
\end{scope}
\begin{scope}[xshift=9cm, yshift=-4cm]
\node[wbullet, label = below:$D$]  (D) at (0,0) {};
\foreach \x in {1,2,3,4,5,6,7}
\node[wbullet, label = below:$C_\x$] (C\x) at (\x,0) {};
\node[wbullet, label = right:$C_8$] at (4, 1) (C8){};
\foreach \x/\y in {1/2,2/3,3/4,4/5,5/6,6/7, 4/8}
\draw (C\x)--(C\y);
\draw (D) -- (C1);
\end{scope}
\end{tikzpicture}
\end{center}
There is exactly one critical value $c=\frac{5}{9}$. If $\frac{5}{9}<c\leq 1$, then $\tilde X^{(c)} = \tilde X$ and 
\[
K_{\tilde X^{(c)}}+ B_{\tilde X^{(c)}} \sim_\RR \tilde E_4  + \frac{1}{9}\left(\tilde D + \sum_{1\leq i\leq 3}(i+1)\tilde C_i\right)+ \frac{a}{4}\left( \sum_{4\leq i\leq 7}(8-i)\tilde C_i + 2 \tilde C_8\right),
\]
where $a=\min\{{4\over 7}, c\}$. If $0<c\leq \frac{5}{9}$, then $\tilde X^{(c)} = Z$ and
\[
K_{\tilde X^{(c)}}+ B_{\tilde X^{(c)}} \sim_\RR \frac{c}{5}\left(D + \sum_{1\leq i\leq 3}(i+1)\tilde C_i \right) + \frac{c}{4}\left( \sum_{4\leq i\leq 7}(8-i)\tilde C_i + 2 \tilde C_8\right),
\]
The formula for the minimal volume is
\begin{equation}\label{eq: III*}
\min \KK^2\left(Z, B_Z^{(c)}; B_Z'\right) = \vol(K_{\tilde X^{(c)}} + B_{\tilde X^{(c)}}) =
\begin{cases}
\frac{1}{20}c^2 & \text{if } c\leq {5\over 9} \\
-{7\over 4}c^2 + 2 c - {5\over 9}& \text{if } {5\over 9}< c\leq \frac{4}{7} \\
{1\over 63} & \text{if }  c >{4\over 7}
\end{cases}
\end{equation}

\medskip

\noindent{\bf Case: $C$ supports a curve of canonical type $\IV^*$}.
The dual graphs of the visible curves on $\begin{tikzcd}W\arrow[r, "\tilde \gamma"]\arrow[rr, bend right,"\rho_W"']& \tilde X\arrow[r, "\rho"] & Z\end{tikzcd}$ are as follows:
\begin{center}
\begin{tikzpicture}[font=\tiny, scale=.7]
\begin{scope}[xshift=3.3cm, scale=1.5]
\node[wbullet, label = below:$D^w$]  (D) at (0,0) {};
\foreach \x in {1,2,3,4,5}
\node[wbullet, label = below:$C_\x^w$] (C\x) at (\x,0) {};
\foreach \y in {6, 7}
\node[wbullet, label = right:$C_\y^w$] at (3, \y-5) (C\y){};
\foreach \x/\y in {1/2,2/3,3/4,4/5}
\draw (C\x)--node[bbullet, label = above: $E_\y$](E\y){}(C\y);
\draw (D) -- node[bbullet, label = above: $E_1$](E1){}(C1);
\draw (C3) -- node[bbullet, label = left: $E_6$]{} (C6)-- node[bbullet, label = left: $E_7$]{}(C7);
\draw[->] (2,-1) --node[above, sloped]{$\tilde \gamma$} (1.4,-1.6);
\draw[->] (4,-1) --node[above, sloped]{$\rho_W$} (4.6,-1.6);
\end{scope}
\begin{scope}[yshift=-4cm]
\node[wbullet, label = below:$\tilde D$]  (D) at (0,0) {};
\foreach \x in {1,2,3,4,5}
\node[wbullet, label = below:$\tilde C_\x$] (C\x) at (\x,0) {};
\foreach \y in {6, 7}
\node[wbullet, label = right:$\tilde C_\y$] at (3, \y-5) (C\y){};
\draw (D) --(C1)--(C2)--node[bbullet, midway, label = above: $\tilde E_3$](E3){}(C3)--(C4)--(C5);
\draw (C3) --(C6)--(C7);
\draw[->] (6.5,0) --node[above]{$\rho$} (7.5,0);
\end{scope}
\begin{scope}[xshift=9cm, yshift=-4cm]
\node[wbullet, label = below:$D$]  (D) at (0,0) {};
\foreach \x in {1,2,3,4,5}
\node[wbullet, label = below:$C_\x$] (C\x) at (\x,0) {};
\foreach \y in {6, 7}
\node[wbullet, label = right:$C_\y$] at (3, \y-5) (C\y){};
\foreach \x/\y in {1/2,2/3,3/4,4/5,3/6,6/7}
\draw (C\x)--(C\y);
\draw (D) -- (C1);
\end{scope}
\end{tikzpicture}
\end{center}
In this case the critical value is $c=\frac{4}{7}$. If $\frac{4}{7}< c\leq 1$, then $\tilde X^{(c)} = \tilde X$, and we have
\[
K_{\tilde X^{(c)}} + B_{\tilde X^{(c)}}= \tilde E_3+\frac{1}{7}(\tilde D + 2\tilde C_1 +3 \tilde C_2) +\frac{a}{3}\left (\sum_{3\leq i\leq 5}(6-i) \tilde C_i+ 2 \tilde C_6 + \tilde C_7\right ) 
\]
where $a=\min\{c, {4\over 7}\}$. If $c\leq\frac{4}{7}$, then $\tilde X=Z$, and we have
\[
K_{\tilde X^{(c)}} + B_{\tilde X^{(c)}} =\frac{c}{4}(D + 2 C_1 +3 C_2) + \frac{c}{3}\left (\sum_{3\leq i\leq 5}(6-i)  C_i+ 2  C_6 +  C_7\right ) 
\]
The formula for the minimal volume is
\begin{equation}\label{eq: IV*}
\min \KK^2\left(Z, B_Z^{(c)}; B_Z'\right) = \vol(K_{\tilde X^{(c)}} + B_{\tilde X^{(c)}}) =
\begin{cases}
\frac{1}{12}c^2 & \text{if } c\leq\frac{4}{7} \\
-{5\over 3}c^2 + 2 c -\frac{4}{7}& \text{if } \frac{4}{7} < c\leq \frac{3}{5} \\
{1\over 35} & \text{if } c>\frac{3}{5} 
\end{cases}
\end{equation}

\begin{rmk}
The above cases does not exhaust the possibilities how the $(-2)$-curve $D$ is attached to the curve $C$. But the other choices of attachment would result in configurations of $(-2)$-curves, say $\sG'$,  strictly containing one of the configurations considered above, say $\sG$. Then the resulting minimal volume for $\sG'$ would be larger than that for $\sG$.
\end{rmk}

\section{The semistable part of the boundary divisor}\label{sec: ss}
In this subsection, we study the part of the boundary divisor $B$ that contributes to the geometric genus of a smooth projective log canonical surface $(X, B)$; see Definition~\ref{defn: ss} and Lemma~\ref{lem: pg}.
\begin{defn}
A reduced curve $B$ on a smooth projective surface $X$ is called \emph{semistable} if it has at most nodes as singularities and each of its smooth rational components intersects the other components of $B$ in more than one point. 
\end{defn}

\begin{lem}\label{lem: ss rat}
Let $B$ be a nodal reduced curve on a smooth projective surface $X$. Then $B$ is semistable if and only if $(K_X+B)\cdot B_i\geq 0$ for any smooth rational component $B_i$ of $B$.
\end{lem}
\begin{proof}
Let $B_i$ be a smooth rational component  of $B$. Then 
\[
(K_X+B)\cdot B_i = (K_X+B_i)\cdot B_i + (B-B_i)\cdot B_i = -2  + (B-B_i)\cdot B_i
\] 
Thus $B_i$ intersects the other components of $B$ in more than one point if and only if $(K_X+B)\cdot B_i\geq 0$.
\end{proof}
%If $B'$ is a connected component of $B$ then \[(K_X+B')B' = (K_X+B) B'\geq 0,\]and thus $p_a(B')>0$ by the adjunction formula. 

\begin{lem}\label{lem: ss max}
Let $(X, B)$ be a smooth projective log canonical surface. Then there is a (possibly empty) maximal semistable reduced subcurve $B^\s$ of $\lfloor B\rfloor$, in the sense that any semistable subcurve $B'$ of $\lfloor B\rfloor$ is contained in $B^\s$. 
\end{lem}
\begin{proof}
Since $(X, \lfloor B\rfloor)$ is log canonical, it is well-known that $\lfloor B\rfloor$ is a nodal curve (\cite[Theorem~2.31]{Kol13}). Suppose that $B'$ is a semistable subcurve of $\lfloor B\rfloor$. Any irreducible component $B_i$ of $B$ with $(K_X+B)\cdot B_i <0$ cannot be an irreducible component of $B'$, since otherwise we have by Lemma~\ref{lem: ss rat}
\[
(K_X+B)\cdot B_i \geq (K_X+B')\cdot B_i\geq 0.
\] 
Starting with $B^{(0)}=\lfloor B\rfloor$, we construct inductively $B^{(j+1)} = B^{(j)} - B_j$, where $B_j$ is an irreducible component of $B^{(j)}$ such that $(K_X+B^{(j)})\cdot B_j<0$. We stop at a subcurve $B^\s$ of $B$ such that $(K_X+B^\s)\cdot B_i\geq 0$ for any irreducible component $B_i$ of $B^\s$. Then $B^\s$ is semistable by Lemma~\ref{lem: ss rat}, and, by the above argument, we have $B'\subset B^\s$. Thus $B^\s$ is  the desired  maximal semistable subcurve of $\lfloor B\rfloor$.
\end{proof}

\begin{defn}\label{defn: ss}
Let $(X, B)$ be a projective log canonical surface and $\pi\colon(\tilde X, B_{\tilde X})\rightarrow (X, B)$ the minimal resolution. Let $B_{\tilde X}^\s$ be as in Lemma~\ref{lem: ss max}, and $B_{\tilde X}^\ns:=B_{\tilde X} - B_{\tilde X}^\s$. Then $B^\s:=\pi_*B_{\tilde X}^\s$ and $B^\ns:=\pi_*B_{\tilde X}^\ns $ are called the \emph{semistable part} and the \emph{non-semistable part} of $B$ respectively, and $B=B^\s+ B^\ns$ is called the \emph{semistable} decomposition of $B$.
\end{defn}

\begin{lem}\label{lem: pg}
Let $(X, B)$ be a smooth projective log canonical surface. Let $B =B^\s + B^\ns$ be decomposition into the the semistable part and the non-semistable part. Then the following holds.
\begin{enumerate}[leftmargin=*]
\item The irreducible components of $\lfloor B^\ns\rfloor$ are all smooth rational curves; each connected component of $\lfloor B^\ns\rfloor$ intersects $B^\s$ in at most one point and its dual graph is a tree. 
\item $p_g(X, B) = p_g(X, B^\s)$.
\end{enumerate}
 \end{lem}
\begin{proof}
(i)  By the maximality of $B^\s$, any irreducible component subcurve $B_i$ of $\lfloor B\rfloor$ with $p_a(B_i)>0$ is contained in $B^\s$, so the irreducible components of $\lfloor B^\ns\rfloor$ are smooth rational curves. 

Suppose that there is a cycle of smooth rational curves $B'\subset\lfloor B\rfloor$. Then $(K_X+B')\cdot B_i=0$ and thus $B'$ is semistable by Lemma~\ref{lem: ss rat}. By the maximality of $B^\s$, we have $B'\subset B^\s$. It follows that the dual graph of each connected component of  $\lfloor B^\ns\rfloor$ is a tree.

Suppose on the contrary that there is  a connected component $B''$ of $\lfloor B^\ns\rfloor$ such that $B''\cdot B^\s=2$. Then one can find a chain of irreducible components $B_i$ of $B''$ with $1\leq i\leq r$ such that $B_i\cdot B_{i+1}=1$ for $1\leq i\leq r-1$, and $B_1\cdot B^\s$ and $B_r\cdot B^\s$ are both positive. But then $B^\s+\sum_{1\leq i\leq r} B_i$ is semistable, contradicting the maximality of $B^\s$.

(ii) By definition, we have $p_g(X, B) = p_g(X,\lfloor B\rfloor)$. If $B^\s = \lfloor B\rfloor$ then there is nothing to prove. So we assume that $\lfloor B \rfloor -B^\s>0$. Let $E$ be an irreducible component of $\lfloor B \rfloor -B^\s$, which is necessarily a smooth rational curve by (i). 

Consider the exact sequence of cohomology groups
\[
0\rightarrow H^0(X, K_X+ \lfloor B \rfloor-E)\rightarrow H^0(X, K_X+ \lfloor B \rfloor) \rightarrow H^0(E,(K_X+ \lfloor B \rfloor)|_E) 
\]
Since $$\deg (K_X+ \lfloor B \rfloor)|_E =  (K_X+ \lfloor B \rfloor)\cdot E = -2 + (\lfloor B \rfloor -E)\cdot E <0,$$
we have $H^0(E,(K_X+ \lfloor B \rfloor)|_E)=0$ and hence $$H^0(X, K_X+ \lfloor B \rfloor-E) \cong H^0(X, K_X+ \lfloor B \rfloor).$$ Since $\lfloor B \rfloor-E$ and $\lfloor B \rfloor$ have the same semistable part, namely $B^\s$, we finish the proof by induction on the number of irreducible components of $\lfloor B \rfloor -B^\s$.
\end{proof}

Now we describe the behavior of the semistable decomposition of the boundary divisor under blow-ups. 
\begin{lem}\label{lem: bir ss}
Let $(X, B_{X})$ be a smooth projective log canonical surface, and 
\[
\rho\colon U=X_n \xrightarrow{\rho_n} X_{n-1} \xrightarrow{\rho_{n-1}}\cdots \xrightarrow{\rho_2} X_1\xrightarrow{\rho_1} X_0 = X
\]
is a sequence of blow-ups. For each $1\leq i\leq n$,  let $B_{X_i}$ be the $\RR$-divisor on $X_i$ such that $(X_i, B_{X_i})$ is a crepant higher model of $(X, B_Z)$. Let 
\[
B_{X_i}^{>0} = B_{X_i}^{\s} + B_{X_i}^{\ns}
\]
be the decomposition of $B_{X_i}^{>0}$ into the semistable part and the non-semistable part.  Then the following holds.
\begin{enumerate}[leftmargin=*]
\item $\supp(B_{U}^{\s})\cap\, \supp(B_{U}^{<0}) = \emptyset$.
\item $\rho_* B_{U}^{\s} = B_{X}^{\s}, \hspace{.2cm} \rho_* B_{U}^{\ns} = B_{X}^{\ns}$.
\item For $1\leq i\leq n$, let $p_i\in X_{i-1}$ be the point blown up by $\rho_i$, and $E_i\subset X_i$ the exceptional $(-1)$-curve over $p_i$.  Let $\Lambda:=\{i\mid 0\leq i\leq n, p_i\notin \supp(B_{X_{i-1}}^\s)\}$, and $\sE_i$ is the total transform of $E_i\subset X_i$ on $U$. Then
\[
\disp K_{U}+B_{U}^{\s}  = \rho^*(K_{X}+B_{X}^\s) + \sum_{i\in \Lambda} \sE_i.
\] 
\item $\supp(B_{U}^{\s}) \cap\, \supp(B_{U}^{\ns}) \neq \emptyset$ if and only if $\supp(B_{X}^{\s}) \cap\, \supp(B_{X}^{\ns}) \neq \emptyset$.
\end{enumerate}
\end{lem}
\begin{proof}

We  show by induction on $i$ that the following statements $\text{(a)}_i-\text{(d)}_i$ hold, from which the lemma follow immediately by induction:
\begin{enumerate}[leftmargin=*]
\item[$\text{(a)}_i$] $\supp(B_{X_i}^{\s})\cap\, \supp(B_{X_i}^{<0}) = \emptyset$;
\item[$\text{(b)}_i$] $\rho_{i*} B_{X_i}^{\s} = B_{X_{i-1}}^{\s}, \hspace{.1cm} \rho_{i*} B_{X_i}^{\ns} = B_{X_{i-1}}^{\ns}$, \hspace{.1cm} $\rho_{i*} B_{X_i}^{<0} = B_{X_{i-1}}^{<0}$;
\item[$\text{(c)}_i$]
$K_{X_i}+B_{X_i}^{\s} = 
\begin{cases}
\rho_i^*(K_{X_{i-1}}+B_{X_{i-1}}^\s)   &\text{if $p_{i}\in B_{X_{i-1}}^\s$}, \\
\rho_i^*(K_{X_{i-1}}+B_{X_{i-1}}^\s) + E_i &\text{if $p_i\notin B_{X_{i-1}}^\s$};
\end{cases}
$
\item[$\text{(d)}_i$]  $\supp(B_{X_i}^{\s}) \cap\, \supp(B_{X_i}^{\ns}) \neq \emptyset$ if and only if $\supp(B_{X_{i-1}}^{\s}) \cap\, \supp(B_{X_{i-1}}^{\ns}) \neq \emptyset$.
\end{enumerate}

Note that $\supp(B_{X_0}^{\s})\cap\, \supp(B_{X_0}^{<0}) = \emptyset$ is trivially true, since $(X_0, B_{X_0})$ is log canonical and hence $B_{X_0}^{<0}=0$. For $i\geq 1$, we can assume by induction that $\supp(B_{X_{i-1}}^{\s})\cap\, \supp(B_{X_{i-1}}^{<0}) = \emptyset$.

We distinguish two cases depending on whether the blown-up point $p_i\in X_{i-1}$ lies on $\supp(B_{X_{i-1}}^\s)$ or not.

\medskip

\noindent{\bf Case 1.} Suppose that $p_i\in \supp(B_{X_{i-1}}^\s)$. Then $p_i\notin \supp(B_{X_{i-1}}^{<0})$ by the induction hypothesis. 
%Since $(X_{i-1}, B_{X_{i-1}})$ is sub-log canonical, we have $\mult_{p_i} B_{X_{i-1}} \leq 2$ and hence $p\notin \supp(B_{X_{i-1}}^{\ns})$.  
Now we can check directly that
\begin{equation}
1\geq \mult_{E_i} B_{X_i} = \mult_{p_i} B_{X_{i-1}} -1 =\mult_{p_i} B_{X_{i-1}}^\s + \mult_{p_i} B_{X_{i-1}}^\ns -1 \geq 0, 
\end{equation}
where the first inequality is because $(X_i, B_{X_i} )$ is sub-log canonical.  If $p_i\in \supp(B_{X_{i-1}}^\s)$ is a node, then $\mult_{p_i} B_{X_{i-1}}^\ns = 0$, $E_i$ intersects the strict transform $\rho^{-1}_*B_{X_{i-1}}^\s$ at two points, and
 \begin{equation}\label{eq: p node1}
B_{X_i}^\s =  \rho_*^{-1}B_{X_{i-1}}^\s + E_i= \rho^*B_{X_{i-1}}^\s - E_i 
\end{equation}
If $p_i\in \supp(B_{X_{i-1}}^\s)$ is a smooth point, then the connected components of $E_i+\lfloor \rho_*^{-1} B_{X_{i-1}}^\ns\rfloor$ are again trees of smooth rational curves (cf.~Lemma~\ref{lem: pg}), each intersecting $\rho^{-1}_*B_{X_{i-1}}^\s$ at \emph{less than} two points.  It follows that $E_i$ is not a component of $B_{X_{i}}^\s$, and we have 
\begin{equation}\label{eq: p node2}
B_{X_i}^\s =\rho_*^{-1}B_{X_{i-1}}^\s =  \rho^*B_{X_{i-1}}^\s - E_i 
\end{equation} 
%Also we have  $K_U+B_U^{\s} =  \rho^*(K_X+B^\s)$. 

\medskip

\noindent{\bf Case 2.} Suppose that $p_i\notin \supp(B_{X_{i-1}}^\s)$.  One checks as before that $E_i$ is not a component of $B_{X_i}^\s$, and
\begin{equation}\label{eq: p not on B}
B_{X_i}^\s =\rho_*^{-1}B_{X_{i-1}}^\s = \rho^*B_{X_{i-1}}^\s.
\end{equation} 

In both cases, we have
\begin{equation}
\supp(B_{X_i}^\s)\cap\, \supp(B_{X_i}^{<0}) \subset \rho^{-1}\left (\supp(B_{X_{i-1}}^\s) \cap\, \supp(B_{X_{i-1}}^{<0}) \right)= \emptyset
\end{equation}
where the equality is because of the induction hypothesis.  Thus $\supp(B_{X_i}^\s)\cap\, \supp(B_{X_i}^{<0}) =\emptyset$, which is $\text{(a)}_i$.  

Also we have $\rho_{i*}B_{X_i}^\s = B_{X_{i-1}}^\s$. It follows directly from the definition that $\rho_{i*}B_{X_i}^{>0} = B_{X_{i-1}}^{>0}$ and $\rho_{i*}B_{X_i}^{<0} = B_{X_{i-1}}^{<0}$, and hence
\[
 \rho_{i*}B_{X_i}^\ns = \rho_{i*}B_{X_i}^{>0}  - \rho_{i*}B_{X_i}^\s  = B_{X_{i-1}}^{>0} - B_{X_{i-1}}^\s = B_{X_{i-1}}^\ns.
\]
This proves $\text{(b)}_i$. 

Combining \eqref{eq: p node1}, \eqref{eq: p node2}, and \eqref{eq: p not on B} with the fact that $K_{X_i} = \rho_i^*K_{X_{i-1}} + E_i$, one obtains $\text{(c)}_i$:
\[
K_{X_i}+B_{X_i}^{\s} = 
\begin{cases}
\rho^*(K_{X_{i-1}}+B_{X_{i-1}}^\s)   &\text{if $p_{i}\in B_{X_{i-1}}^\s$}, \\
\rho^*(K_{X_{i-1}}+B_{X_{i-1}}^\s) + E_i &\text{if $p_i\notin B_{X_{i-1}}^\s$}. 
\end{cases}
\] 

Now we prove $\text{(d)}_i$: Suppose that there is a point $p\in\supp(B_{X_{i-1}}^{\s}) \cap\, \supp(B_{X_{i-1}}^{\ns})$. Then $p\notin \supp(B_{X_{i-1}}^{<0})$ by $\text{(a)}_i$, and $p$ is not a node of $B_{X_{i-1}}^\s$. If $p=p_i$ is blown up by $\rho_i$, then 
\[
\mult_{E_i} B_{X_i} = \mult_{p} B_{X_i}^\s + \mult_{p} B_{X_i}^\ns -1 = \mult_{p} B_{X_i}^\ns >0.
\]
Hence $\supp(B_{X_i}^\ns) = \rho_i^*\supp(B_{X_{i-1}}^\ns)$ and $B_{X_i}^\s = \rho_*^{-1}B_{X_{i-1}}^\s$, which have nonempty intersection.
For the other implication, suppose on the contrary that there is a point $q\in\supp(B_{X_i}^{\s}) \cap\, \supp(B_{X_i}^{\ns})$ while $\supp(B_{X_{i-1}}^{\s}) \cap\, \supp(B_{X_{i-1}}^{\ns}) =\emptyset$. Then the connected component $G$ of $\supp(B_{X_i}^{\ns})$ containing $q$ is necessarily contracted by $\rho_i$. This can only happen if $G=E_i$, $p_i=\rho_i(q)\not\in\supp(B_{X_{i-1}}^{\ns})$, and $p_i\in \supp(B_{X_{_i}}^{\s})$ is smooth point. But then, as the computation in Case 1 above shows, $\mult_{E_i} B_{X_i} =0$ and hence $q\notin \supp(B_{X_i}^{\ns})$, which is a contradiction.
\medskip
\end{proof}
\begin{rmk}
By Lemma~\ref{lem: bir ss}, given a projective log canonical surface $(X, B_X)$ and an effective higher model $(U, B_U, \rho)$, we have $\rho_*B_U^\s =B^\s$ and $\rho_*B_U^\ns =B^\ns$. Thus the system of $\RR$-divisors $\{B_U^\s\}_U$ (resp.~$\{B_U^\ns\}_U$) form a so-called $\rmb$-$\RR$-divisor.
\end{rmk}

\begin{cor}\label{cor: bir ss}
Let $(X, B_X)$ be a smooth projective log canonical surface, and $\rho\colon (U, B_U)\rightarrow (X, B_X)$ a crepant resolution. Let $B_U^\s$ and $B_X^\s$ be semistable parts of $B_U^{>0}$ and $B_X$ respectively.  Then, for any $m\in \ZZ_{\geq 0}$, we have $\dim H^0(U, m(K_U+B_U^\s)) = \dim H^0(X, m(K_X+B_X^\s))$. In particular, $p_g(U, B_U^\s) = p_g(X, B_X^\s)$ and $\kappa(K_U+B_U^\s)= \kappa(K_U+B_U^\s)$.
\end{cor}
\begin{proof}
By Lemma~\ref{lem: bir ss} (ii) and (iii), we have $\rho_*(K_U+B_U^\s) = K_X+B_X^\s$, and $K_U+B_U^\s- \rho^*(K_X+B_X^\s)$ is effective $\RR$-divisor supported on $\Exc(\rho)$. Thus there is a natural map $ \rho^*\colon H^0(X, m(K_X+B_X^\s)) \rightarrow H^0(U, m(K_U+B_U^\s))$ which is easily seen to be an isomorphism for any positive integer $m$. 
\end{proof}

\section{Necessary conditions for realizing the minimal volumes}\label{sec: nec cond}
In this section we pin down several necessary conditions  a projective log canonical surface with prescribed geometric genus should satisfy, in order that it achieves the minimal possible volume. 

\subsection{Decreasing the volume inside $\sS(\sC, p_g; \kappa)$}\label{sec: decrease}
In this subsection we present several conditions, which forbid a log surface $(X, B)\in \sS(\sC, p_g; \kappa)$ to realize the minimal volume. This is shown by constructing another log surface $(X', B')\in \sS(\sC, p_g; \kappa)$ with smaller volume.

We need to set up the notation before giving the exact statements.
\begin{nota}\label{nota: triangle}
 For a projective log canonical surface $(X, B)$, we consider the following diagram:
\begin{equation}\label{eq: triangle}
\begin{split}
\begin{tikzpicture}
\node (tX) at (0,0) {$({\tilde X}, B_{\tilde X})$};
\node (X)[below left=1cm of  tX] {$(X, B)$};
\node (Z)[below right=1cm of  tX] {$(Z,B_Z)$};
\draw[->](tX.south west)--(X) node[above, midway]{$\pi$};
\draw[->](tX.south east)--(Z)node[above, midway]{$\rho$};
\end{tikzpicture} 
\end{split}
\end{equation}
where $\pi\colon (\tilde X, B_{\tilde X} )\rightarrow (X, B)$ is the minimal resolution, and $\rho\colon \tilde X\rightarrow Z$ is a minimal model program of the divisor $K_{\tilde X} + B_{\tilde X}^\s$, contracting successively the $(-1)$-curves that are disjoint from (the images of) $B_{\tilde X}^\s$, the semistable part of $B_{\tilde X}$.  Set $B_Z=\rho_*B_{\tilde X}$ and $B_Z'=\rho_* B_{\tilde X}^\s$. Then the birational morphism $\rho$ is an isomorphism over a neighborhood of $\supp(B_Z')$, but $(Z, B_Z)$ may have worse than log canonical singularities at $\supp(B_Z)\setminus\supp(B_Z')$. Therefore, we have $(\tilde X, B_{\tilde X}, \rho)\in \sS(Z, B_Z; B_Z')$, as in Notation~\ref{nota: SZBZ}.  The following relations are clear by Lemma~\ref{lem: pg}
\begin{equation}
p_g(X, B) = p_g(\tilde X, B_{\tilde X}^\s) = p_g(Z, B_Z') \text{ and }\kappa(K_{\tilde X} + B_{\tilde X}^\s) = \kappa(K_Z+B_Z').
\end{equation}
\begin{nota}\label{nota: fix pg k}
Let $\sC\subset(0,1]$ be a subset, $p_g$ a non-negative integer and $\kappa\in\{-\infty, 0,1,2\}$. Recall from the introduction that $\sS(\sC)$ consists of projective log canonical surfaces $(X, B)$ of general type with $\sC_B\subset \sC$, and $\sS(\sC, p_g)$ consists of those $(X, B) \in  \sS(\sC)$ with $p_g(X, B)=p_g$. Now for $\kappa\in\{-\infty, 0,1,2\}$, define
\[
\sS(\sC, p_g;\kappa) = \{(X, B) \in \sS(\sC, p_g) \mid \kappa(K_{\tilde X}+B_{\tilde X}^\s)=\kappa\}, 
\] 
where $(\tilde X, B_{\tilde X})\rightarrow (X, B)$ is the minimal resolution.  Correspondingly, we denote the set of volumes of log surfaces from  $\sS(\sC, p_g;\kappa)$ by
\[
\KK^2(\sC, p_g;\kappa):=\{\vol(K_X+B) \mid (X, B)\in\sS(\sC, p_g;\kappa)\}
\]
\end{nota}

\begin{rmk}
We have two easy observations:
\begin{enumerate}[leftmargin=*]
\item $\sS(\sC, p_g;\kappa)=\emptyset$ if $p_g>0$ and if $\kappa=-\infty$.
\item If $\sC'\subset\sC\subset(0,1]$ then $\sS(\sC', p_g; \kappa) \subset \sS(\sC, p_g;\kappa)$.
\end{enumerate}
\end{rmk}
\end{nota}

Now we can give the first criteria for non-minimal volumes.
\begin{lem}\label{lem: smaller1}
Let $\sC\subset(0,1]$ be a subset, $p_g$ a non-negative integer, and $\kappa\in\{-\infty, 0,1,2\}$. Suppose that $(X, B)\in\sS(\sC, p_g;\kappa)$ has ample $K_X+B$ and let $\pi\colon (\tilde X, B_{\tilde X})\rightarrow (X, B)$ be the minimal resolution. If one of the following conditions holds:
\begin{enumerate}[leftmargin=*]
\item[(a)] $\supp(B_{\tilde X}^\s)\cap\, \supp(B_{\tilde X}^\ns)\neq \emptyset$,
\item[(b)] $\supp(B_{\tilde X}^\ns)$ contains a nonklt center of $(X, B)$,
\end{enumerate} 
then there is another projective log canonical  surface $(X', B') \in\sS(\sC_B, p_g;\kappa)$ such that $\vol(K_{X'}+B') <\vol(K_X+B)$.
\end{lem}
\begin{proof}
The construction of the log surface $(X', B')$ runs along the same line as \cite[Theorem~3.3]{AL19b}, but here one needs to check additionally that the newly constructed $(X', B')$ has the same $p_g$ and $\kappa$ as $(X,B)$. 
 
 Let $\alpha\colon (U, B_U)\rightarrow (X, B)$ be an effective crepant resolution. By possibly blowing up points on $\lfloor B_U\rfloor$ further, we may assume that $\supp(B_U)$ has normal crossing in a neighborhood of $\lfloor B_U\rfloor$.  Since $\pi\colon \tilde X\rightarrow X$ is the minimal resolution, there is a birational morphism $\tilde \alpha\colon U\rightarrow \tilde X$ such that $\pi\circ\tilde \alpha = \alpha$. By Lemma~\ref{lem: bir ss}, the condition (a) (resp.~(b)) still holds with $(\tilde X, B_{\tilde X})$ replaced by $(U, B_U)$; moreover, under the condition (b),  one can require that $\lfloor B_{U}^\ns\rfloor \neq 0$ by blowing up $U$ further, if needed. In both cases, there is an irreducible component $B_{U,2}$ of $\lfloor B_{U}\rfloor$, intersecting some other irreducible component $B_{U,1}$ that is contained in $\supp(B_{U}^{\ns})$.  In particular, $B_{U,2}$ is an accessible nonklt center of $(U, B_U)$. 
 
 Take a point $p\in B_{U,2} \cap\, B_{U,1} $. Blow up $p\in U$ and then its pre-images on the strict transforms of $B_{U,2}$ on the blown-up surfaces. Let $\beta\colon  (W, B_W)\rightarrow (U, B_U)$ be the resulting crepant effective resolution after $n$ such blow-ups, shown by dual graphs as follows:
 \begin{center}
  \begin{tikzpicture}[font=\tiny]
    \begin{scope}
\node [inner sep=0pt, label = below:$\beta^{-1}_*B_{U,1}$](BU1) at (0,0) {$\otimes$};
\node[wbullet, label = below:$E_1$,  label = above:$(-2)$] (E1) at (1,0) {};
\node[wbullet, label = below:$E_{2}$,  label = above:$(-2)$] (E2) at (2,0) {};
\node[wbullet, label = below:$E_{n-1}$,  label = above:$(-2)$] (En-1) at (4,0) {};
\node[bbullet, label = below:$E_n$,  label = above:$(-1)$] (En) at (5,0) {};
\node [inner sep=0pt, label = below:$\beta^{-1}_*B_{U,2}$](BU2) at (6,0) {$\otimes$};
\draw (BU1)--(E1)--(E2);
      \draw[dashed] (E2)--(En-1);
      \draw (En-1)--(En)--(BU2);
      \draw[->] (6.5,0)--node[above]{$\beta$}(7.5,0);
          \end{scope}
          \begin{scope}[xshift=8cm]
          \node [inner sep=0pt, label = below:$B_{U,1}$](BU1) at (0,0) {$\otimes$};
          \node [inner sep=0pt, label = below:$B_{U,2}$](BU2) at (1,0) {$\otimes$};
          \draw (BU1)--(BU2);
          \end{scope}
  \end{tikzpicture}
  \end{center}
where the $E_i$ ($1\leq i\leq n$) denote the exceptional curves of $\beta$, and the numbers in the brackets above the vertices are the self-intersection numbers of the corresponding curves.  By the construction of $\beta$, one sees easily that $ B_{W}=\beta^{-1}_* B_{U}+ a \sum_{1\leq i\leq n} E_i$, where $a\in (0,1]$ is the coefficient of $B_{U,1}$ in $B_U$.  Decreasing the coefficients of the exceptional divisors $E_i$ in $B_W$, we set 
\begin{equation}\label{eq: BU'}
B_W'=B_{W} - a\sum_{1\leq i\leq n} \frac{i}{n} E_i =  \beta^{-1}_* B_{U}+ a\sum_{1\leq i\leq n-1}\left(1 -  \frac{i}{n}\right)E_i. 
\end{equation}
Since $K_{W}+B_{W} = \beta^*\alpha^*(K_X+B)$ is big and nef, and $B_W'<B_{W} $, we have by Lemma~\ref{lem: vol ZD}
\begin{equation}\label{eq: smaller}
\vol (K_{W}+B_W') < \vol(K_{W}+B_{W}) = \vol(K_X+B)
\end{equation}
Taking $n$ large enough, for example, $n>\frac{a^2}{\vol(K_X+B)}$, we have
\begin{equation}\label{eq: >0}
\begin{aligned}
\left(K_{W}+B_W'\right)^2&= \left(K_{W}+B_{W} - a\sum_{1\leq i\leq n} \frac{i}{n} E_i\right)^2\\
&=\left(K_{W}+B_{W}\right) ^2 + a^2\left(\sum_{1\leq i\leq n} \frac{i}{n} E_i\right)^2\\
&=\vol(K_X+B) -\frac{a^2}{n} >0
\end{aligned}
\end{equation}
Together with the fact that $\beta_*(K_{W}+B_W') = K_{U}+B_{U}$ is big,  \eqref{eq: >0} implies that $K_{W}+B_W'$ is big for $n>\frac{a^2}{\vol(K_X+B)}$. 

Let $\pi'\colon (W, B_W') \rightarrow (X', B')$ be the contraction onto  the ample model.
By \eqref{eq: smaller}, we have
\[
0< \vol(K_{X'}+B') = \vol (K_W+B_W') <  \vol(K_X+B).
\]

Next we check that $(X', B')\in\sS(\sC, p_g;\kappa)$. 

Write $\gamma=\alpha\circ \beta\colon W\rightarrow X$. As one checks above, $(K_W+B_W') \cdot E_i =0$ holds for $1\leq i\leq n-1$, while for any $\gamma$-exceptional curve $F$ other than the $E_i, 1\leq i\leq n$, we have $(K_W+B_W') \cdot  F\leq 0$.
It follows that all $\gamma$-exceptional curves except for $E_n$ are contracted to points on the ample model $(X', B')$. Note also that the coefficient of $E_n$ in $B_W'$ is $0$. Therefore, the coefficients of $B'$ are among those of $B$.

The semistable parts of $B_{W}$ and $B_W'$ are the same, which is denoted by $B_{W}^\s$. It follows from Lemma~\ref{lem: pg} and Corollary~\ref{cor: bir ss} that $$p_g(X',B') = p_g(W, B_W') =p_g(W, B_{W}^{s}) =p_g(U, B_{U}^\s)=p_g$$ and $$ \kappa(K_{W}+ B_{W}^{s})=\kappa(K_{U}+B_{U}^\s)=\kappa.$$ 

In conclusion, we have constructed under one of the conditions (a) and (b) a projective log canonical surface $(X', B')\in\sS(\sC_B, p_g;\kappa)$ and it has a smaller volume than $(X,B)$.
\end{proof}

\begin{lem}\label{lem: smaller2}
Let $\sC\subset(0,1]$ be a subset such that $\min(\sC\cup\{1\})$ attains the minimum, say $c$. Then, for $p_g\in \ZZ_{\geq 0}$ and $\kappa\in\{0,1,2\}$, the following holds.
\begin{enumerate}
\item Suppose that $(X, B)\in \sS(\sC, p_g; \kappa)$ has ample $K_X+B$, and let $\pi\colon(\tilde X, B_{\tilde X})\rightarrow (X, B)$ be the minimal resolution. If one of the following conditions (a) and (b) holds:
\begin{enumerate}
\item[(a)] $\max \sC_{B_{\tilde X}^\ns} > c$;
\item[(b)] there is an irreducible component $\tilde B_0$ of $B_{\tilde X}^\ns$, say, with coefficient $b_0>0$, such that $K_{\tilde X}+B_{\tilde X}-b_0\tilde B_0$ is still big;
\end{enumerate}
then there is another log surface $(X', B')\in \sS(\{c,1\}\cap \sC_B, p_g; \kappa)$ such that $K_{X'}+B'$ is ample and $\vol(K_{X'}+B')<\vol (K_X+B_X)$:
\item The volume set $\KK^2(\sC, p_g; \kappa)$ attains the minimum, and we have $\min\KK^2(\sC, p_g; \kappa) =\min \KK^2(\{c,1\}\cap\sC, p_g; \kappa)$.
\end{enumerate}
\end{lem}
\begin{proof}
(i) Under either condition, we will define a boundary divisor $\tilde B$ strictly smaller than $B_{\tilde X}$ and then take $(X', B')$ to be ample model of $(\tilde X, \tilde B)$.

We first deal with the condition (a). Write $B_{\tilde X} = \sum_{j\in J} b_j \tilde B_j$ as the sum of distinct prime divisors $\tilde B_j$, and let $J^\ns\subset J$ be the index set for the components $\tilde B_j$ contained in $B_{\tilde X}^\ns$. Now define a new boundary divisor on $\tilde X$ as follows:
\[
\tilde B := B_{\tilde X}^\s + \sum_{j\in J^\ns} \min\{b_j, c\} \tilde B_j.
\]
Since $\kappa(K_{\tilde X}+B_{\tilde X}^\s)\geq 0$,  $K_{\tilde X} + \tilde B$ is still big by Lemma~\ref{lem: vol ZD} (iii). Let $\pi'\colon (\tilde X, \tilde B)\rightarrow (X', B')$ be the ample model. Under the condition (a), $K_{\tilde X} +\tilde B$ is strictly smaller than the big and nef $K_{\tilde X}+B_{\tilde X}$, and thus $\vol(K_{X'} + B')  = \vol(K_{\tilde X} +\tilde B)< \vol(K_{\tilde X}+B_{\tilde X})=\vol(K_X+B)$ by Lemma~\ref{lem: vol ZD} (ii). By construction we have $\sC_{\tilde B}\subset(\sC_{B_{\tilde X}})_{<c}\cup \{c, 1\}$. The components $\tilde B_j, j\in J$ contracted by contracted by $\pi\colon \tilde X\rightarrow X$ are  necessarily contracted by $\pi'$, since $(K_{\tilde X} + \tilde B)\cdot \tilde B_j \leq (K_{\tilde X} + B_{\tilde X})\cdot \tilde B_j  = 0$. Now the components $\tilde B_j$ in $B_{\tilde X}^\ns$ with $b_j<c$ are contracted by $\pi$. Therefore,  $\sC_{B'} \subset  \sC_B \cap \{c,1\}$. Since the semistable part of $\tilde B$ and $B_{\tilde X}$ are the same, namely $B_{\tilde X}^\s$, we  infer that $(X', B')\in\sS(\{c,1\}\cap \sC_B, p_g;\kappa)$. 

Under the condition (b), we define $\tilde B=B_{\tilde X}-b_0\tilde B_0$. The arguments are similar to (and easier than) those under condition (a).

(ii) we have $\inf \KK^2(\sC, p_g; \kappa) \geq \inf\KK^2(\{c,1\}\cap\sC, p_g; \kappa)$ by (i). On the other hand, the inclusion $\KK^2(\{c,1\}\cap\sC, p_g; \kappa) \subset \KK^2(\sC, p_g; \kappa)$ implies that $\inf \KK^2(\sC, p_g; \kappa) \leq \inf\KK^2(\{c,1\}\cap\sC, p_g; \kappa)$. Therefore, $\inf \KK^2(\sC, p_g; \kappa) = \inf\KK^2(\{c,1\}\cap\sC, p_g; \kappa)$. By \cite{Ale94}, the latter attains the minimum. It follows that $\KK^2(\sC, p_g; \kappa)$ also attains the \emph{same} minimum.
\end{proof}

\subsection{Necessary conditions for realizing the minimal volumes}
Based on the criteria for non-minimal volumes in Section~\ref{sec: decrease}, we can now formulate several necessary conditions for realizing the minimal volumes inside $\sS(\sC, p_g; \kappa)$ with $\kappa\geq 0$. A notable one is the separation of the semistable part of the boundary divisor from the non-semistable part. This enables us, in searching for the minimal volume, to operate as if we were always on a surface with non-negative Kodaira dimension. 

\begin{prop}\label{prop: nec cond}
Let $\sC\subset(0,1]$ be a subset, possibly empty, such that $\sC\,\cup\,\{1\}$ attains the minimum, say $c$. Let $p_g\in\ZZ_{\geq 0}$, $\kappa\in\{0,1,2\}$ be two nonnegative integers. Suppose that $(X, B)\in\sS(\sC, p_g;\kappa)$ has ample $K_X+B$ and $\vol(K_X+B)=\min \KK^2(\sC, p_g;\kappa)$, and let $\pi\colon(\tilde X, B_{\tilde X})\rightarrow (X, B)$ be the minimal resolution. Then the following holds.
\begin{enumerate}[leftmargin=*]
\item $(\tilde X, B_{\tilde X})$ has klt singularities in a neighborhood of $\supp(B_{\tilde X}^{\ns})\subset \tilde X$. In particular, $\supp(B_{\tilde X}^\s)$ and $\supp(B_{\tilde X}^\ns)$ are disjoint. 
\item There is an inclusion $\sC_{B_{\tilde X}^\ns}\subset(0, c]_{<1}$. 
\item If $\tilde B_0$ is irreducible component  of $B_{\tilde X}^\ns$, say, with coefficient $b_0>0$, then $K_{\tilde X}+B_{\tilde X}-b_0\tilde B_0$ is not big.
\item There is an inclusion $\sC_B \subset\{c, 1\}\cap \sC$.
\end{enumerate} 
\end{prop}
\begin{proof}
\noindent (i) Suppose on the contrary that $(\tilde X, B_{\tilde X})$ is not klt along $\supp(B_{\tilde X}^{\ns})$. Then one of the conditions (a) and (b) of Lemma~\ref{lem: smaller1} holds, and this would violate the minimality of $\vol(K_X+B)$.

\medskip

\noindent(ii) we have $\sC_{B_{\tilde X}^{\ns}}\subset(0,1)$ by (i) and $\sC_{B_{\tilde X}^{\ns}}\subset (0,c]$ by Lemma~\ref{lem: smaller2} (ia).  (ii) is proved by combining these two inclusions.

\medskip

\noindent(iii) is a direct consequence of  Lemma~\ref{lem: smaller2} (ib).

\medskip

\noindent(iv) We have $\sC_{B} \subset ((\sC_{B_{\tilde X}})_{\leq c} \cup \{1\})\cap\, \sC \subset \{c, 1\}\cap\, \sC$, where the first inclusion is by (ii), and the second inclusion is because $c=\min (\sC\cup\{1\})$. 
\end{proof}

\begin{cor}\label{cor: pro min Z}
Keep the notation and assumptions of Proposition~\ref{prop: nec cond}. Let $\rho\colon  (\tilde X, B_{\tilde X}^\s)\rightarrow (Z, B_Z')$ be the $(K_{\tilde X}+B_{\tilde X}^\s)$-minimal model as in \eqref{eq: triangle}, with
\[
B_Z'=\rho_* B_{\tilde X}^\s, \, B_Z''=\rho_* B_{\tilde X}^\ns, \, \text{and } \, B_Z=\rho_* B_{\tilde X}
\]
Then the following holds.
\begin{enumerate}[leftmargin=*]
\item $(Z, B_Z')$ has log canonical singularities, and $\supp(B_Z')\cap\, \supp(B_Z'')=\emptyset. $
\item There is an inclusion $\sC_{B_Z''}\subset(0, c]_{<1}$. 
\item If $B_{Z, 0}$ is an irreducible component of $B_{Z}^\ns$, say, with coefficient $b_0>0$, then $K_{Z}+B_{Z}-b_0B_{Z, 0}$ is not big.
\item $B_{\tilde X}^\ns$ is the strict transform of $B_Z''$.
\end{enumerate} 
\end{cor}
\begin{proof}
(i) is clear from the corresponding property of $(\tilde X, B_{\tilde X}^\s + B_{\tilde X}^\ns)$ as in Proposition~\ref{prop: nec cond} and from the fact that $\rho\colon \tilde X\rightarrow Z$ is an isomorphism over a neighborhood of $B_Z'$. 

(ii) follows from Proposition~\ref{prop: nec cond} (ii), since $\sC_{B_Z''}\subset \sC_{B_{\tilde X}^\ns}$.

(iii) follows from Proposition~\ref{prop: nec cond} (iii), since the bigness of $K_{Z}+B_{Z}-b_0B_{Z, 0}$ would imply that of $K_{\tilde X}+B_{\tilde X}-b_0\tilde B_0$ by Lemma~\ref{lem: min fix surf}, where $\tilde B_0\subset\tilde X$ is the strict transform of $B_{Z, 0}$.

(iv) As noticed in Notation~\ref{eq: triangle}, we have $(\tilde X, B_{\tilde X}, \rho)\in \sS(Z, B_Z; B_Z')$. By Lemma~\ref{lem: min fix surf}, $K_{\tilde X} + B_{\tilde X}^\s + \rho_*^{-1}B_Z''$ is big. By Proposition~\ref{prop: nec cond}, we infer that $B_{\tilde X}^\ns = \rho_*^{-1}B_Z''$.
\end{proof}

\section{Lower bounds of $\KK^2(\sC, p_g)$}\label{sec: lower bound}

%In search for the minimal volume, we reverse the procedure of  \eqref{eq: triangle}, starting with the log surface $(Z, B_Z)$ instead of $(X, B)$ and then trying to find conditions for $(Z, B_Z)$, in order to minimize $\vol(K_X+B)$. 

In this section we fix a (possibly empty) coefficient set $\sC\subset(0,1]$ such that $\sC\cup\{1\}$ attains the minimum, say $c$, and a nonnegative integer $p_g$. Our aim is to uniformly bound  $\vol(K_X+B)$ for $(X, B)\in \sS(\sC, p_g)$ from below. Note that, by taking the ample model of $(X, B)$, we can assume that $K_X+B$ is ample. If $p_g>0$, the set $\sS^2(\sC, p_g)$ can be decomposed as follows (see Notation~\ref{nota: fix pg k}):
\[
\sS^2(\sC, p_g) = \bigcup_{0\leq \kappa \leq 2} \sS^2(\sC, p_g; \kappa).
\]
Thus it suffices to give, separately for $0\leq \kappa \leq 2$, the lower bounds of $\vol(K_X+B)$ for $(X, B)\in \sS^2(\sC, p_g; \kappa)$. Most of the lower bounds found turn out to be optimal, as the examples   in Section~\ref{sec: ex} show.

\subsection{The lower bound of $\KK^2(\sC, p_g; 2)$}
\begin{prop}\label{prop: 2}
For any $(X, B)\in \sS(\sC, p_g; 2)$, we have $\vol(K_X+B)\geq \max\{1, p_g-2\}$.
\end{prop}
\begin{proof}
Suppose that $(X, B)\in \sS(\sC, p_g; 2)$ has ample $K_X+B$ and $\vol(K_X+B) = \min\KK^2(\sC, p_g; 2)$. It suffices to show the inequality of the lemma for this log surface.

Note that $(X, B)$ and its minimal resolution $\pi\colon (\tilde X, B_{\tilde X})\rightarrow (X,B)$ are subject to the restrictions given by Proposition~\ref{prop: nec cond}. Since $(X, B)\in\sS(\sC, p_g; 2)$, we have $\kappa(K_{\tilde X}+B_{\tilde X}^\s)=2$, where $B_{\tilde X}^\s$ is the semistable part of $B_{\tilde X}$. It follows from Proposition~\ref{prop: nec cond} (iii) that $B_{\tilde X}^\s = B_{\tilde X}$.  Then $B=\pi_*B_{\tilde X}$ is reduced and $K_X+B$ is Cartier. Thus $\vol(K_X+B)=(K_X+B)^2\geq 1$ holds. 

Next we show that $\vol(K_X+B)\geq p_g -2$. For this purpose, we may assume that $p_g\geq 3$. Let $\phi\colon X\dashrightarrow \PP^n$ with $n=p_g-1$ be the canonical map defined by the linear system $|K_X+B|$. Let $|M|$ be the movable part of $|K_S+B|$, so $M$ is nef and $K_X+B=M+N$, where $N$ is the fixed part of $|K_X+B|$. If $\phi(X)$ is a surface, then we have
\[
\vol(K_X+B)=(K_X+B)^2 \geq M^2 \geq \deg\phi(X) \geq p_g-2.
\] 
where the last inequality is by \cite[Proposition~0]{EisenbudHarris87}.
If $\phi(X)$ is a curve, then we have $M\equiv mF$, where $F$ is a general element of the pencil, and $K_X+B=M+N=mF+N$. Also $m=\deg \phi(X)\geq p_g-1$ by \cite[Proposition~0]{EisenbudHarris87}. Thus we have
\[
\vol(K_X+B) = (K_X+B)^2 = (mF+N)\cdot (mF+N) \geq m(K_X+B)\cdot F \geq p_g -1. 
\]
In conclusion, we have shown that $\vol(K_X+B)\geq \max\{1, p_g-2\}$.  
\end{proof}

\subsection{The lower bound of $\KK^2(\sC, p_g; 1)$}
\begin{prop}\label{prop: 1}
For any $(X, B)\in \sS(\sC, p_g; 1)$, the following holds.
\begin{enumerate}[leftmargin=*]
\item
If $p_g\geq 2$, then
\[
\vol(K_X+B)\geq 
\begin{cases}  
(2c-c^2)(p_g-1)-2c^2 & \text{ if } c < \frac{p_g-1}{p_g+1}\\
 p_g-3 + \frac{4}{p_g+1} & \text{ if }  c\geq  \frac{p_g-1}{p_g+1} 
\end{cases}  
\]
\item If $p_g \leq 1$, then
\[
\vol(K_X+B)\geq
\begin{cases}
2c - 3c^2 & \text{if $c\leq \frac{1}{3}$} \\
\frac{1}{3} & \text{if $c>  \frac{1}{3}$} 
\end{cases}
\]
\end{enumerate}
\end{prop} 
\begin{proof}
{\noindent \bf Step 0.} Suppose that $(X, B)\in \sS(\sC, p_g; 1)$ has ample $K_X+B$ and $\vol(K_X+B) = \min\KK^2(\sC, p_g; 1)$. It suffices to show the inequalities of the proposition for this log surface. Then $(X, B)$ and the minimal resolution $\pi\colon(\tilde X, B_{\tilde X})\rightarrow (X, B)$ are subject to the restrictions of Proposition~\ref{prop: nec cond}. Since  $(X, B)\in \sS(\sC, p_g; 1)$, we have $\kappa(K_{\tilde X}+B_{\tilde X}^\s)=1$. For $l$ sufficiently large and divisible, the linear system $|l(K_{\tilde X}+B_{\tilde X}^\s)|$ defines a fibration $\tilde f\colon \tilde X\rightarrow C$. For a general fiber $\tilde F$ of $\tilde f$, we have $(K_{\tilde X}+B_{\tilde X}^\s)\cdot \tilde F=0$ and $g(\tilde F)\leq 1$. 

Since the divisor $K_{\tilde X}+B_{\tilde X}$ is big, there is a component $\tilde B_0$ of $B_{\tilde X}^\ns$, horizontal with respect to the fibration $f$. Let $b>0$ be the coefficient of $\tilde B_0$ in $B_{\tilde X}^\ns$. Then the log canonical divisor $K_{\tilde X}+B_{\tilde X}^\s+b\tilde B_0$ is big by Lemma~\ref{lem: nef+curve}. Due to the minimality of $\vol(K_X+B)$ we infer from Proposition~\ref{prop: nec cond} that \begin{itemize}[leftmargin=*]
\item $B_{\tilde X}^\ns=b\tilde B_0$ with $b\in (0,c]_{<1}$, and
\item $\tilde B_0\cap\, \supp(B_{\tilde X}^\s)=\emptyset$, thus $K_{\tilde X}\cdot\tilde B_0 = (K_{\tilde X} + B_{\tilde X}^\s) \cdot \tilde B_0>0$.
\end{itemize}

\medskip

{\noindent \bf Step 1.} We pass to the $(K_{\tilde X}+B_{\tilde X}^\s)$-minimal model $\rho\colon (\tilde X, B_{\tilde X}^\s)\rightarrow (Z, B_Z')$, as in Notation~\ref{nota: triangle}. Then $\kappa(K_Z + B_Z') =\kappa(K_{\tilde X}+B_{\tilde X}^\s) =1$, and $|l(K_Z+B_Z')|$ is base point free, defining a fibration $f\colon Z\rightarrow C$ so that $\tilde f=f\circ\rho$. The general fiber $F$ of $f$ is isomorphic to that of $\tilde f$. Note also that $\tilde B_0$ is not contracted by $\rho\colon\tilde X\rightarrow Z$, since $\rho$ only contracts curves vertical with respect to $f$; denote $B_{Z,0} = \rho_* \tilde B_0$.

\medskip

{\noindent \bf Step 2.} In this step, we decompose $\rho\colon \tilde X\rightarrow Z$ into a sequence of simple blow-ups, say
\[
\rho\colon \tilde X=Z_n \xrightarrow{\rho_n} Z_{n-1} \xrightarrow{\rho_{n-1}}\cdots \xrightarrow{\rho_{2}} Z_1 \xrightarrow{\rho_1} Z_0=Z
\]
and compute $\vol(K_X+B)$ in terms of these blow-ups. For each $1\leq i\leq n$, let $p_i\in Z_{i-1}$ be the point blown up by $\rho_i$ and $E_i\subset Z_i$ the $\rho_i$-exceptional $(-1)$-curve over $p_i$. Let $\sE_i\subset \tilde X$ be the total transform of $E_i$. Using the projection formula, we have for $1\leq i, j\leq n$
\[
\sE_i\cdot \sE_j = 
\begin{cases}
-1 &\text{if $i= j$}\\
0 & \text{if $i\neq j$}
\end{cases}
\]
and $\sE_i\cdot \rho^* D = 0$ for any $\RR$-divisor $D$ on $Z$. Let $B_{Z_i}$, $B_{Z_i}'$, $B_{Z_i}'' = bB_{Z_i,0}$ be the pushforward of the $\RR$-divisors $B_{\tilde X}$, $B_{\tilde X}^\s$, $B_{\tilde X}^\ns = b\tilde B_0$ to $Z_i$, respectively, and set $\mu_i:=\mult_{p_i} B_{Z_{i-1},0}$. Then 
\[
K_{\tilde X} = \rho^*K_Z+\sum_{1\leq i\leq n} \sE_i, \hspace{.1cm} B_{\tilde X}^\s = \rho^*B_Z',\hspace{.1cm} B_{\tilde X}^\ns =  \rho_*^{-1}B_Z'' = \rho^*B_Z''-\sum_{1\leq i\leq n} b\mu_i\sE_i,
\]
where the last equality is by Corollary~\ref{cor: pro min Z}. Combining the above equalities, we obtain
\[
K_{\tilde X} +B_{\tilde X}  = \rho^*(K_Z+B_Z) - \sum_{1\leq i\leq n}  (b\mu_i-1) \sE_i.
\]
Setting $d= K_{Z} \cdot B_{Z,0}$, we have
\begin{claim}\label{claim: vol k1}
$\disp\vol(K_X+B) = (2b-b^2)d +b^2\left(2p_a(\tilde B_0) -2 \right) +\sum_{1\leq i\leq n} [(b-b^2)\mu_i + (b\mu_i-1)]$.
\end{claim}
\begin{proof}[Proof of Claim~\ref{claim: vol k1}]
we have
\begin{equation}\label{eq: 1}
\begin{split}
(K_{\tilde X} + B_{\tilde X}^\s)\cdot \tilde B_0 &=K_{\tilde X}\cdot \tilde B_0=\left( \rho^*K_{Z}+ \sum_{1\leq i\leq n} \mathcal E_j\right)\cdot\left(\rho^* B_{Z,0}- \sum_{1\leq i\leq n} \mu_i\mathcal E_j\right)\\
&=K_{Z} \cdot B_{Z,0}+\sum_{1\leq i\leq n} \mu_i = d+\sum_i \mu_i
\end{split}
\end{equation}
where the first equality is because $B_{\tilde X}^\s$ and $\tilde B_0$ do not intersect. By the adjunction formula,
\begin{equation}\label{eq: 2}
\tilde B_0^2 = 2p_a(\tilde B_0) -2 - K_{\tilde X}\cdot \tilde B_0 = 2p_a(\tilde B_0) -2 - \left(d+\sum_i \mu_i\right).
\end{equation}
Note also that $K_{\tilde X} + B_{\tilde X}^\s= \rho^*(K_Z+B_Z') + \sum_{1\leq i\leq n} \mathcal E_j$ and hence
\begin{equation}\label{eq: 3}
(K_{\tilde X} + B_{\tilde X}^\s)^2 =\left (\rho^*(K_Z+B_Z') + \sum_{1\leq i\leq n} \mathcal E_j\right)^2 = (K_Z+B_Z')^2 -n =-n.
\end{equation}
Combining \eqref{eq: 1}, \eqref{eq: 2} and \eqref{eq: 3}, we obtain
\begin{align*}
\vol(K_X+B) &=\left (K_{\tilde X} + B_{\tilde X}\right)^2  =\left (K_{\tilde X} + B_{\tilde X}^\s + b\tilde B_0\right)^2 \\
 & =  \left (K_{\tilde X} + B_{\tilde X}^\s \right)^2 + 2b \left(K_{\tilde X} + B_{\tilde X}^\s\right)\cdot\tilde B_0 + b^2 \tilde B_0^2 \\
 & = -n+2b\left(d+\sum_{1\leq i\leq n} \mu_i\right)+ b^2\left(2p_a(\tilde B_0) -2-d-\sum_{1\leq i\leq n} \mu_i\right) \\
 & = (2b-b^2)d +b^2\left(2p_a(\tilde B_0) -2\right) +\sum_{1\leq i\leq n} \left[(b-b^2)\mu_i + (b\mu_i-1)\right].
\end{align*}    
\end{proof}

{\noindent \bf Step 3.} In this step, we get rid of the $\mu_i$'s and $d$ in Claim~\ref{claim: vol k1} by observing the following lower bounds:
\begin{claim}\label{claim: >1}
For each $1\leq i\leq n$, we have $b\mu_i>1$. 
\end{claim}
\begin{proof}[Proof of Claim~\ref{claim: >1}]
In fact, $b\mu_i-1 = (K_{\tilde X}+B_{\tilde X})\cdot \sE_i = (K_X+B)\cdot \pi_*\sE_i >0$, where the last inequality is because $\sE_i$ is not contracted by $\pi$ and $K_X+B$ is ample.
\end{proof}

\begin{claim}\label{claim: d}
We have $d\geq \max\{1, p_g-1\}$
\end{claim}
\begin{proof}[Proof of Claim~\ref{claim: d}]
It is clear that $d = (K_Z+ B_Z')\cdot B_{Z,0} \geq 1$. In case $p_g\geq 2$, $K_Z+ B_Z' - (p_g-1)F$ is numerically equivalent to some effective vertical curve, and hence $d = (K_Z+ B_Z')\cdot B_{Z,0} \geq (p_g-1)F\cdot B_{Z,0} \geq p_g-1$.
\end{proof}

In view of Claim~\ref{claim: vol k1}, the following Claim~\ref{claim: vol k1 ineq} follows from Claims~\ref{claim: >1} and \ref{claim: d}, together with the fact that $b -b^2> 0$. 
\begin{claim}\label{claim: vol k1 ineq}
We have $\vol(K_{\tilde X}+B_{\tilde X}) \geq  (2b-b^2)\max\{1, p_g-1\} + b^2(2p_a(\tilde B_0)-2)$ with equality holds if and only if $d= \max\{1, p_g-1\}$, and $n=0$, that is, $\rho\colon \tilde X\rightarrow Z$ is an isomorphism.
\end{claim}

If $p_a(\tilde B_0)> 0$ then $b=c<1$, and hence $\vol(K_X+B)\geq  (2c-c^2)\max\{1, p_g-1\}$ by Claim~\ref{claim: vol k1 ineq}. In particular, in this case $\vol(K_X+B)$ is strictly larger than the lower bounds given in Proposition~\ref{prop: 1}.

\medskip

{\noindent \bf Step 4.} From now on, we can assume that $p_a(\tilde B)=0$. Under this assumption, we have
\begin{equation}\label{eq: b0}
b = \min\left\{c, \frac{\tilde B_0^2+2}{\tilde B_0^2}\right\} = \min\left\{c, \frac{m+\sum_i \mu_i}{2+m+\sum_i \mu_i}\right\}
\end{equation}
and $b=\frac{d+\sum_i \mu_i}{2+d+\sum_i \mu_i}$ if and only if $\tilde B_0$ is contracted by $\pi\colon\tilde X\rightarrow X$. 
By Claim~\ref{claim: d}, we have $d\geq \max\{1, p_g-1\}$ and hence
\begin{equation}\label{eq: b0'}
 \frac{d+\sum_i \mu_i}{2+d+\sum_i \mu_i} \geq  \frac{d}{2+d}  \geq \frac{\max\{1, p_g-1\}}{\max\{1, p_g-1\}+2}
\end{equation}
with equality if and only if $n=0$ and $d=\max\{1, p_g-1\}$.

\medskip

(i) Suppose that $p_g\geq 2$. In this case, we have $\max\{1, p_g-1\} = p_g-1$. We divide the discussion according to the value of $c$.

(i.1) If $c< \frac{p_g-1}{p_g+1}$ then by \eqref{eq: b0'}
\[
 \frac{d+\sum_i \mu_i}{2+d+\sum_i \mu_i}\geq  \frac{p_g-1}{p_g+1} \geq c.
\] 
It follows from \eqref{eq: b0} that $b=c$, and by Claim~\ref{claim: vol k1 ineq} we have
\begin{equation}
\vol(K_X+B)\geq (2c-c^2)d -2c^2 \geq  (2c-c^2)\max\{1, p_g-1\} - 2c^2
\end{equation}
with equalities if and only if $n=0$ and $d=\max\{1, p_g-1\}$.

(i.2) If $c\geq \frac{p_g-1}{p_g+1}$ then by \eqref{eq: b0}  and \eqref{eq: b0'} we have
\[
b = \min\left\{c, \frac{d+\sum_i \mu_i}{2+d+\sum_i \mu_i}\right\} \geq \min\left\{\frac{p_g-1}{p_g+1}, \frac{d+\sum_i \mu_i}{2+d+\sum_i \mu_i}\right\} = \frac{p_g-1}{p_g+1}
\]
with equalities if and only if $n=0$ and $d=p_g-1$. By Claim~\ref{claim: vol k1 ineq} we have
\[
\vol(K_X+B)\geq (2b-b^2)m -2b^2 \geq  (2b-b^2)(p_g-1) -2b^2 \geq p_g-3 +\frac{4}{p_g+1}
\]
with equalities if and only if $n=0$, $b =  \frac{p_g-1}{p_g+1}$, and $d=p_g-1$.

\medskip

(ii) Suppose that $p_g\leq 1$. Again, we divide the discussion according to the value of $c$.

(ii.1) If $c< \frac{1}{3}$, then 
\[
 \frac{d+\sum_i \mu_i}{2+d+\sum_i \mu_i} \geq  \frac{d}{2+d} \geq  \frac{1}{3} > c.
\]
It follows from \eqref{eq: b0} that $b=c$, and by Claim~\ref{claim: vol k1 ineq} we have
\begin{equation}
\vol(K_X+B)\geq (2c-c^2)d -2c^2 \geq  (2c-c^2) - 2c^2 = 2c-3c^2
\end{equation}
where equalities hold if and only if $n=0$ and $d=1$.

(ii.2) If $c\geq \frac{1}{3}$, then by \eqref{eq: b0}  and \eqref{eq: b0'} we have
\[
b \geq \min\left\{c, \frac{d+\sum_i \mu_i}{2+d+\sum_i \mu_i}\right\} \geq \min\left\{\frac{1}{3}, \frac{d+\sum_i \mu_i}{2+d+\sum_i \mu_i}\right\} = \frac{1}{3}
\]
with equalities if and only if $n=0$ and $d=1$. By Claim~\ref{claim: vol k1 ineq} we have
\[
\vol(K_X+B)\geq (2b-b^2)d -2b^2 \geq  (2b-b^2) -2b^2 \geq \frac{1}{3}
\]
with equalities if and only if $n=0$, $d=1$, and $b = \frac{1}{3}$.

Obviously, the lower bounds obtained in Step 4 (when $p_a(\tilde B_0)=0$) is smaller than the one obtained in Step 3  (when $p_a(\tilde B_0)>0$). The proof of Proposition~\ref{prop: 1} is now complete. \end{proof}

\begin{cor}\label{cor: linear lower bound}
For any projective log canonical surface of general type $(X, B)$, we have
\begin{equation}\label{eq: d=2 lc N}
\vol(K_X+B)\geq (2c-c^2)p_g(X, B)-(2c+c^2) 
\end{equation}
where $c:=\min(\sC_B\cup\{1\})$. 
\end{cor}
\begin{proof}
If $p_g:=p_g(X, B)\geq 2$ then $(X, B)\in \sS(\sC_B, p_g; \kappa)$ with $\kappa\in\{1,2\}$. One checks directly that the following strict inequality holds
\begin{equation}\label{eq: strict ineq}
(2c-c^2)(p_g-1)-2c^2 <  p_g-3 + \frac{4}{p_g+1} 
\end{equation}
Thus the lemma holds in this case by Propositions~\ref{prop: 2} and \ref{prop: 1}. If $p_g\leq 1$, then the right hand side of \eqref{eq: d=2 lc N} is negative, so the inequality also holds.
\end{proof}

\subsection{The lower bound of $\KK^2(\sC, p_g; 0)$}
\begin{prop}\label{prop: 0}
For any $(X, B)\in \sS(\sC, p_g; 0)$, we have
\[
\vol(K_X+B)\geq  
\begin{cases}
\frac{1}{42}c^2 & \text{ if } c\leq \frac{7}{13} \\
-\frac{11}{6} c^2 +2c -\frac{7}{13} & \text{ if }  \frac{7}{13}<c\leq\frac{6}{11} \\
\frac{1}{143} & \text{ if }  c>\frac{6}{11}
\end{cases}
\] 
\end{prop}

\begin{proof}
\noindent{\bf Step 0.} Suppose that $(X, B)\in \sS(\sC, p_g; 0)$ has ample $K_X+B$ and $\vol(K_X+B) = \min\KK^2(\sC, p_g; 0)$. It suffices to show the inequalities of the proposition for this log surface. Then $(X, B)$ and the minimal resolution $\pi\colon(\tilde X, B_{\tilde X})\rightarrow (X, B)$ are subject to the restrictions of Proposition~\ref{prop: nec cond}.  In particular, we have $B_{\tilde X}^{\ns} = (B_{\tilde X}^{\ns})_{\leq c}$. Also, observe that $(B_{\tilde X}^{\ns})_{<c}$ is necessarily contracted by $\pi\colon \tilde X\rightarrow X$. Therefore,
\begin{equation}\label{eq: raise to c}
\vol(K_X+B)= \vol(K_{\tilde X} + B_{\tilde X}) = \vol(K_{\tilde X} + B_{\tilde X}^\s +  B_{\tilde X}^\ns )  = \vol(K_{\tilde X} + B_{\tilde X}^\s + c\lceil B_{\tilde X}^{\ns} \rceil).
\end{equation}

\medskip

\noindent{\bf Step 1.} In this step, we pass to the $(K_{\tilde X}+B_{\tilde X}^\s)$-minimal model $\rho\colon (\tilde X, B_{\tilde X}^\s)\rightarrow (Z, B_Z')$ as in Notation \ref{nota: triangle}. Then $Z, B_Z':=\rho_*B_{\tilde X}^\s$ and $B_Z'':=\rho_*B_{\tilde X}^\ns$ enjoy the properties of Corollary~\ref{cor: pro min Z}. Since $\kappa(K_{\tilde X}+ B_{\tilde X}^\s)=0$ by assumption, we have $K_Z+B_Z'\sim_{\QQ} 0$ and hence $B_Z''\sim_{\RR} K_Z+B_Z'+B_Z'' = \rho_*(K_{\tilde X}+B_{\tilde X})$ is big. Using Notation~\ref{nota: SZBZ}, we have $(X, B_{\tilde X}^\s + c\lceil B_{\tilde X}^{\ns}\rceil) \in \sS(Z, B_Z^{(c)}; B_Z')$ and $\vol(K_X+B) \in \KK^2(Z, B_Z^{(c)}; B_Z')$, where $B_Z^{(c)} = B_Z' + c\lceil B_Z'' \rceil$. Hence it suffices to bound the set $\KK^2(Z, B_Z^{(c)}; B_Z')$ from below. 

\medskip

\noindent{\bf Step 2.}  In this step we prove that $\supp(B_Z'')$ is connected. Otherwise, we can write $B_Z''=D_1+D_2$ with $D_1$ and $D_2$ two effective $\RR$-divisors with disjoint supports.  Let $P_1\leq D_1$ and $P_2\leq D_2$ be the positive parts. Then $P=P_1+P_2$ is the positive part of $B_Z''$.  Since $P_1P_2=0$ and $B_Z''$ is big, we have $0<P^2=P_1^2+P_2^2$. So either $P_1^2>0$ or $P_2^2>0$. It follows that either $D_1$ or $D_2$ is big, contradicting Proposition~\ref{prop: nec cond} (iii).

\medskip

\noindent{\bf Step 2.} In this step, we assume that $p_a(\lceil B_Z''\rceil)\geq 2$. By the adjunction formula $\lceil B_Z''\rceil^2\geq 2$. 

\begin{claim}\label{claim: no -2 tail}
There is no smooth rational component $D$ of $\lceil B_Z''\rceil$ such that $D\cdot (\lceil B_Z''\rceil-D)=1$.
\end{claim}
\begin{proof}[Proof of Claim~\ref{claim: no -2 tail}]
Otherwise $D^2=-2$ and $(\lceil B_Z''\rceil-D)^2=\lceil B_Z''\rceil^2>0$, so $\lceil B_Z''\rceil-D$ is big, contradicting  Proposition~\ref{prop: nec cond} (iii).    
\end{proof}                                                                            

%Suppose first that $C$ has a component $C_1$ with $p_a(C_1)\geq 2$. Then $C_1$ is already big and  $C=C_1$ by the minimality of $C$. Since $p_a(\hat C)=0$ and
%\[
%2\leq p_a(C)=p_a(\hat C)+\sum_i \frac{m_i(m_i-1)}{2}
%\]%
%we have $\sum_i m_i\geq 3$. We have 
%\[
%\hat C_1^2 = -K_{\hat X}\hat C -2 =-(m_1+\dots+m_k+2)\leq -5.
%\]
%The divisor $K_{\hat X}+\hat C+\hat C$ is numerically equivalent to 
%\[
%\mathcal E_0+ \dots+ \mathcal G_k +\hat C
%\]
%The nef part of $K_{\hat X}+\hat C+\hat C$ is
%\[
%\phi^*(K_X+\Celta)=\varphi^*(K_{X'}+\Celta')=K_{\hat X} + \hat C+\frac{\sum_{1\leq i\leq r} m_i}{2+\sum_{1\leq i\leq r} m_i} \hat C,
%\]%
%So
%\begin{align*}
%(K_X+\Celta)^2&=(K_{\hat X} + \hat C+\frac{\sum_{1\leq i\leq r } m_i}{2+\sum_{1\leq i\leq r} m_i} \hat C)^2\\
%&=(\mathcal E_0+ \dots+ \mathcal G_k+\frac{\sum_{1\leq i\leq r } m_i}{2+\sum_{1\leq i\leq r} m_i} \hat C)(\mathcal E_0+ \dots+ \mathcal G_k)\\
%&=-k+\frac{(\sum_{1\leq i\leq r } m_i)^2}{2+\sum_{1\leq i\leq r} m_i}\\
%&\geq \frac{2}{3}
%\end{align*}
%where the equality case is achieved if and only if $k=2, m_1=m_2=2$.

%If $ c\mu\leq 1$ then $\rho\colon \tilde X\rightarrow{Z}$ is an isomorphism (cf.~Claim~\ref{claim: >1}), so $K_{\tilde X}+B_{\tilde X}^\s = K_{Z}+B_{Z}^\s \sim 0$ and by \eqref{eq: raise to c}
%\begin{multline}
%\vol(K_X+B)  = \vol(K_Z+B_Z'+ c \lceil B_Z''\rceil)  = \vol(c\lceil B_Z''\rceil) = c^2\vol(\lceil B_Z''\rceil) \geq 2 c^2.
%\end{multline}
%
%Now we assume that $c \mu> 1$. 

Let $\mu$ be the maximal multiplicity of $\lceil B_Z''\rceil$ at a point. By Lemma~\ref{lem: bdd below},  we have
\begin{equation}
K_{\tilde X}+B_{\tilde X}^\s+c\lceil B_{\tilde X}^\ns\rceil \geq \rho^*(K_{Z}+B_Z') +\frac{1}{c\mu}\rho^*(c\lceil B_Z''\rceil)\equiv\frac{1}{\mu}\rho^*\lceil B_Z''\rceil,
\end{equation}
 and hence by \eqref{eq: raise to c}
\begin{equation}\label{eq: pa2}
\vol(K_X+B)  =\vol( K_{\tilde X}+B_{\tilde X}^\s+c\lceil B_{\tilde X}^\ns\rceil) \geq \frac{1}{\mu^2}\vol(\lceil B_Z''\rceil).
\end{equation}

If $\mu\leq 3$, then by \eqref{eq: pa2} we have
  \[
\vol(K_X+B)\geq \frac{1}{9}\vol(\lceil B_Z''\rceil)\geq \frac{1}{9} \lceil B_Z''\rceil^2\geq  \frac{2}{9},
\]
where the second inequality is because of Lemma~\ref{lem: vol ZD} (i). 

Now suppose that $\mu\geq 4$. Let $p\in \lceil B_Z''\rceil$ be a point of multiplicity $\mu$ and $\tilde Z \rightarrow Z$ the blow-up of $p$. Let $\tilde B_Z''$ be the strict transform of $B_Z''$. Then $\lceil\tilde B_Z''\rceil$ has at most $\mu$ connected components, and we have
\[
p_a(\lceil B_Z''\rceil) = p_a(\lceil\tilde B_Z''\rceil) +\frac{\mu(\mu-1)}{2}\geq 1-\mu+\frac{\mu(\mu-1)}{2}=\frac{(\mu-1)(\mu-2)}{2}.
\]
It follows that 
\[
\vol(\lceil B_Z''\rceil)\geq \lceil B_Z''\rceil^2= K_Z. \lceil B_Z''\rceil+ \lceil B_Z''\rceil^2= 2p_a(\lceil B_Z''\rceil) -2\geq \mu(\mu-3),
\]
where the first inequality is by Lemma~\ref{lem: vol ZD} (i), so by \eqref{eq: pa2} we have
\[
\vol(K_X+B)\geq \frac{1}{\mu^2}\vol(\lceil B_Z''\rceil) \geq 1-\frac{3}{\mu}\geq \frac{1}{4}.
\]

\medskip

\noindent{\bf Step 3.} In this step we assume that $p_a(\lceil B_Z''\rceil)\leq 1$. Then every connected subcurve of $\lceil B_Z''\rceil$ has arithmetic genus at most $1$; each smooth rational component of $\lceil B_Z'' \rceil$, if present, has zero intersection number of $K_Z$ and hence should be a $(-2)$-curve. It is then not hard to see that, the assumption $p_a(\lceil B_Z''\rceil)\leq 1$ together with Proposition~\ref{prop: nec cond} implies that $\lceil B_Z''\rceil=C+D$ where $C$ supports a curve of canonical type and $D$ is a $(-2)$-curve such that $C\cdot D=1$. By the explicit computation done in Section~\ref{sec: can curve}, the minimal possible volume occurs when $C$ supports a curve of type $\II^*$, and it serves as the uniform lower bound of Proposition~\ref{prop: 0}.
\end{proof}

Combining Propositions~\ref{prop: 2}, \ref{prop: 1}, and \ref{prop: 0}, we obtain
\begin{thm}\label{thm: lower bound}
Let $\sC\subset(0,1]$ be a subset such that $\sC\cup\{1\}$ attains the minimum, say $c$, and let $p_g$ be a positive integer. Then the following holds for any $(X, B)\in \sS(\sC, p_g)$.
\begin{enumerate}
\item If $p_g\geq 2$, then
\[
\vol(K_X+B)\geq
\begin{cases}  
(2c-c^2)(p_g-1)-2c^2 & \text{ if } c < \frac{p_g-1}{p_g+1}\\
 p_g-3 + \frac{4}{p_g+1} & \text{ if }  c\geq \frac{p_g-1}{p_g+1} 
\end{cases}  
\]
\item If $p_g = 1$, then
\[
\vol(K_X+B)\geq
\begin{cases}
\frac{1}{42}c^2, & \text{ if } c\leq \frac{7}{13} \\
-\frac{11}{6} c^2 +2c -\frac{7}{13}, & \text{ if }  \frac{7}{13}<c\leq\frac{6}{11} \\
\frac{1}{143}, & \text{ if }  c>\frac{6}{11}.
\end{cases}
\] 
\end{enumerate}
\end{thm}
In Section~\ref{sec: ex} we will show that Theorem~\ref{thm: lower bound} is sharp in that the equalities can be achieved for any given $p_g>0$.

The next proposition gives practical criteria for a projective log canonical surface $(X, B)$ to lie in $\sS(\sC, p_g; \kappa)$ with $\kappa\geq 0$.
\begin{prop}\label{prop: irr}
Let $(X, B)$ be a smooth projective log canonical surface of general type, and let $B = B^{\s} + B^\ns$ be the semistable decomposition. Suppose that one of the following holds:
\begin{enumerate}
\item $q(X)>0$ and either $\supp(B^\ns)$ is empty or it consists of rational curves.
\item $q(X)=0$ and $B^{\s}\neq 0$. 
\end{enumerate}
Then $\kappa(K_X+B^{\s})\geq 0$. Moreover, in case (ii), we have $p_g(X, B)>0$.
\end{prop}
\begin{proof}
(i) Suppose that $q(X)>0$ and $\supp(B^\ns)$ consists of rational curves. Suppose on the contrary that $\kappa(K_X+B^{\s})=-\infty$. Then $X$ is birational to a ruled surface and the Albanese map  $\alpha\colon X\rightarrow A$ is a $\PP^1$-fibration.  Any minimal model program on $X$ contracts only rational curves, which are necessarily contained in fibers of $\alpha$; in other words, it is over $A$. Since $\kappa(K_X+B^{\s})=-\infty$, we have $(K_X+B^{\s})\cdot F <0$ for a general fiber $F$ of $\alpha$. On the other hand, since $B^{\ns}$ consists of rational curves, it is contracted by $\alpha$ and thus $B^{\ns}\cdot F=0$.  Due to the bigness of $K_X+B^{\s}+B^{\ns}$, we have $(K_X+B^{\s})\cdot F=(K_X+B^{\s}+B^{\ns})\cdot F> 0,$ which is a contradiction. 

(ii) Now suppose that $q(X)=0$ and $B^{\s}\neq 0$. Then $h^0(B^{\s}, K_{B^{\s}})>0$ because every connected component of $B^{\s}$ has positive arithmetic genus. A portion of the long exact sequence associated to the short exact sequence $0\rightarrow \sO_X(K_X)\rightarrow \sO_X(K_X+B^{\s})\rightarrow \sO_{B^{\s}}(K_{B^{\s}})\rightarrow 0$ reads
\begin{equation}\label{eq: surj}
H^0(X, K_X+B^{\s})\rightarrow H^0(B^{\s}, K_{B^{\s}}) \rightarrow H^1(X, K_X)
\end{equation}
Since $q(X)=0$, we have $H^1(X, K_X)=0$ and thus the first map of \eqref{eq: surj} is surjective. It follows that $H^0(X, K_X+B^{\s}) \neq 0$ and hence $\kappa(K_X+B^{\s})\geq 0$.
\end{proof}

\begin{cor}\label{cor: <1/143}
Suppose that $(X, B)\in \sS(\{1\}, p_g)$ has $\vol(K_X+B)<\frac{1}{143}$. Then $\tilde X$ is a rational surface and $B_{\tilde X}^\s =0$, where $\pi\colon (\tilde X, B_{\tilde X})\rightarrow (X, B)$ is the minimal resolution.
\end{cor}
\begin{proof}
Since $\vol(K_X+B)<\frac{1}{143}$, we have $\kappa(K_{\tilde X} + B_{\tilde X}^\s) = -\infty$ by Theorem~\ref{thm: main} applied to the case $\sC=\{1\}$. Note that the non-rational components of $B_{\tilde X}$ are necessarily contained in the semistable part $B_{\tilde X}^\s$, so $B_{\tilde X}^\ns$ consists of rational curves. By Proposition~\ref{prop: irr} (i), applied to $(\tilde X, B_{\tilde X})$, we have $q(\tilde X)=0$ and $B_{\tilde X}^\s =0$. It follows that $\tilde X$ is a rational surface. 
\end{proof}

\section{Finding the minimal volumes}\label{sec: ex}
Given a subset $\sC\subset (0,1]$ such that $\min (\sC\cup \{1\}) = c$, a positive integer $p_g$, and $\kappa\in\{0,1,2\}$, we have obtained low bounds of $\KK^2(\sC, p_g; \kappa)$ in Section~\ref{sec: lower bound}. We construct in this section projective log surfaces $(X, B)\in \sS(\sC, p_g; \kappa)$ realizing the low bounds in most cases. In particular, those with $\vol(K_X+B) = \min \KK^2(\sC, p_g; \kappa)$ are presented when $p_g>0$. 

\subsection{Log surfaces in $\sS(\sC, p_g; 2)$ achieving the minimal volume}
The following theorem of Sakai gives a characterization of the equality case of Proposition~\ref{prop: 2}. Note that the proof of \cite[Theorem~6.7]{Sak80} works in all characteristic.
\begin{thm}[\cite{Sak80}, Theorem~6.7]\label{thm: Sak80}
If $(X, B)\in \sS(\sC, p_g; 2)$ has ample $K_X+B$ and $\vol(K_X+B) =  p_g-2$, then the minimal resolution $(\tilde X, B_{\tilde X})$ of $(X, B)$ is one of the following:
\begin{enumerate}
\item ($\PP^2$, nodal quartic curve),
\item ($\PP^2$,  nodal quintic curve),
\item $\tilde X =\FF_e$, $B_{\tilde X}$ a nodal curve in $|3\Gamma_0 + (2e+k+2)F|$ with $k\geq 1, e\geq 0$,
\item $\tilde X =\FF_e$, $B_{\tilde X}$ a nodal curve in $|3\Gamma_0 + (2e+2)F|$ with $e>0$,
\end{enumerate}
where $\FF_e=\PP(\sO_{\PP^1}\oplus \sO_{\PP^1}(-e))\rightarrow \PP^1$ denotes a Hirzebruch surface with  fiber $F$ and a section $\Gamma_0$ such that $\Gamma_0^2=-e$.
\end{thm}

\begin{cor}
 $\min\KK^2(\sC, p_g; 2)= \max\{1, p_g-2\}$ if and only if either $0\leq p_g\leq 2$ or $1\in \sC$.
\end{cor}
\begin{proof}
One can find smooth projective surfaces of general type with $0\leq p_g\leq 2$ and $K_S^2=1$ (see \cite{Cat79, Cat87}). If $p_g\geq 3$, Theorem~\ref{thm: Sak80} provides examples $(X, B) \in \sS(\{1\}, p_g; 2)$ attaining the lower bound of Proposition~\ref{prop: 2}, that is, $\vol(K_X+B)=  p_g-2$; the same theorem shows that the inequality of Proposition~\ref{prop: 2} is strict if $1\notin \sC$.
\end{proof}
\subsection{Log surfaces in $\sS(\sC, p_g; 1)$ achieving the minimal volume}
In this subsection, we construct for each $p_g>0$ a log surface $(X, B)\in \sS(\sC, p_g; 1)$ realizing the equality of Proposition~\ref{prop: 1}. By the proof of Proposition~\ref{prop: 1}, we need to construct a smooth projective log canonical surface $(Z, B_Z'+D)$ such that
\begin{enumerate}
\item $B_Z'$ is a (reduced) semistable curve such that $\kappa(K_{Z}+B_Z')=1$ and $K_{Z}+B_Z'$ is nef,
\item $D$ is a smooth rational curve with $K_Z\cdot D =\max\{p_g-1, 1\}$, and
\item $\supp(B_Z')\cap\, \supp(D)=\emptyset$.
\end{enumerate}
Then, setting $B_Z^{(c)} = B_Z'+c D$, $(X, B)$ can be taken as the ample model of $(Z, B_Z^{(c)})$.

\begin{ex}\label{ex: pg>1}
Let $p_g\geq 2$ be an integer. Let $h\colon S\rightarrow \PP^1$ be a relatively minimal rational elliptic surface with a section $\Gamma$. Then by the canonical bundle formula, we can take $K_S = -F$, where $F$ is a fiber of $f$. Set $B_S'= \sum_{1\leq i \leq p_g}{F_i}$, where $F_1, \dots, F_{p_g}$ are $p_g$ general fibers of  $h$. Then $K_S+B_S' \sim (p_g-1)F$.

Let $p_i$ be the intersection point of $\Gamma$ and $F_i$ for $1\leq i\leq p_g$, and let $\rho\colon Z\rightarrow S$ be the blow-up of $\{p_i\}_{1\leq i\leq p_g}$. Let $B_Z'$ and $D$ be the strict transforms of $B_S'$ and $\Gamma$ on $Z$ respectively, and set $B_Z^{(c)}= B_Z' + c D$. Then $B_Z'$ is the semistable part of $B_Z^{(c)}$,  $K_Z+B_Z' \sim (p_g-1) f^*t$ for $t\in \PP^1$, and $K_Z+B_Z^{(c)}$ is big for $c=\min (\sC\cup\{1\})$. Thus $p_g(Z, B_Z') = p_g$.

Let $\pi\colon (Z, B_Z^{(c)}) \rightarrow (X^{(c)}, B^{(c)})$ be the contraction to the ample model. Then $B_Z'$ is contracted by $\pi$, and $D$ is contracted by $\pi$  if and only if $c\geq \frac{p_g-1}{p_g+1}$. It follows that $B^{(c)} = c\pi_*D$ and hence $(X^{(c)}, B^{(c)})\in \sS(\{c\}_{<1}, p_g; 1)$. The volume is exactly the lower bound of Proposition~\ref{prop: 1}:
\[
\vol\left(K_{X^{(c)}}+B^{(c)}\right) = \vol\left(K_Z+B_Z^{(c)}\right)=
\begin{cases}  
(2c-c^2)(p_g-1)-2c^2 & \text{ if } c < \frac{p_g-1}{p_g+1}\\
 p_g-3 + \frac{4}{p_g+1} & \text{ if }  c\geq \frac{p_g-1}{p_g+1} 
\end{cases}  
\]
\end{ex}

\begin{ex}
Let $p_g\geq 2$ be an integer. Take $(Z, B_Z^{(c)}) = (\FF_e, \Gamma+c\Gamma_0) $, where $e=p_g-1$, $\FF_e$ is the Hirzebruch surface with a section $\Gamma_0$ such that $\Gamma_0^2=-e$, and $\Gamma\in |2(\Gamma_0+eF)|$ is a general element. Then the ample model $(X^{(c)}, B^{(c)})$ lies in $\sS(\{c,1\}, p_g; 1)$ and $\vol(K_{X^{(c)}}+B^{(c)})$ achieves the lower bound of Proposition~\ref{prop: 1}. In this example, we have $(B^{(c)})^\s\neq 0$.
\end{ex}

\begin{ex}\label{ex: pg=1}
Let $S$ be a (classical) Enriques surface and $h\colon S\rightarrow \PP^1$ an elliptic fibration. Then $h$ has two double fibers, say $2F_1$ and $2F_2$, and a double section $\Gamma$; see \cite[Theorem~5.7.2]{CossecDogachev89}. we have $K_S\sim F_2-F_1$. Note that the double fibers are necessarily of type $2I_{b_i}$, $i\in\{1,2\}$, and hence the reduced parts $F_i$ have at most nodes as singularities, and the curves $\Gamma$ and $F_i$ intersect transversally at one point, say $p_i$. Let $\rho\colon Z\rightarrow S$ be the blow-up of $p_1$. Let $B_Z'$ and $D$ be the strict transforms of $F_1$ and $\Gamma$ respectively. Then $K_Z+B_Z'=\rho^*(K_S+F_1)\sim \rho^* F_2$ has Iitaka--Kodaira dimension $1$. Moreover, we have
\[
B_Z'\cap\, D=\emptyset,\quad  K_Z\cdot D = 1, \quad p_g(Z,  B_Z') = 1.
\]
Let $B_Z^{(c)}= B_Z' + c D$ and $\pi\colon (Z, B_Z^{(c)}) \rightarrow (X^{(c)}, B^{(c)})$ the contraction to the ample model. Then, as in Example~\ref{ex: k3}, we have $(X^{(c)}, B^{(c)})\in \sS(\{c\}_{<1}, 1; 1)$. One computes easily that
\[
\vol\left(K_{X^{(c)}}+ B^{(c)}\right)=\vol\left(K_Z+B_Z^{(c)}\right)=
\begin{cases}  
2c-3c^2 & \text{ if } c \leq \frac{1}{3} \\
 \frac{1}{3} & \text{ if }  c> \frac{1}{3} 
\end{cases}  
\]
\end{ex}

As a consequence of Proposition~\ref{prop: 1} and the above examples, we obtain
\begin{thm}\label{thm: 1}
Let $\sC\subset(0,1]$ be a subset such that $\sC\subset (0,1]$ attains the minimum, say $c$. Then the following holds.
\begin{enumerate}[leftmargin=*]
\item
For $p_g\geq 2$, we have
\[
\min \KK^2(\sC, p_g; 1) = \min \KK^2(\{c\}_{<1}, p_g; 1)=
\begin{cases}  
(2c-c^2)(p_g-1)-2c^2 & \text{ if } c < \frac{p_g-1}{p_g+1}\\
 p_g-3 + \frac{4}{p_g+1} & \text{ if }  c\geq  \frac{p_g-1}{p_g+1} 
\end{cases}  
\]
\item For $p_g = 1$, we have
\[
\min \KK^2(\sC, 1; 1) = \min \KK^2(\{c\}_{<1}, p_g; 1)=
\begin{cases}
2c - 3c^2 & \text{if $c< \frac{1}{3}$} \\
\frac{1}{3} & \text{if $c\geq  \frac{1}{3}$} 
\end{cases}
\]
\end{enumerate}
\end{thm} 
 \begin{cor}\label{cor: linear lower bound 2}
 The inequality of Corollary~\ref{eq: d=2 lc N} is optimal in the following sense:
\begin{enumerate}[leftmargin=*]
\item  If  $c<1$, then the equality can be attained for  $(X, B)$ with  $p_g(X, B)\geq \frac{1+c}{1-c}$. 
\item  if $c=1$ then the inequality is strict, but there is a sequence of projective log canonical surfaces  $X_i$ of general type such that
\[
\lim_{i\to \infty}\vol(K_{X_i}) - p_g(X_i) + 3 =0
\]
\end{enumerate}
 \end{cor}

\subsection{Log surfaces in $\sS(\sC, p_g; 0)$ achieving the minimal volume}
In this subsection, we construct for a log surface $(X, B)\in \sS(\sC, 1; 0)$ realizing the equality of Proposition~\ref{prop: 0}. If the characteristic of the base field $\bbk$ is $2$,  the same can be done when $p_g=0$.

The point is to exhibit a smooth projective log canonical surface $(Z, B_Z')$ together with a simple normal crossing curve $C+D$ such that 
\begin{enumerate}[leftmargin=*]
\item $B_Z'$ is a (reduced) semistable curve such that $K_Z+ B_Z' \sim_\QQ 0$,
\item $\disp C=\sum_{1\leq i\leq 9} C_i$ supports a curve of canonical type $\II^*$ and $D$ is a $(-2)$-curve intersecting $C$ transversally at exactly one point, so that the dual graphs is
\begin{center}
  \begin{tikzpicture}
\node[wbullet, label = below:$D$]  (D) at (0,0) {};
\foreach \x in {1,2,3,4,5,6,7,8}
\node[wbullet, label = below:$C_\x$] (C\x) at (\x,0) {};
\node[wbullet, label = right:$C_9$] at (6, 1) (C9){};
\foreach \x/\y in {1/2,2/3,3/4,4/5,5/6,6/7,7/8, 6/9}
\draw (C\x)--(C\y);
\draw (D) -- (C1);
\end{tikzpicture}
  \end{center}
  \item $\supp(B_Z')\cap \supp(C+D)=\emptyset$
\end{enumerate}
Let $\rho\colon W\rightarrow Z$ be the blow up of all nodes of $C+D$ and $B_W^{(c)} = \rho_*^{-1}B_Z' + c\rho_*^{-1}(C+D)$. Then the ample model $(X^{(c)}, B^{(c)})$ of $(W, B_W^{(c)})$ lies in $\sS(\{c,1\}, p_g; 0)$ and $\vol(K_X+B) = \min \KK^2(\{c,1\}, p_g; 0)$, where $p_g=p_g(Z, B_Z')\in \{0,1\}$; see Section~\ref{sec: can curve}. In Examples~\ref{ex: k3}, \ref{ex: rational} and \ref{ex: enriques}, the coefficient set of $B^{(c)}$ is as follows, depending on the value of $c$:
\[
\sC_{B^{(c)}} = 
\begin{cases}
\{c\} & \text{if $c< \frac{7}{13}$}\\
\emptyset & \text{if $c\geq \frac{7}{13}$} 
\end{cases}
\] 

\begin{ex}\label{ex: k3}
Let $Z$ be a K3 surface, $C+D$ the reduced simple normal crossing divisor, consisting of $(-2)$-curves with dual graph as above (see \cite{Shi06} for the existence of such a K3 surface, at least in characteristic zero). It suffices to take $B_Z'=0$, and we have $p_g(Z, B_Z')=p_g(Z)= 1$ in this case. 
\end{ex}
\begin{ex}\label{ex: rational}
Let $A$ be a nodal cubic curve on $\PP^2$ with at most nodes as singularities and a line  $L$ intersecting $A$ transversally at 3 points. We blow up the intersection points $A\cap\, L$ and their infinitely near points on the strict transforms of $A$ successively to arrive at the following configuration of curves
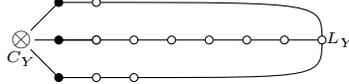
\begin{center}
  \begin{tikzpicture}[font=\tiny]
    \begin{scope}
    \foreach \x in {1,...,7}
    \node[wbullet] at (\x+1, 0) (D\x){};
    \node[right] at (D7.east) {$L_Z$};
    \foreach \x in {8,9}
    \node[wbullet] at (\x-6,-1) (D\x){};
    \node[wbullet] at (2,1) (D10){};
       \foreach \y in {1,2,3}
\node [bbullet] at (1,\y-2) (E\y) {};  
\node[wbullet,label=below: $A_Z$] at (0,0) (A) {};
    \draw (A)--(E1)--(D8) -- (D9) ..controls (8,-1).. (D7);
    \draw (A)--(E2)--(D1)--(D2)--(D3)--(D4)--(D5)--(D6)--(D7);
    \draw (A)--(E3)--(D10)..controls (8,1)..(D7);
\end{scope} 
\end{tikzpicture}
  \end{center}
 where $A_Z$ and $L_Z$ are the strict transforms of $A$ and $L$ respectively, the white bullets denote $(-2)$-curves, and the black bullets denote $(-1)$-curves. We take $B_Z'=A_Z$, $C+D$ the sum of all the $(-2)$-curves visible in the dual graph, and we have $p_g(Z, B_Z')=1$ in this case. 
\end{ex}

\begin{ex}\label{ex: enriques}
Suppose that $\cha\bbk =2$. Let $Z$ be a classical/supersingular Enriques surface of type $\tilde E_8$ (\cite{Salomonsson03}), so there is a configuration $C+D$ of $(-2)$-curves as above. We take $B_Z'=0$, and we have 
\[
p_g(Z, B_Z')=p_g(Z)=
\begin{cases}
0 & \text{ if $Z$ is classical}\\
1 & \text{ if $Z$ is supersingular}
\end{cases}
\]
\end{ex}
By Proposition~\ref{prop: 0} and the above examples, we obtain
\begin{thm}\label{thm: 0}
Let $\sC\subset(0,1]$ be a subset such that $\sC\subset (0,1]$ attains the minimum, say $c$. Then
\[
\min \KK^2(\sC, 1; 0) = \min \KK^2(\{c\}_{<1}, 1; 0)=
\begin{cases}
\frac{1}{42}c^2 & \text{ if } c\leq \frac{7}{13} \\
-\frac{11}{6} c^2 +2c -\frac{7}{13} & \text{ if }  \frac{7}{13}<c\leq\frac{6}{11} \\
\frac{1}{143} & \text{ if }  c>\frac{6}{11}
\end{cases}
\]
\end{thm}
We are going to address the sharpness of the inequality in Proposition~\ref{prop: 0} in the case $p_g=0$. First, a simple lemma on the algebraic fundamental group.
\begin{lem}\label{lem: pi1}
Let $Z$ be a smooth projective surface and $C$ a big connected curve on $Z$. If the algebraic fundamental group $\pi_1^{\alg}(C)$ is trivial, then so is $\pi_1^{\alg}(S)$. 
\end{lem}
\begin{proof}
Suppose on the contrary that there is an \'etale cover $f\colon \tilde Z\rightarrow Z$ of degree $d>1$. Then $f^*C = \cup_{1\leq i\leq d}\tilde C_i$, where the $\tilde C_i$'s are pairwise disjoint and $\tilde C_i\cong C$. Since each $C_i$ is big, this contradicts the Hodge index theorem. \end{proof}

\begin{prop}\label{prop: 0=}
Let $\sC\subset(0,1]$ be a subset such that $\sC\cup\{1\}$ attains the minimum, say $c$. Let $(X, B)\in\sS(\sC, p_g; 0)$ be such that $K_X+B$ is ample and $\vol(K_X+B) =\min\KK^2(\sC, 1; 0)$. Then the minimal resolution $\pi\colon (\tilde X, B_{\tilde X}) \rightarrow (X, B)$ has the following characterization:
\begin{enumerate}[leftmargin=*]
\item  $\tilde X$ is birationally a K3 surface,  or a rational surface, or in case $\cha \bbk =2$, $\tilde X$ is a classical/supersingular Enriques surface. In particular, the algebraic fundamental group $\pi_1^\alg(\tilde X)$ is trivial and if $\bbk\cong \CC$ then the topological fundamental group $\pi_1(\tilde X)$ is trivial.
\item  The semistable part $B_{\tilde X}^\s$ and the non-semistable part $B_{\tilde X}^\ns$ of $B_{\tilde X}$ are disjoint. 
\item if  $c\leq \frac{7}{13}$ then $B_{\tilde X}^\ns$ a simple normal crossing divisor consisting of $(-2)$-curves, and its dual graph is
\begin{center}
  \begin{tikzpicture}[font=\tiny]
    \begin{scope}
        \foreach \x in {0,1,...,7}  
    \draw (\x,0)--(\x+1,0);
    \draw (6,0)--(6,1);    
   \foreach \x in {0,1,...,8}
\node[wbullet] at (\x,0){};  
\node [wbullet] at (6,1) {};  
   \end{scope} 
\end{tikzpicture}
  \end{center}
\item if  $c> \frac{7}{13}$ then $B_{\tilde X}^\ns$ is obtained by blowing up the configuration of curves in (ii), so there is exactly one $(-1)$-curve, say $E$, not intersecting $B_{\tilde X}^\s$, and the dual graph of $B_{\tilde X}^\ns+E$ is 
\begin{center}
  \begin{tikzpicture}[font=\tiny]
        \foreach \x in {0,1,...,8}  
    \draw (\x,0)--(\x+1,0);
    \draw (7,0)--(7,1);    
   \foreach \x in {0,1,...,9}
\node[wbullet] at (\x,0){};  
\node[bbullet, label=below: $E$] at (6,0) {};
\node[label=below: $(-3)$] at (5,0) {};
\node[label=below: $(-3)$] at (7,0) {};
\node[wbullet] at (7,1) {};
\end{tikzpicture}
  \end{center}
where the two white bullets adjacent to $E$ represent $(-3)$-curves and all the other white bullets are \emph{($-2$)}-curves.
\item[(v)] If $c\geq \frac{6}{11}$ then $B=0$, and if $c<\frac{6}{11}$ then $B=cB_0$ where $B_0$ is the image of the $(-3)$-curve to the right of $E$ in the dual graph of (ii).
\end{enumerate}
\end{prop}
\begin{proof}
The statements (ii)-(iv) follow from the proof of Proposition~\ref{prop: 0} and the computation of \eqref{eq: II*} in Section~\ref{sec: can curve}. (v) can be seen by finding the Zariski decomposition on $K_{\tilde X} + B_{\tilde X}^\s +c\lceil B_{\tilde X}^\ns \rceil$.

For (i), consider  the minimal model $\rho\colon (\tilde X, B_{\tilde X})\rightarrow (Z, B_Z')$ of $K_{\tilde X} + B_{\tilde X}$. Then $(Z, B_Z')$ is a smooth projective log canonical surface as in the beginning of this subsection. In particular, there is a big connected curve $C+D$ such that $K_Z\cdot (C+D)=0$. It follows that $\kappa(Z)\leq 0$. Moreover, $\pi_1^{\alg}(S)$ is trivial by Lemma~\ref{lem: pi1}. By the Enriques--Kodaira classification of algebraic surfaces, if $\cha \bbk\neq 2$, then $\tilde X$ is birational to a K3 surface or a rational surface. If $\cha \bbk =2$, then $\tilde X$ can also be a classical/supersingular Enriques surface, which does not admit any nontrivial \'etale cover. 
\end{proof}

%\item $\min \KK^2(\sC, 0; 0)  \geq \min \KK^2(\sC, 1; 0)$, and the equality holds if $\cha \bbk=2$.

\begin{cor}\label{cor: pi1}
Let $\sC\subset(0,1]$ be a subset such that $\sC\subset (0,1]$ attains the minimum, say $c$. Then
\[
\min \KK^2(\sC, 0; 0) \geq 
\begin{cases}
\frac{1}{42}c^2 & \text{ if } c\leq \frac{7}{13} \\
-\frac{11}{6} c^2 +2c -\frac{7}{13} & \text{ if }  \frac{7}{13}<c\leq\frac{6}{11} \\
\frac{1}{143} & \text{ if }  c>\frac{6}{11}
\end{cases}
\]
and equality holds if and only if $\cha\bbk=2$.
 \end{cor}
 \begin{proof}
The inequality follows from Proposition~\ref{prop: 1}. By Example~\ref{ex: enriques} the equality holds if $\cha\bbk=2$. Now suppose that $\cha\bbk\neq 2$ and $(X, B)\in \sS(\sC, 0; 0)$ realizes the equality. Then by Proposition~\ref{prop: 0=} (i) the minimal resolution $(\tilde X, B_{\tilde X})$ of $(X,B)$ is birationally a K3 surface or a rational surface. In both cases, the fact that $\kappa(K_{\tilde X}+B_{\tilde X}^\s )= 0$ implies that $p_g=p_g(K_{\tilde X}, B_{\tilde X}^\s)=1$, which is a contradiction to the assumption that $p_g=0$. In other words, the inequality is strict if $\cha\bbk\neq 2$.
 \end{proof}
 
\subsection{Proofs of the main results}
We conclude this section by giving the proofs of the main results presented in the introduction.
\begin{proof}[Proof of Theorem~\ref{thm: main}]
Combine Lemma~\ref{lem: smaller2} (ii), Theorems~\ref{thm: lower bound}, \ref{thm: 1}, and \ref{thm: 0}.
\end{proof}
\begin{proof}[Proof of Theorem~\ref{thm: main2}]
This is just a reformulation of Theorem~\ref{thm: main}.
\end{proof}
\begin{proof}[Proof of Theorem~\ref{thm: linear lower bound}]
Combine Corollaries~\ref{cor: linear lower bound} and \ref{cor: linear lower bound 2}.
\end{proof}

\section{Applications}\label{sec: app}
In this section, we work over an algebraically closed field $\bbk$ of characteristic 0. We apply the main results to several closely related problems.
\subsection{A Noether type inequality for log canonical threefolds}
Given $c\in (0,1]$ and a positive integer $n$, \cite{JLZ21} proved a Noether type inequality of the form $\vol(K_X+B)\geq a_n(c) p_g(X, B)- b_n(c)$ for any $n$-dimensional projective log canonical pairs $(X, B)$ such that $K_X+B$ is big and $\min(\sC_B\cup\{1\}) = c$, where $ a_n(c)$ and $b_n(c)$ are positive constants depending on $c$. For example, we have the optimal value: $a_1(c) = b_1(c) = 1$. By Corollary~\ref{cor: linear lower bound}, we have the optimal values: $a_2(c) =2c-c^2$ and $b_2(c) = 2c+c^2$. Using this, we can also make $a_3(c)$ and $b_3(c)$ explicit as follows:
\begin{thm}\label{thm: 3 Noether}
Let $(X, B)$ be a projective log canonical threefold such that $K_X+B$ is big. Denote by $\sC_B$ the coefficient set of $B$ and $c=\min(\sC_B\cup\{1\})$. Then
\[
\vol(K_X+B) \geq a_3(c)\cdot p_g(X, B)  - b_3(c)
\]
where $b_3(c) = c$ and $a_3(c)$ is specified as follows:
\[
a_3(c) = 
\begin{cases}
\frac{1}{168}c^2, & \text{ if } c\leq \frac{7}{13} \\
-\frac{11}{24} c^2 +\frac{1}{2}c -\frac{7}{52}, & \text{ if }  \frac{7}{13}<c\leq\frac{6}{11} \\
\frac{1}{572}, & \text{ if }  c>\frac{6}{11}
\end{cases}
\] 
\end{thm}
\begin{proof}
Let $v_2^+(\{c,1\})$ be the minimal volume of projective log canonical surface of general type $(S, B_S)$ with coefficient in $\{c,1\}$ and $p_g(S, B_S)>0$. Then, by Proposition~\ref{prop: 1}, we have
\[
v_2^+(\{c,1\}) = 
\begin{cases}
\frac{1}{42}c^2, & \text{ if } c\leq \frac{7}{13} \\
-\frac{11}{6} c^2 +2c -\frac{7}{13}, & \text{ if }  \frac{7}{13}<c\leq\frac{6}{11} \\
\frac{1}{143}, & \text{ if }  c>\frac{6}{11}.
\end{cases}
\] 
Now, by \cite[Theorem~A.2]{JLZ21} and its proof, we have 
\[
\vol(K_X+B) \geq a_3(c) p_g(X, B)  - b_3(c)
\]
where
\[
a_3(c) = \frac{1}{4} \min\{a_2(c), v_2^+(\{c,1\})\} = \frac{1}{4} v_2^+(\{c,1\}) = 
\begin{cases}
\frac{1}{168}c^2, & \text{ if } c\leq \frac{7}{13} \\
-\frac{11}{24} c^2 +\frac{1}{2}c -\frac{7}{52}, & \text{ if }  \frac{7}{13}<c\leq\frac{6}{11} \\
\frac{1}{572}, & \text{ if }  c>\frac{6}{11}
\end{cases}
\]
and
\[
b_3(c)= \frac{1}{4} \max\{a_2(c)+b_2(c), v_2^+(\{c,1\})\} = \frac{1}{4} (a_2(c)+b_2(c)) = c 
\]
\end{proof}

\subsection{A Noether type inequality for stable surfaces}\label{sec: stable}
A stable surface is a projective surface with semi-log-canonical singularities whose canonical class is ample. They were introduced by \cite{KSB88} to compactify the moduli spaces of canonical surfaces of general type. It is natural and often possible to extend the results for canonical surfaces of general type to stable surfaces. For example, for a stable surface $X$, the inequality $K_X^2>p_g(X)-3$ holds if $X$ is either normal or Gorenstein (see \cite{TZ92, LR16} and Corollary~\ref{cor: linear lower bound}). By Example~\ref{ex: 25/84} given at the end of this subsection, this is not true for general stable surfaces. Instead, we are able to provide a weaker Noether type inequality that holds for all stable surfaces:
\begin{thm}\label{thm: noether}
Let $X$ be a stable surface. Then $K_X^2\geq \frac{1}{143}p_g(X).$
\end{thm}

Before giving the proof of Theorem~\ref{thm: noether}, we recall how one computes the canonical volume and the geometric genus of a possibly non-normal stable surface in terms of its normlization. So let $X$ be a non-normal stable surface and $\nu\colon\bar X\rightarrow X$ the normalization. Let $\bar D\subset \bar X$ the conductor divisor, which is a reduced curve on $\bar X$. Then the generically two-to-one map $\bar D\rightarrow D$ induces a double cover $\bar D^\nu\rightarrow D^\nu$ of the normalized curves, which in turn induces an involution $\tau$ on $\bar D^\nu$ such that $D^\nu=\bar D^\nu/\tau$ and the different $\Diff_{\bar D^\nu}(0)$ is $\tau$-invariant. The stable surface $X$ can be viewed as glued from $(\bar X, \bar D)$ along $\bar D$ via the involution $\tau$ on $\bar D^\nu$ (see \cite[ Theorem~5.13]{Kol13}).

The normal surface $\bar X$ is often not connected. We can write 
\[
\bar X=\bigcup_{1\leq i\leq n} \bar X_i
\]
as the (disjoint) union of its irreducible components. Let $\bar D_i$ be the part of $\bar D$ on $\bar X_i$. Then $(\bar X_i, \bar D_i)$ are all (connected) projective log canonical surfaces with ample $K_{\bar X_i} + \bar D_i$, and we have
\begin{equation}\label{eq: nn vol}
K_X^2=(K_{\bar X}+\bar D)^2=\sum_{1\leq i\leq n}(K_{\bar X_i}+\bar D_i)^2.
\end{equation}

The computation of the geometric genus $p_g(X)$ is more subtle. First of all, there is a natural inclusion obtained by pulling back the sections of $\sO_X(K_X)$ restricted to the Gorenstein locus of $(X, B)$ and then extending to global sections of $\sO_X(K_{\bar X}+\bar D)$:
\begin{equation*}
\nu^*: H^0(X, K_X)\hookrightarrow H^0(\bar X, K_{\bar X}+\bar D).
\end{equation*}
In particular, we have 
\begin{equation}\label{eq: pg}
\sum_{1\leq i\leq n} p_g(\bar X_i, \bar D_i) =\dim_\bbk H^0(\bar X, K_{\bar X}+\bar D) \geq \dim_\bbk H^0(X, K_X) = p_g(X).
\end{equation}
The image of the map $\nu^*$ consists of sections whose residue at $\bar D^\nu$ is $\tau$-anti-invariant (\cite[Proposition~5.8]{Kol13}). Let  $\res\colon  H^0(\bar X, K_{\bar X}+\bar D) \rightarrow H^0(\bar D^\nu, \Diff_{\bar D^\nu}(0))$ be the residue map. Then there is a short exact sequence
\[
0\rightarrow H^0(X, K_X)\rightarrow H^0(\bar X, K_{\bar X})\rightarrow \im(\res)^-\rightarrow 0,
\]
where $\im(\res)^-$ denotes $\tau$-anti-invariant part of $\im(\res)$. It follows that
\begin{equation*}
p_g(X)= p_g(\bar X) -\dim_\C \im(\res)^-
\end{equation*}
In particular, if $\im(\res)=0$, that is, if all of the global sections of $K_{\bar X}+\bar D$ vanish along $\bar D$, then we have
\begin{equation}\label{eq: nn pg}
p_g(X)=p_g(\bar X) =\sum_i p_g(\bar X_i).
\end{equation}

Now we can turn to
\begin{proof}[Proof of Theorem~\ref{thm: noether}]
Let $\bar X\rightarrow X$ be the normalization and $\bar D\subset \bar X$ the conductor divisor. Write $\bar X=\cup_{1\leq i\leq n} \bar X_i$ as the union of its irreducible components. Applying Theorem~\ref{thm: main} with $\sC = \{1\}$, we obtain for each $1\leq i\leq n$ 
\[(K_{\bar X_i} +\bar D_i)^2\geq \frac{1}{143} p_g(\bar X_i, \bar D_i)\] 
and it follows from \eqref{eq: nn vol} and  \eqref{eq: pg} that 
\[K_X^2=\sum_{1\leq i\leq n} (K_{\bar X_i} + \bar D_i)^2\geq \sum_{1\leq i\leq n} \frac{1}{143} p_g(\bar X_i, \bar D_i) \geq \frac{1}{143} p_g(X).\]
\end{proof}
\begin{rmk}\label{rmk: not opt}
\begin{enumerate}[leftmargin=*]
\item If the equality in Theorem~\ref{thm: noether} holds, then we see from the proof that $(K_{\bar X_i} +\bar D_i)^2=\frac{1}{143} p_g(\bar X_i, \bar D_i)$ for each $1\leq i\leq n$, so the conductor divisors $\bar D_i$ are actually empty by Proposition~\ref{prop: 0=} (v). In other words, $X=\bar X = \bar X_1$, which is normal, and $p_g(X)=1$. 
\item Using the same proof, one can show that $(K_X+B)^2\geq \frac{1}{143}p_g(X, B)$ for any projective semi-log-canonical surface $(X, B)$ with reduced $B$ and ample $K_X+B$. 

\end{enumerate}
\end{rmk}

Finally we construct a series of stable surfaces with $K_X^2=\frac{25}{84}p_g(X)$ where $p_g(X)$ can take any positive integer.
\begin{ex}\label{ex: 25/84}
We start with a configuration of a nodal cubic curve $C$ and three lines $L_1, L_2, L_3$ on $\PP^2$ satisfying the following conditions (see Figure~\ref{fig: cubic+3line}):
\begin{itemize}[leftmargin=*]
\item each pair of the curves in $\{L_1, L_2, L_3, C\}$ intersect transversally, and
\item there is exactly one triple point on $L_1+L_2+L_3+C$, which is $\{g_0\} = L_1\cap\, L_2\cap\, C$.
\end{itemize} 
\begin{figure}[h]
\begin{tikzpicture}[inner sep = 0, font=\tiny, scale=1.5]
\node (O) at (0,0) {};
\node (X) at (1, 0) {};
\node[label={[label distance = 1mm]above: $g_0$}] (g0) at ($(O) ! 1 ! 90: (X)$){$\bullet$};
\node[label={-15:$e_3$}] (e3) at ($(O) ! 1 ! 210: (X)$){$\bullet$};
\node[label={-100:$f_3$}] (f3) at ($(O) ! 1 ! -30: (X)$){$\bullet$};
\draw (g0) -- ($ (g0) ! 1.5 ! (e3) $) coordinate[label = below: $L_1$](L1);
\draw (e3) -- ($ (e3) ! 1.5 ! (g0) $) coordinate (L1');
\draw (g0) -- ($ (g0) ! 1.5 ! (f3) $) coordinate[label = below: $L_2$](L2);
\draw (f3) -- ($ (f3) ! 1.5 ! (g0) $) coordinate (L2');
\draw (e3) -- ($ (e3) ! 1.5 ! (f3) $) coordinate[label = right: $L_3$] (L3);
\draw (f3) -- ($ (f3) ! 1.5 ! (e3) $) coordinate  (L3');
\node[label={45: $C$}] (C) at (0.5, 0.8) {};
\node[label={-100: $f_{1}$}] (f1) at (30: 0.5){$\bullet$};
\node[label={-30: $e_1$}] (e1) at (150: 0.5){$\bullet$};
\node[label={150: $g_1$}] (g1) at (-1.5, -0.5){$\bullet$};
\node[label={-30: $e_2$}] (e2) at ($ (g0) ! 1.3 ! (e3) $) {$\bullet$};
\node[label={-100: $f_2$}] (f2) at ($ (f3) ! 0.2! 0: (g0) $) {$\bullet$};
\node[label={90: $g_2$}] (g2) at (0, -0.5) {$\bullet$};
\node[label={90: $g_3$}] (g3) at ($(e3)!1.3!(f3)$) {$\bullet$};
\draw plot [smooth]coordinates{ (-.5, 1) (0,1) (C) (f1)(e1)(-1, 0)  (g1)(-1.5, -.9)(e2) (g2)  (f2) (g3) (1.8, -.7)};
\end{tikzpicture}
\caption{Three lines $L_1, L_2, L_3$ and one cubic curve $C$ on $\PP^2$}
\label{fig: cubic+3line}
\end{figure}
Let $\rho\colon \tilde X\rightarrow \PP^2$ be a composition of blow-ups at the singularities of $L_1+L_2+L_3+C$ as well as their infinitely near points over $L_3$ and $C$, such that $\rho^{-1}(L_1+L_2+L_3+C)$ is a simple normal crossing curve with dual graph as in Figure~\ref{fig: bl cubic+3line}:
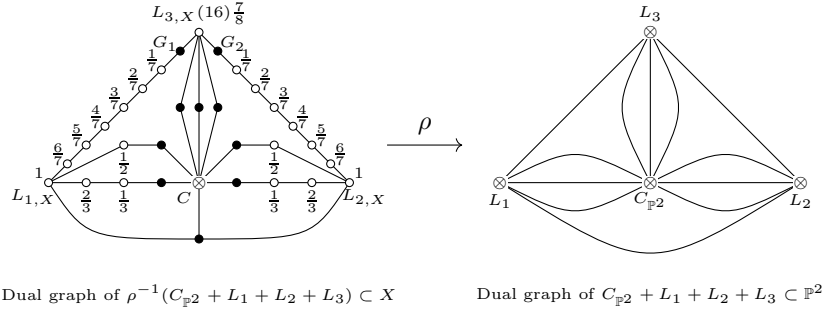
\begin{figure}[h]
\begin{tikzpicture}[inner sep = 0, font=\tiny]
\node[inner sep=0pt, label = left:$\tilde L_1(-3)$](L1) at (-4, 0){$\otimes$};
\node[inner sep=0pt, label=right:$\tilde L_2(-3)$](L2) at (4, 0){$\otimes$};
\node[wbullet, label=above: $\tilde L_3(-16)$] (L3) at (0,4){};
\node [inner sep=0pt, label={below right:$\tilde C(-5)$}](C) at (0, 0){$\otimes$};
\node[bbullet, label=above: $F_2$] (F2) at (1,0){};
\node[bbullet, label=above: $E_2$] (E2) at (-1,0){};
\node[bbullet, label=left: $G_1$] (G1) at (-.8, 2){};
\node[bbullet, label=right: $G_2$] (G2) at (0, 2){};
\node[bbullet, label=right: $G_3$] (G3) at (.8, 2){};
\node[bbullet, label=above: $E_1$] (E1) at (-1,1){};
\node[bbullet, label=above: $F_1$] (F1) at (1,1){};
\node[bbullet, label=below: $G_0$] (G0) at (0,-1){};
\node[bbullet, label=above left: $E_3$] (E3) at (-0.5, 3.5){};
\node[bbullet, label=above right: $F_3$] (F3) at (0.5, 3.5){};
\node[wbullet] (L1C4) at (-2,1){};
\foreach \x in {1,2}
\node[wbullet] (L1C\x) at (\x-4,0){};
\foreach \x in {1,2}
\node[wbullet] (L2C\x) at (4-\x,0){};
\node[wbullet] (L2C4) at (2,1){};
\node[wbullet] (L131) at (-3.5, 4-3.5){};
\node[wbullet] (L132) at (-3, 4-3){};
\node[wbullet] (L133) at (-2.5, 4-2.5){};
\node[wbullet] (L134) at (-2, 4-2){};
\node[wbullet] (L135) at (-1.5, 4-1.5){};
\node[wbullet] (L136) at (-1, 4-1){};
\node[wbullet] (L231) at (3.5, 4-3.5){};
\node[wbullet] (L232) at (3, 4-3){};
\node[wbullet] (L233) at (2.5, 4-2.5){};
\node[wbullet] (L234) at (2, 4-2){};
\node[wbullet] (L235) at (1.5, 4-1.5){};
\node[wbullet] (L236) at (1, 4-1){};
\draw (L2)--(L2C1)--(L2C2)--(F2)--(C)--(E2)--(L1C2)--(L1C1) --(L1)--(L131)--(L132)--(L133)--(L134)--(L135)--(L136)--(E3)--(L3)--(F3)--(L236)--(L235)--(L234)--(L233)--(L232)--(L231)--(L2);
\draw (L2)--(L2C4)--(F1)--(C);
\draw (L1)--(L1C4)--(E1)--(C);
\foreach \x in {1,2,3}
\draw (L3) -- (G\x) -- (C);
\draw (G0)--(C);
\draw (L1) ..controls (-3,-1)..(0,-1)..controls (3,-1)..(L2);
\end{tikzpicture}
\caption{Dual graph of $\rho^{-1}(C+L_1+L_2+L_3)\subset\tilde X$}
\label{fig: bl cubic+3line}
\end{figure}
\begin{itemize}[leftmargin=*]
\item the curves $\tilde L_i\, (1\leq i\leq 3)$ and $ \tilde C$ are the strict transforms of $L_i$ and $C$ respectively,
\item the white bullets without labels denote $(-2)$-curves,
\item the black bullets  $E_i, F_j\,(1\leq i,j\leq 3)$ and $G_k\,(0\leq k\leq 3)$ denote the $\rho$-exceptional $(-1)$-curves over $e_i, f_j, g_k$ respectively.
\end{itemize}  
Note that the $\tilde L_i$ are smooth rational curves with $\tilde L_1^2 = \tilde L_2^2=-3,  \tilde L_3^2=-16$, and $\tilde C$ is a smooth elliptic curve with $\tilde C^2=-5$.

Let $\tilde B\subset \tilde X$ be the reduced subcurve of $\rho^{-1}(C+L_1+L_2+L_3)$ consisting of the components that are not $(-1)$-curves (corresponding to the black bullets in the dual graph). Then the semistable part is $\tilde B^\s=\tilde C$ and, as before, we set $\tilde B^\ns:=\tilde B-\tilde B^\s$. Since $K_{\PP^2}+ C\sim 0$, $K_{\tilde X}+{\tilde B}\sim K_{\tilde X}+{\tilde B}-\rho^*(K_{\PP^2}+ C)$ is linearly equivalent to an effective divisor with the same support as $\tilde B^\ns+E_3+F_3$, which is big.  

Let $\pi\colon(\tilde X, \tilde B) \rightarrow (\bar X, \bar B)$ be the ample model. Then $\pi$ contracts $\tilde B-\tilde L_1-\tilde L_2$. We denote $\bar L_i = \pi_*\tilde L_i$ for $i=1,2$, $\bar E_i = \pi_* E_i$ and $\bar F_j=\pi_* F_j$ for $1\leq i,j\leq 3$, and $\bar G_k = \pi_* G_k$ for $0\leq k\leq 3$. Then we have $\bar B=\bar L_1+\bar L_2$, and $\bar L_1\cong \bar L_2\cong \PP^1$; see Figure~\ref{fig: visible on X} for a sketch of the visible curves on $\bar X$, where the bullets denotes the singularities on $\bar X$; $\bar e_1, \bar f_1$ are quotient singularities of type $\frac{1}{2}(1,1)$, $\bar e_2, \bar f_2$ of type $\frac{1}{3}(1,2)$, $\bar e_3, \bar f_3$ of type $\frac{1}{7}(1,6)$, $\bar l_3:=\pi(\tilde L_3)$ of type $\frac{1}{16}(1,1)$, and finally $\bar c=\pi(\tilde C)$ is a simple elliptic singularity of $\bar X$. Note that $\{\bar g_{0i}\}:=\bar G_0\cap \bar L_i$, $i=1,2$, are smooth points of $\bar X$.
\begin{figure}[h]
\begin{tikzpicture}[font=\tiny, scale = 1.5, inner sep =0pt]
\node[label={[label distance = 2pt]below:$\bar c$}] (c) at (0,0){$\bullet$};
\node[fill=white, label={[label distance = 3pt]above:$\bar l_3$}] (l3) at (0,1){$\bullet$};
\node[label=left:$\bar G_1$] (G1) at (-.3, .5){};
\node[label=right:$\bar G_3$]  (G3) at (.3, .5){};
\node[label=15: $\bar e_3$] (e3) at (-1.5, 1.5){$\bullet$};
\node[label=right:$\bar e_1$] (e1) at (-1.5, 0.8){$\bullet$};
\node[label=10:$\bar e_2$] (e2) at (-1.5, 0.4){$\bullet$};
\node[label=-15:$\bar g_{01}$]  (g01) at (-1.5, -.5){};
\node[label=170:$\bar f_3$] (f3) at (1.5, 1.5){$\bullet$};
\node[label=170: $\bar f_1$] (f1) at (1.5, 0.8){$\bullet$};
\node[label=165: $\bar f_2$] (f2) at (1.5, 0.4){$\bullet$};
\node[label=190:$\bar g_{02}$] (g02) at (1.5, -.5){};
\draw  ($ (e3) ! -.2 ! (l3) $)node[label=left:$\bar E_3$] (E3) {}  --($ (e3) ! 1.2 ! (l3) $);
\draw ($ (f3) ! -.2 ! (l3) $)node[label= right:$\bar F_3$] (F3) {} --  ($ (f3) ! 1.2 ! (l3) $);
\draw ($ (e1) ! -.2 ! (c) $)node[label=left:$\bar E_1$] (E1) {}  --  ($ (e1) ! 1.2 ! (c) $);
\draw ($ (e2) ! -.2 ! (c) $) node[label=left:$\bar E_2$] (E2) {} --  ($ (e2) ! 1.2 ! (c) $);
\draw ($ (f1) ! -.2 ! (c) $)node[label=right:$\bar F_1$] (F1) {}   --  ($ (f1) ! 1.2 ! (c) $);
\draw ($ (f2) ! -.2 ! (c) $)node[label=right:$\bar F_2$] (F2) {}   --  ($ (f2) ! 1.2 ! (c) $);
\node[label=left:$\bar G_0$] (G01) at ($ (g01) ! -.2 ! (c) $) {};
\node[label=right:$\bar G_0$] (G02) at ($ (c) ! 1.2 ! (g02) $) {};
\draw  plot [smooth] coordinates {($ (g01) ! -.2 ! (c) $) (g01)(c) (g02)($ (c) ! 1.2 ! (g02) $)};
\draw ($(e3) !-0.2! (g01)$) --  ($(e3) !1.2!(g01)$)node[below]{$\bar L_1$}
 ($(f3) !-0.2! (g02)$) --  ($(f3) !1.2! (g02)$)node[below]{$\bar L_2$};
\draw($(c)!-0.05!(l3)$)--node[fill=white]{$\bar G_2$}($(c)!1.1!(l3)$);
\draw plot [smooth] coordinates {($(c)!-0.1!(G1)$) (G1) (l3)($(G1)!1.1!(l3)$)};
\draw plot [smooth] coordinates {($(c)!-0.1!(G3)$) (G3) (l3)($(G3)!1.1!(l3)$)};
\end{tikzpicture}
\caption{Visible curves on $\bar X$}
\label{fig: visible on X}
\end{figure}

The positive part of $K_{\tilde X}+{\tilde B}$ is 
\[
\pi^*(K_{\bar X}+\bar B) = K_{\tilde X}+\tilde C+ \sum b_j \tilde B_j + \tilde L_1+\tilde L_2 + \frac{7}{8}\tilde L_3,
\]
where the $\tilde B_j$ are the visible $(-2)$-curves and  $b_j\in (0,1)$ are appropriate coefficients. The volume of $K_X+B$ is then
\[
\vol(K_X+B) = \pi^*(K_{\bar X}+\bar B)^2 = \pi^*(K_{\bar X}+\bar B)(7E_3+7F_3 + \tilde L_1+\tilde L_2) = \frac{25}{84}.
\]
One sees easily that $p_g(\bar X, \bar B)=h^0(\tilde X, K_{\tilde X}+\tilde C)=1$. The different along $\bar B=\bar L_1+\bar L_2$ is
\[
\Diff_{\bar B}(0) = \frac{1}{2}(\bar e_1 + \bar f_1) +\frac{2}{3}(\bar e_2+ \bar f_2) + \frac{6}{7}(\bar e_3 +\bar f_3),
\]  
and it is clear that $H^0(\bar B, \Diff_{\bar B}(0))  = 0$.

We take $n$ copies $(\bar X_s, \bar B_s = \bar L_{s1} + \bar L_{s2})$ of $(\bar X,\bar B=\bar L_1+\bar L_2)$ and glue them in a cycle along the $\bar B_s$, to obtain a stable surface $X^{(n)}$ (see Figure~\ref{fig: stable}): The boundary component $\bar L_{s2}$ of $(\bar X_s,\bar B_s)$ is identified with the boundary component $\bar L_{s+1,1}$ of $(\bar X_{s+1},\bar B_{s+1})$, so that the differents $\Diff_{\bar L_{s,2}}(0)$ and $\Diff_{\bar L_{s+1,1}}(0)$ match,  where $1\leq s\leq n$ is taken modulo $n$:
\begin{figure}[h]
\begin{tikzpicture}[font=\tiny, inner sep =0pt]
\begin{scope}
\node(c) at (0,0){$\bullet$};
\node (l3) at (0,1){$\bullet$};
\node (G1) at (-.3, .5){};
\node (G3) at (.3, .5){};
\node[label=190: $r$] (e3) at (-1.5, 1.5){$\bullet$};
\node[label=190:$p$] (e1) at (-1.5, 0.8){$\bullet$};
\node[label=190:$q$] (e2) at (-1.5, 0.4){$\bullet$};
\node[label=170: $u$] (g01) at (-1.5, -.5){};
\node (f3) at (1.5, 1.5){$\bullet$};
\node (f1) at (1.5, 0.8){$\bullet$};
\node (f2) at (1.5, 0.4){$\bullet$};
\node (g02) at (1.5, -.5){};
\draw  ($ (e3) ! -.2 ! (l3) $) --($ (e3) ! 1.2 ! (l3) $);
\draw ($ (f3) ! -.2 ! (l3) $) --  ($ (f3) ! 1.2 ! (l3) $);
\draw ($ (e1) ! -.2 ! (c) $) --  ($ (e1) ! 1.2 ! (c) $);
\draw ($ (e2) ! -.2 ! (c) $) --  ($ (e2) ! 1.2 ! (c) $);
\draw ($ (f1) ! -.2 ! (c) $) --  ($ (f1) ! 1.2 ! (c) $);
\draw ($ (f2) ! -.2 ! (c) $)  --  ($ (f2) ! 1.2 ! (c) $);
\node (G01) at ($ (g01) ! -.2 ! (c) $) {};
\node (G02) at ($ (c) ! 1.2 ! (g02) $) {};
\draw  plot [smooth] coordinates {($ (g01) ! -.2 ! (c) $) (g01)(c) (g02)($ (c) ! 1.2 ! (g02) $)};
\draw ($(e3) !-0.2! (g01)$) --  ($(e3) !1.2!(g01)$)
 ($(f3) !-0.2! (g02)$) --  ($(f3) !1.2! (g02)$);
\draw($(c)!-0.05!(l3)$)--($(c)!1.1!(l3)$);
\draw plot [smooth] coordinates {($(c)!-0.1!(G1)$) (G1) (l3)($(G1)!1.1!(l3)$)};
\draw plot [smooth] coordinates {($(c)!-0.1!(G3)$) (G3) (l3)($(G3)!1.1!(l3)$)};
\end{scope}
\begin{scope}[xshift=3cm]
\node(c) at (0,0){$\bullet$};
\node (l3) at (0,1){$\bullet$};
\node (G1) at (-.3, .5){};
\node (G3) at (.3, .5){};
\node(e3) at (-1.5, 1.5){$\bullet$};
\node (e1) at (-1.5, 0.8){$\bullet$};
\node (e2) at (-1.5, 0.4){$\bullet$};
\node (g01) at (-1.5, -.5){};
\node (f3) at (1.5, 1.5){$\bullet$};
\node (f1) at (1.5, 0.8){$\bullet$};
\node (f2) at (1.5, 0.4){$\bullet$};
\node (g02) at (1.5, -.5){};
\draw  ($ (e3) ! -.2 ! (l3) $) --($ (e3) ! 1.2 ! (l3) $);
\draw ($ (f3) ! -.2 ! (l3) $) --  ($ (f3) ! 1.2 ! (l3) $);
\draw ($ (e1) ! -.2 ! (c) $) --  ($ (e1) ! 1.2 ! (c) $);
\draw ($ (e2) ! -.2 ! (c) $) --  ($ (e2) ! 1.2 ! (c) $);
\draw ($ (f1) ! -.2 ! (c) $) --  ($ (f1) ! 1.2 ! (c) $);
\draw ($ (f2) ! -.2 ! (c) $)  --  ($ (f2) ! 1.2 ! (c) $);
\node (G01) at ($ (g01) ! -.2 ! (c) $) {};
\node (G02) at ($ (c) ! 1.2 ! (g02) $) {};
\draw  plot [smooth] coordinates {($ (g01) ! -.2 ! (c) $) (g01)(c) (g02)($ (c) ! 1.2 ! (g02) $)};
\draw ($(e3) !-0.2! (g01)$) --  ($(e3) !1.2!(g01)$)
 ($(f3) !-0.2! (g02)$) --  ($(f3) !1.2! (g02)$);
\draw($(c)!-0.05!(l3)$)--($(c)!1.1!(l3)$);
\draw plot [smooth] coordinates {($(c)!-0.1!(G1)$) (G1) (l3)($(G1)!1.1!(l3)$)};
\draw plot [smooth] coordinates {($(c)!-0.1!(G3)$) (G3) (l3)($(G3)!1.1!(l3)$)};
\end{scope}
\begin{scope}[xshift=5.5cm]
\draw[dotted] (0,0.5)--(1,0.5);
\end{scope}
\begin{scope}[xshift=9cm]
\node(c) at (0,0){$\bullet$};
\node (l3) at (0,1){$\bullet$};
\node (G1) at (-.3, .5){};
\node (G3) at (.3, .5){};
\node(e3) at (-1.5, 1.5){$\bullet$};
\node (e1) at (-1.5, 0.8){$\bullet$};
\node (e2) at (-1.5, 0.4){$\bullet$};
\node (g01) at (-1.5, -.5){};
\node[label=-10:$r$] (f3) at (1.5, 1.5){$\bullet$};
\node[label=-10: $p$] (f1) at (1.5, 0.8){$\bullet$};
\node[label=-10: $q$] (f2) at (1.5, 0.4){$\bullet$};
\node [label=10: $u$] (g02) at (1.5, -.5){};
\draw  ($ (e3) ! -.2 ! (l3) $) --($ (e3) ! 1.2 ! (l3) $);
\draw ($ (f3) ! -.2 ! (l3) $) --  ($ (f3) ! 1.2 ! (l3) $);
\draw ($ (e1) ! -.2 ! (c) $) --  ($ (e1) ! 1.2 ! (c) $);
\draw ($ (e2) ! -.2 ! (c) $) --  ($ (e2) ! 1.2 ! (c) $);
\draw ($ (f1) ! -.2 ! (c) $) --  ($ (f1) ! 1.2 ! (c) $);
\draw ($ (f2) ! -.2 ! (c) $)  --  ($ (f2) ! 1.2 ! (c) $);
\node (G01) at ($ (g01) ! -.2 ! (c) $) {};
\node (G02) at ($ (c) ! 1.2 ! (g02) $) {};
\draw  plot [smooth] coordinates {($ (g01) ! -.2 ! (c) $) (g01)(c) (g02)($ (c) ! 1.2 ! (g02) $)};
\draw ($(e3) !-0.2! (g01)$) --  ($(e3) !1.2!(g01)$)
 ($(f3) !-0.2! (g02)$) --  ($(f3) !1.2! (g02)$);
\draw($(c)!-0.05!(l3)$)--($(c)!1.1!(l3)$);
\draw plot [smooth] coordinates {($(c)!-0.1!(G1)$) (G1) (l3)($(G1)!1.1!(l3)$)};
\draw plot [smooth] coordinates {($(c)!-0.1!(G3)$) (G3) (l3)($(G3)!1.1!(l3)$)};
\end{scope}
\end{tikzpicture}
\caption{The stable surface $X^{(n)}$ glued from $n$ copies of $(\bar X, \bar B)$}
\label{fig: stable}
\end{figure}
By \eqref{eq: nn vol} and \eqref{eq: nn pg} the volume and the geometric genus of $X^{(n)}$ are as follows:
\[
K_{X^{(n)}}^2=n(K_{\bar X}+\bar B)^2=\frac{25}{84}n,\,\, p_g(X^{(n)})=n.
\]
If $n\geq 5$ then these stable surfaces violate the inequality $K_{X^{(n)}}^2> p_g(X^{(n)})-3$ which was suggested as a working hypothesis in \cite{LR16}. 

\begin{rmk}
$X^{(1)}$ is obtained from $(\bar X, \bar B)$ by glueing $\bar L_1$ and $\bar L_2$, and there is an \'etale cyclic covering map $X^{(n)}\rightarrow X^{(1)}$ of degree $n$. In fact, if $\bbk=\CC$ then the fundamental group is $\pi_1(X^{(n)})\cong \ZZ$ for each $n\geq 1$.
\end{rmk}
\end{ex}

\subsection{Bounding symplectic automorphisms of surfaces of general type}
\begin{thm}\label{thm: SymAut}
Let $S$ be a smooth projective surface of general type over an algebraically closed field $\bbk$ of characteristic $0$. Let $\Aut_s(S)$ denote the group of symplectic automorphisms of $S$, i.e., those automorphisms inducing the trivial action on $H^0(S, K_S)$. If $p_g(S)\geq 34$, then $|\Aut_s(S)|\leq 12$. 
\end{thm}
\begin{proof}
Let $G=\Aut_s(S)$ and $\pi\colon S\rightarrow X=S/G$ the quotient map. Then there is an effective divisor $B$ with coefficients in $\sC_2$ such that $(X, B)$ is klt and $K_S=\pi^*(K_X+B)$. It follows that $\vol(K_S) = |G|\vol(K_X+B)$. Since $G$ induces the trivial action on $H^0(S, K_S)$, we have $p_g(X, B)=p_g(S)$. By Proposition~\ref{prop: 1} for the coefficient set $\sC_2$, we have $\vol(K_X+B)\geq \frac{3}{4}p_g-\frac{5}{4}$. Therefore,
\[
\vol(K_S) =  |G|\vol(K_X+B) \geq |G|\left(\frac{3}{4}p_g-\frac{5}{4}\right).
\]
Combining the Bogomolov--Miyaoka--Yau inequality $\vol(K_S)\leq 9\chi(\sO_S)\leq 9(p_g+1)$ (see \cite[Theorem~5.1]{Lan16}), we obtain 
\[
|G| \left(\frac{3}{4}p_g-\frac{5}{4}\right) \leq 9\left(p_g+1\right)
\]
and it follows that 
\[
|G|\leq \frac{36(p_g+1)}{3p_g-5}=12+\frac{96}{3p_g-5}.
\]
Now one sees easily that $|G|\leq 12$ if $p_g\geq 34$.
\end{proof}
\begin{rmk}
\begin{enumerate}[leftmargin=*]
\item In Du's thesis \cite{Du22}, it was shown that $|\Aut_s(S)|\leq 417$ if $\chi(\sO_S) \geq 21$. 
\item Since surfaces of general type with bounded $p_g(S)$ or $\chi(\sO_S)$ form a bounded family,  Theorem~\ref{thm: SymAut} (and also \cite{Du22}) implies that there is a universal constant $N$ such that $|\Aut_s(S)|\leq N$ holds for any smooth projective surface of general type.
\end{enumerate}
\end{rmk}

\end{document}